\documentclass[11pt,reqno, a4paper]{amsart}
\usepackage{graphicx} 
\usepackage{mathrsfs}
\usepackage{color}
\usepackage{epsfig}
\usepackage[outdir=./]{epstopdf}
\usepackage{enumerate}
\usepackage{amssymb,amsmath,amsthm,graphicx,amsfonts}
\usepackage{hyperref}
\usepackage[margin=1.0in]{geometry}
\usepackage[ruled,linesnumbered]{algorithm2e}

\usepackage{standalone}
\usepackage{tikz} 
\usepackage{pgfplots}
\usepackage{xstring}
\usepackage{subcaption}

\newcommand{\projectroot}{.} 
\newcommand{\datapath}{\projectroot/data}

\usepackage{titlesec} 
\titleformat{\section}{\vskip10pt\large\bfseries}{\thesection.}{0.5em}{\centering\vspace{5pt}}
\titleformat{\subsection}{\vskip10pt\normalsize\bfseries}{\thesubsection.}{0.5em}{}

\newtheorem{theorem}{Theorem}[section]
\newtheorem{assumption}[theorem]{Assumption}
\newtheorem{lemma}[theorem]{Lemma}

\newtheorem{remark}[theorem]{Remark}
\newtheorem{corollary}[theorem]{Corollary}
\theoremstyle{definition}

\def\R{\mathbb{R}}

\def\d{\mathrm{d}}
\def\P{\mathbb{P}}
\def\L {\mathcal{L}}
\allowdisplaybreaks

\newcommand{\taufac}{C_{\textup{scal}}}
\newcommand{\tauref}{\tau_{\textup{ref}}}
\newcommand{\hspaceref}{h_{\textup{ref}}}

\newcommand{\massmatrix}{\mathbf{M}_h}
\newcommand{\stiffmatrix}{\mathbf{A}_h}
\newcommand{\solutionvec}{\mathbf{u}_h}
\newcommand{\solutiondtvec}{\mathbf{v}_h}
\newcommand{\loadvectorn}{\mathbf{f}_h^n}

\numberwithin{equation}{section}
\numberwithin{figure}{section}

\def\vec#1{\mathbf{#1}}

\usepackage{xcolor}

\begin{document}

\title[]{Finite element exponential integration for rough solutions of nonlinear wave equations. Part II: Dynamic boundary conditions on curved domains
}

\author[]{Jiachuan Cao\,\,}
\address{Institute for Applied and 
	Numerical Mathematics, Karlsruhe Institute of Technology, Englerstr.~2, 76131 Karlsruhe, Germany}
\email{\{jiachuan.cao,benjamin.doerich\}@kit.edu}

\author[]{Benjamin D\"{o}rich\,\,}
%

\author[]{Buyang Li}
\address{Department of Applied Mathematics, The Hong Kong Polytechnic University,
	Hung Hom, Hong Kong}
\email{buyang.li@polyu.edu.hk}

\subjclass[2010]{65M12, 65M15, 65M60, 35L05}


\keywords{nonlinear wave equation, low regularity, nontrivial boundary conditions, error estimates}
\thanks{The first two  authors acknowledge funding by the Deutsche Forschungsgemeinschaft (DFG, German Research Foundation) - 
Project-ID 258734477 - SFB 1173.
The work of B. Li was partially supported by the National Natural Science Foundation of China (project no. 12525111) and Hong Kong Research Grants Council (project no. 15306123). 
}

\maketitle
\vspace{-20pt}

\begin{abstract}\noindent
{\small 
	We study nonlinear wave equations with dynamic boundary conditions on smooth bounded domains and analyze a fully discrete approximation in the low-regularity regime. The method combines isoparametric bulk--surface finite elements of degree \(k\) with an exponential integrator in time. Assuming only bounded energy of the exact solution, we prove convergence of the displacement--velocity pair in the weak norm \(L^2(\Omega;\Gamma)\times H^{-1}(\Omega;\Gamma)\). The scheme achieves first-order convergence in time and spatial convergence of order \(h^{2/3}\) for \(k=1\) and \(h^{(k+2)/(k+3)}\) for \(k\ge 2\). In particular, these rates show that higher-order finite elements retain a provable asymptotic advantage even at low regularity. A central difficulty is that the continuous and discrete bulk--surface problems are posed on different geometries and must therefore be compared directly in weak norms. To address this, we develop a weak-norm framework for non-conforming geometries based on lift and adjoint-lift operators, combined with a frequency-decomposition argument. To the best of our knowledge, this is the first fully discrete low-regularity convergence result for nonlinear wave equations with dynamic boundary conditions in a non-conforming bulk--surface finite element setting. Numerical experiments confirm the predicted rates and illustrate the improved efficiency of higher-order methods.
}
\end{abstract}


\setlength\abovedisplayskip{4.5pt}
\setlength\belowdisplayskip{4.5pt}

\section{Introduction}
We study the numerical approximation of rough solutions to nonlinear wave equations with dynamic boundary conditions. A prototypical example is the following coupled bulk--surface system with purely second-order dynamic boundary conditions; see, for example, \cite{HipK20,vitillaro2017wave}:
\begin{equation}\label{eq:NLW_db}
	\left\{
	\begin{array}{lll}
		\partial_{tt}u - \Delta u = f_{\Omega}(u) & \text{in } \Omega \times (0, T], \\[2mm]
		\partial_{tt}u - \Delta_{\Gamma} u +\partial_{\vec{n}}u= f_{\Gamma}(u) & \text{in } \Gamma \times (0, T], \\[2mm]
		u|_{t=0} = u^{0}, \quad \partial_{t}u|_{t=0} = v^{0} & \text{in } \overline{\Omega}.
	\end{array}
	\right.
\end{equation}
Here, \(\Omega \subset \mathbb{R}^d\), with \(d=2,3\), is a smooth bounded domain with boundary \(\Gamma = \partial\Omega\), \(\Delta_\Gamma\) denotes the Laplace--Beltrami operator on \(\Gamma\), and \(\partial_{\vec{n}}\) is the outer normal derivative. Problems of the form \eqref{eq:NLW_db} arise in models of wave propagation with active boundary dynamics, transmission phenomena, and bulk--surface interactions. 
In contrast to standard boundary conditions, the boundary is part of the evolution itself: it carries its own wave dynamics and is coupled to the bulk through the normal derivative. For analytical results on well-posedness and qualitative properties of such problems, we refer to \cite{vitillaro2017wave,vitillaro2013strong}.

In recent years, numerical methods for wave equations with dynamic or other nonstandard boundary conditions have been developed in several directions. For linear wave-type equations, a unified error analysis for spatial discretizations was established in \cite{HipHS19,HipK20}, and this was extended to nonlinear problems in \cite{HL2020}. Fully discrete schemes based on bulk--surface finite element spaces have been analyzed together with various time discretizations, including exponential integrators \cite{DL2022} and implicit--explicit methods \cite{HL2021}. We also mention related developments for elliptic problems \cite{ER2013,ER2021}, parabolic equations with bulk--surface coupling \cite{ER2021,KL2017,altmann2024second}, and quasilinear wave equations \cite{dorich2022exponential,dorich2025strong}.
However, available convergence analyses typically assume sufficient smoothness of the exact solution. For hyperbolic problems, this assumption is often too strong, as limited regularity of the initial data generally persists in time. Consequently, classical error estimates are not applicable, making it essential to design methods whose convergence holds under minimal regularity assumptions.

Recent years have seen significant progress in the numerical treatment of low-regularity dispersive equations through the development of low-regularity integrators; see, for example, \cite{feng2025explicit,Hofmanova-Schratz-2017,ji2025filtered,li2025unfiltered,OS18,wu2022first} and, for wave-type equations, \cite{CLLY2024,li2023second,RS21,wang2022symmetric}. However, nearly all available convergence results for such methods rely on Fourier or spectral discretizations and are therefore largely restricted to simple geometries. The reason is structural: in spectral settings, the analysis benefits crucially from commutation properties between the projection operators and the underlying linear dynamics. For other spatial discretizations, this commutativity is lost, and in the low-regularity regime the resulting commutator errors cannot be absorbed by assuming additional smoothness of the exact solution.

A central question is therefore whether rigorous fully discrete error estimates can be established beyond periodic geometries and Fourier-based spatial discretizations. Our companion paper~\cite{cao2026approximating} provided a first answer to this question for nonlinear wave equations with standard boundary conditions on polygonal domains. In the present work, we show that this low-regularity theory can be pushed substantially further: it extends to smooth domains, where the curved bulk--surface geometry leads naturally to non-conforming finite element meshes, and also to dynamic boundary conditions of the form \eqref{eq:NLW_db}. This is precisely a setting in which previous rigorous convergence analyses have typically required substantially stronger regularity assumptions on the exact solution. By contrast, we prove convergence already for finite-energy solutions
\begin{align}\label{finite-energy}
U(0)=(u^0,v^0)\in H^1(\Omega;\Gamma)\times L^2(\Omega;\Gamma),
\end{align}
using a fully discrete scheme that combines bulk--surface isoparametric finite elements in space with an exponential integrator in time.

This extension is far from routine. In a geometric bulk--surface finite element method for wave equations with dynamic boundary conditions, one must deal with a coupled bulk--surface operator whose structure is substantially more intricate than that of the standard Laplacian. Moreover, the exact problem is posed on \((\Omega,\Gamma)\), whereas the discrete problem is formulated on the approximate geometry \((\Omega_h,\Gamma_h)\), so that the continuous and discrete solutions live on different domains. In the low-regularity regime considered here, the comparison must be carried out in weak Sobolev norms, where the geometric mismatch near the boundary requires particularly delicate estimates in order not to compromise the sharpness of the overall error bounds. A central analytical task is therefore to develop a lift and adjoint-lift framework that provides sufficiently sharp approximation estimates in \(L^2\)-, \(H^{-1}\)-, and \(H^1\)-based bulk--surface scales for the weak-error analysis.

A further difficulty comes from the nonlinear term. Because of the geometric complexity, the fully discrete method evaluates the nonlinearity on the approximate geometry by means of nodal interpolation. The analysis must therefore control, again in weak norms, the combined effect of low regularity, geometric transfer, and interpolation in the treatment of the nonlinear term.

In this paper we prove such a result. We consider finite-energy  solutions with initial value \eqref{finite-energy}
and analyze an implementable fully discrete method combining isoparametric bulk--surface finite elements of degree \(k\) with the exponential Euler scheme
\begin{equation}\label{eq:exp_intro}
	U_h^{n+1}
	= e^{\tau L_h}U_h^n
	+\tau\varphi_1(\tau L_h)F_h(U_h^n),
	\qquad
	\varphi_1(z)=\frac{e^z-1}{z}.
\end{equation}
Here \(L_h\) is the discrete coupled bulk--surface wave operator and \(F_h\) is the interpolated discrete nonlinearity on \((\Omega_h,\Gamma_h)\). Under the sole assumption of finite energy, i.e. \(U\in C([0,T];H^1(\Omega;\Gamma)\times L^2(\Omega;\Gamma))\),
we prove first-order convergence in time and the spatial error bounds \(h^{2/3}\) for linear elements and \(h^{(k+2)/(k+3)}\) for isoparametric elements of degree \(k\ge 2\). These rates are obtained without any additional Sobolev regularity of the exact solution. In particular, higher-order isoparametric finite elements retain a provable asymptotic advantage even in the energy regime.

The proof rests on several key ingredients, which are developed in Section~\ref{subsec:preliminary_db}:
\begin{itemize}
	\item We introduce a spectral truncation associated with the coupled bulk--surface elliptic operator. This truncation is used only at the analytical level and does not enter the numerical method itself. It yields an auxiliary regularized dynamics that allows the linear, nonlinear, temporal, and geometric errors to be balanced within a unified fully discrete argument.
	
	\item We develop a geometric transfer theory between \((\Omega_h,\Gamma_h)\) and \((\Omega,\Gamma)\) based on lift and adjoint-lift operators. This theory provides the stability and approximation estimates in weak Sobolev norms needed in the low-regularity regime.
	
	\item We establish weak-norm error estimates for the continuous and discrete coupled bulk--surface wave propagators. These estimates yield a weak \(H^{-1}\)-type gain for the linear error and are crucial for recovering spatial convergence from finite-energy data.
	
	\item We derive weak-norm estimates for the nonlinear term. In particular, we control both the consistency error arising from the evaluation of the nonlinearity on the approximate geometry via nodal interpolation and the weak-norm stability of the resulting interpolated nonlinear operator, both of which are essential for the convergence analysis of the fully discrete scheme.
\end{itemize}
%
%
%
%
%
%
%
%

The paper is organized as follows. Section~\ref{sec:main_results} introduces the analytical setting, the isoparametric bulk--surface finite element discretization, the fully discrete exponential Euler scheme, and the assumptions underlying the analysis, and it states the main convergence result. Section~\ref{subsec:preliminary_db} develops the main analytical ingredients of the low-regularity error analysis, including the frequency truncation, the weak-norm lift and adjoint-lift estimates, the weak-norm bounds for the filtered linear propagators, and the corresponding estimates for the nonlinear term. Section~\ref{sec:proof} contains the proof of the main convergence theorem: it first derives the consistency estimate for the frequency-localized auxiliary scheme, then analyzes the approximation of the nonlinear term in the fully discrete method, and finally establishes the global error recursion and completes the proof of Theorem~\ref{thm:convergence_db}. Finally, Section~\ref{sec:numerical_experiments} presents numerical experiments.

\textbf{Notations.} For the sake of notational simplicity, we write \( A \lesssim B \) (or equivalently \( B \gtrsim A \)) to denote that
\(A \leq C B\) for some constant \( C > 0 \). The constant \( C \) may vary from line to line, but it is always independent of the time step \(\tau\), the mesh size \(h\), the frequency cut-off constant \(r\) and the bounds of the numerical solution. 
Similarly, we use \( A \sim B \) to denote the equivalence \(C^{-1} B \leq A \leq C B\)
for some constant \( C > 0 \). In other words, \( A \sim B \) is equivalent to \( A \lesssim B \lesssim A \).  We also use the notation \(C(\cdot)\) to denote a positive constant whose value may change from line to line but depends only on the quantities specified in the parentheses.

\section{Analytical framework and main results}\label{sec:main_results}
In this section, we reformulate the nonlinear wave equation with dynamic boundary conditions in an abstract framework and introduce the fully discrete isoparametric bulk--surface finite element method. We then state the main convergence result.

\subsection{Abstract formulation}

Let \(V\) and \(H\) be Hilbert spaces such that \(V\) is densely and continuously embedded in \(H\). In this article, we take
\begin{align*}
	H &:= L^2(\Omega)\times L^2(\Gamma),\\
	V &:= H^1(\Omega;\Gamma)
	:= \{v\in H^1(\Omega)\,:\,\gamma(v)\in H^1(\Gamma)\}.
\end{align*}
Here, \(\gamma\) denotes the trace operator. Equivalently, \(H^1(\Omega;\Gamma)\) may be identified with the subspace of \(H^1(\Omega)\times H^1(\Gamma)\) consisting of pairs whose bulk and boundary components satisfy the trace compatibility condition. Throughout, we use the notation \(H^0(\Omega;\Gamma)=L^2(\Omega;\Gamma)=L^2(\Omega)\times L^2(\Gamma).\) For each integer \(q \geq 2\), we define
	\[
	H^q(\Omega;\Gamma):=\{v\in H^q(\Omega):\gamma(v)\in H^q(\Gamma)\}.
	\]
	See also \cite{Kashiwabara-2015} for further details on these function spaces.

The nonlinear wave equation \eqref{eq:NLW_db} can be reformulated in the following weak form:
\begin{align}\label{eq:weak_form}
	m(u_{tt},v)+a(u,v)=m(f(u),v),
	\qquad \text{for all } v\in V,\ t\in(0,T].
\end{align}
Here, the bilinear forms \(m\) and \(a\) are defined by
\begin{align*}
	m(u,v) &:= \int_\Omega uv\,\mathrm{d}x + \int_\Gamma uv\,\mathrm{d}s,\\
	a(u,v) &:= \int_\Omega \nabla u\cdot\nabla v\,\mathrm{d}x
	+ \int_\Gamma \nabla_\Gamma u\cdot\nabla_\Gamma v\,\mathrm{d}s.
\end{align*}
The nonlinear term \(f(u)=(f_\Omega(u),f_\Gamma(u))\) is understood in the abstract setting through
\begin{align*}
	m(f(u),v)
	:= \int_\Omega f_\Omega(u)\,v\,\mathrm{d}x
	+ \int_\Gamma f_\Gamma(u)\,v\,\mathrm{d}s.
\end{align*}
Thus, \(m\) defines the inner product on \(H\), while \(a\) is a symmetric, nonnegative bilinear form on \(V\).
Furthermore we know that there exist a constant $c_{G}> 0$ such that
\begin{align*}
	\tilde{a}:=a+c_G m
\end{align*}
is an inner product on $V$ with the induced norm $\|\cdot\|_{\tilde{a}}\sim \|\cdot\|_{H^1(\Omega;\Gamma)}$.

 By the standard operator representation associated with the bilinear form \(a\), there exists a self-adjoint, nonnegative operator \(A : D(A) \subset H \to H\), such that
 \begin{align}\label{eq:def_A}
 a(u,v)=m(Au,v), \qquad u\in D(A), \ v\in V.
 \end{align}
 For \(c_G>0\), we further define the positive operator
 \begin{align}\label{eq:def_tilde_A}
 \widetilde{A}=A+c_G.
 \end{align}
Here and in the following, \(c_G\) is understood as \(c_G\mathrm{Id}\). For the sake of notational simplicity, we omit the identity operator whenever no confusion can arise.

Thus, we can reformulate the problem \eqref{eq:NLW_db} in the following abstract form:
\begin{align}\label{eq:abstract_form_db}
	u_{tt}+Au=f(u).
\end{align}

\begin{remark}\upshape
	Introducing the strictly positive operator \(\widetilde{A}\)  serves two purposes. First, the spectral structure of \(\widetilde{A}\) provides an eigenspace decomposition of \(H\), which in turn yields the frequency decomposition operators used in the low-regularity convergence analysis. Second, the strict positivity of \(\widetilde{A}\) ensures the existence of \(\widetilde{A}^{-1}\) and its discrete analogue, which is crucial for the Aubin--Nitsche argument and the derivation of sharp weak-norm error estimates.
\end{remark}




We next rewrite \eqref{eq:abstract_form_db} as a first-order evolution system:
\begin{equation}\label{system_db}
	\left\{
	\begin{array}{lll}
		\partial_t U - L U = F(U) & \text{for } t \in (0, T], \\[2mm]
		U|_{t=0} = U^0,
	\end{array}
	\right.
\end{equation}
where
\begin{align}\label{system_notation_db}
	U = \begin{pmatrix}
		u \\ v
	\end{pmatrix}, \quad
	U^0 = \begin{pmatrix}
		u^0 \\ v^0
	\end{pmatrix}, \quad
	F(U) = \begin{pmatrix}
		0 \\ f(u)
	\end{pmatrix}, \quad
	L = \begin{pmatrix}
		0 & 1 \\ -A & 0
	\end{pmatrix}.
\end{align}
This system has the same formal structure as the classical first-order formulation of nonlinear wave equations, with the bulk and boundary dynamics coupled through \(A\) and \(f\). Throughout the paper, \(U\in C([0,T];H^1(\Omega;\Gamma)\times L^2(\Omega;\Gamma))\) denotes a mild solution of \eqref{system_db}, that is,
\begin{align}\label{eq:variation-of-constant-db}
	U(t+s)
	=
	e^{sL}U(t)
	+
	\int_0^s e^{(s-\sigma)L}
	F\bigl(U(t+\sigma)\bigr)\,\d\sigma,
\end{align}
for \(t,s\ge0\) with \(t+s\le T\). For related well-posedness results for wave equations with dynamic boundary conditions, we refer to \cite{vitillaro2017wave,vitillaro2013strong}.

\subsection{Bulk--surface finite element discretization and main convergence result}\label{subsec:fully_discrete}
Next, we consider solving equation \eqref{system_db} using the fully discrete scheme of the bulk-surface finite element method. Let $\mathcal{T}_h$ be the quasi-uniform mesh of isoparametric elements $K$ of degree $k$ with mesh size $h$. The discrete domain and its boundary are denoted by 
\begin{align*}
	\Omega_h:=\bigcup_{K\in\mathcal{T}_h}K\approx \Omega,\quad \text{and}\quad \Gamma_h=\partial \Omega_h.
\end{align*}
Let $\gamma_h$ denote the trace operator on $\Gamma_h$. We define the isoparametric bulk--surface finite element space of degree \(k\) as \(X^k_h:=\{(v_h,\gamma_h(v_h)): v_h\in V_h^\Omega,\;\text{and}\; \gamma_h(v_h)\in V_h^\Gamma\}\), where:
	\begin{align*}
		&V^{\Omega}_h:=\left\{v_h\in C(\overline{\Omega_h})\; :\; v_h\big|_K=\hat{v}_h\circ (F_K)^{-1}\;\text{with}\; \hat{v}_h\in \mathbb{P}_k(\hat{K})\;\text{for all}\; K\in \mathcal{T}_h\right\},\\
		&V^{\Gamma}_h:=\left\{w_h\in C(\Gamma_h)\;:\; w_h=v_h\big|_{\Gamma_h}\;\text{with}\; v_h\in V^{\Omega}_h\right\}.
	\end{align*}
	Here $\mathbb{P}_k(\hat{K})$ denotes the space of polynomial of total degree at most $k$ on the reference simplex $\hat{K}$ and $F_K$ is a transformation from $\hat{K}$ to $K$.

We note that, in this work, the computational domain of the current numerical solution, $\Omega_h$, does not match the domain of the original problem, $\Omega$. To compare discrete functions on \(\Omega_h\) and \(\Gamma_h\) with functions on the exact geometry, we use the lift operator $\L_h: L^{2}(\Omega_h)\times L^2(\Gamma_h) \rightarrow L^2(\Omega)\times L^2 (\Gamma)$ to align the domains. More precisely, we use the geometric map \(G\) constructed in \cite[Section~4]{ER2013}, which satisfies
\begin{align}\label{eq:def_G}
	G:\Omega_h\rightarrow \Omega \quad \text{such that}\; G(p)-p\; \text{orthogonal to the tangent space} \; T_{G(p)}\Gamma,
\end{align}
for all $p\in \Gamma_h$. For sufficiently small \(h\), the map \(G\) is a bijection, is piecewise \(C^{k+1}\), and is uniformly bi-Lipschitz. Moreover, it coincides with the identity outside the boundary layer and fixes all finite element interpolation nodes.

The lift operator $\L_h$ is then defined as $\L_h u_h(x) = u_h(G^{-1}(x))$. Furthermore, based on the discussion in \cite{ER2013} regarding the mapping $G$, we have the following norm equivalence results:
\begin{equation}\label{eq:boundedness_lift_operator}
	\begin{array}{l}
		\|v_h\|_{L^2(\Omega_h)}\sim \|\L_h v_h\|_{L^2(\Omega)},\quad\quad\quad\quad\, \|v\|_{L^2(\Omega)}\sim \|\L_h^{-1}v\|_{L^2(\Omega_h)},\\[2mm]
		\|\nabla v_h\|_{L^2(\Omega_h)}\sim \|\nabla \L_h v_h\|_{L^2(\Omega)},\quad\quad\;\; \|\nabla v\|_{L^2(\Omega)}\sim \|\nabla\L_h^{-1}v\|_{L^2(\Omega_h)},\\[2mm]
		\|w_h\|_{L^2(\Gamma_h)}\sim \|\L_h w_h\|_{L^2(\Gamma)},\quad\quad\quad\quad \|w\|_{L^2(\Gamma)}\sim \|\L_h^{-1}w\|_{L^2(\Gamma_h)},\\[2mm]
		\|\nabla_{\Gamma_h} w_h\|_{L^2(\Gamma_h)}\sim \|\nabla_{\Gamma} \L_h w_h\|_{L^2(\Gamma)},\quad \|\nabla_{\Gamma} w\|_{L^2(\Gamma)}\sim \|\nabla_{\Gamma_h}\L_h^{-1}w\|_{L^2(\Gamma_h)}.
	\end{array}
\end{equation}
We note that the discrete versions of the bilinear forms $m$ and $a$ are given by:
\begin{equation}\label{eq:discrete_bilinear}
	\begin{aligned}
		&m_h(u_h, v_h) = \int_{\Omega_h} u_h \,  v_h \, \d x+\int_{\Gamma_h} u_h \, v_h \, \d s,\\
		&a_h(u_h, v_h) = \int_{\Omega_h} \nabla u_h \cdot \nabla v_h \, \d x+\int_{ \Gamma_h} \nabla_{\Gamma_h} u_h \cdot \nabla_{\Gamma_h} v_h \, \d s .\\
	\end{aligned}
\end{equation}
Similarly, we define the discrete analogue of \(\tilde a\) by
\begin{align}\label{eq:def_tilde_ah}
	\tilde{a}_h=a_h+c_G m_h. 
\end{align}

We set
\(H^{-1}(\Omega_h;\Gamma_h)
=(H^1(\Omega_h;\Gamma_h))^*\), using \(m_h\) as the pivot
pairing. Moreover, the uniform shape regularity of the isoparametric meshes and the Lipschitz properties of \(G\) imply that the family \(\{\Omega_h\}_{h\le h_0}\) is uniformly Lipschitz. Consequently, all Sobolev embedding and trace constants used below can be chosen independently of \(h\).


With the help of the lift operator, we define the adjoint lift operators of $\L_h$ in the functional spaces $H$ and $V$ for the problem \eqref{system_db} with a dynamic boundary condition. Here, $H = L^2(\Omega) \times L^2(\Gamma)$ and $V = H^1(\Omega; \Gamma)$. Specifically, the adjoint lift operator $\L^{H*}_{h}: H \to X^k_h$ is defined as
\begin{align}\label{eq:def:Lh*H}
	m_h(\L^{H*}_h u, v_h) = m(u, \L_h v_h), \quad \text{for all} \quad v_h \in X^k_h,
\end{align}
and the adjoint lift operator $\L^{V*}_{h}: V \to X^k_h$ is defined as
\begin{align}\label{eq:def:Lh*V}
	\tilde{a}_h(\L^{V*}_h u, v_h) = \tilde{a}(u, \L_h v_h), \quad \text{for all} \quad v_h \in X_h^k.
\end{align}
By duality, we use the same notation \(\L_h^{H*}\) for its extension
from \(V^*=H^{-1}(\Omega;\Gamma)\) to \(X_h^k\), where the right-hand
side of \eqref{eq:def:Lh*H} is understood as the duality pairing
\(\langle u,\L_hv_h\rangle_{V^*,V}\) whenever \(u\in V^*\).
This extension agrees with \eqref{eq:def:Lh*H} for \(u\in H\).

Based on the variation-of-constants formula \eqref{eq:variation-of-constant-db},
we now construct the fully discrete scheme.
Spatial discretization is performed using the bulk–surface finite element method,
while the exponential Euler method is used for the temporal discretization.
This leads to the following fully discrete scheme:
\begin{align}\label{eq:fully_discrete_db}
	U_h^{n+1}
	= e^{\tau L_h} U_h^n + \tau \varphi_1(\tau L_h) F_h(U_h^n),
	\qquad \text{with} \qquad
	\varphi_1(z)=\frac{e^z-1}{z}.
\end{align}
Here,
	\begin{align}\label{eq:fully_discrete_notation_db}
		U_h^n =
		\begin{pmatrix}
			u_h^n \\[2mm] v_h^n
		\end{pmatrix},
		\qquad
		U_h^0 =
		\begin{pmatrix}
			u_h^0 \\[2mm] v_h^0
		\end{pmatrix},
		\qquad
		F_h(U_h^n) =
		\begin{pmatrix}
			0 \\[2mm] f_h(u_h^n)
		\end{pmatrix},
		\qquad
		L_h =
		\begin{pmatrix}
			0 & 1 \\[2mm] -A_h & 0
		\end{pmatrix}.
	\end{align}
	We set \(u_h^0=\mathcal{L}_h^{H*}u^0\) and \(v_h^0=\mathcal{L}_h^{H*}v^0\). The nonlinear term \(f_h\) is defined by
	\begin{align}\label{eq:def_fh}
		f_h(u_h) = I_h f(\mathcal{L}_h u_h),
	\end{align}
	where \(I_h\) denotes the bulk--surface interpolation operator mapping functions defined on \( (\Omega,\Gamma)\) into the isoparametric finite element space \(X_h^k\). More precisely, for \(g=(g_\Omega,g_\Gamma)\), no compatibility condition between the trace of \(g_\Omega\) on \(\Gamma\) and \(g_\Gamma\) is required. 
	The function \(I_h g\in X_h^k\) is then defined by
	\begin{align}\label{eq:def_Ih}
		I_h g=Q_h(I_{h,\Omega}g_{\Omega}, I_{h,\Gamma}g_{\Gamma})
	\end{align}
	where \(I_{h,\Omega}\) and \(I_{h,\Gamma}\) denote the nodal interpolation operators mapping functions defined on \(\Omega\) and \(\Gamma\) into the finite element spaces \(V_h^\Omega\) and \(V_h^\Gamma\), respectively. Here, \(Q_h : L^2(\Omega_h;\Gamma_h)\to X_h^k\)
	be the \(L^2\)-projection defined by
	\begin{align}
		\label{eq:def_Qh}
		m_h(Q_h \eta,w_h)=m_h(\eta,w_h)
		\qquad \text{for all } w_h\in X_h^k.
	\end{align}

Moreover, \(A_h\) in \eqref{eq:fully_discrete_notation_db} is the discrete analogue of \(A\), defined by
\[
a_h(u_h,v_h)=m_h(A_h u_h,v_h).
\]
We further define \(\widetilde{A}_h=A_h+c_G\), which is the discrete analogue of \(\widetilde{A}\).

\begin{remark}\label{rem:relation_Ih}\upshape
		The operator $I_h$ defined in \eqref{eq:def_Ih} extends the standard
		isoparametric nodal interpolation operator to general bulk--surface
		pairs. Indeed, if $g$ is trace compatible, then
		\[
		\bigl(I_{h,\Omega}g_\Omega,I_{h,\Gamma}g_\Gamma\bigr)
		=
		\bigl(I_{h,\Omega}g_\Omega,
		\gamma_hI_{h,\Omega}g_\Omega\bigr)
		\in X_h^k.
		\]
		Since $Q_h$ is the identity on $X_h^k$, it follows immediately that \(I_hg=\bigl(I_{h,\Omega}g_\Omega,
		\gamma_hI_{h,\Omega}g_\Omega\bigr).\)
\end{remark}

We collect several basic properties that will be used repeatedly.

\begin{lemma}[Basic discrete stability properties]
	\label{lem:basic_results}
	There exist constants \(c,C,C_L>0\), independent of \(h\), such that:
	\begin{itemize}
		\item[(a)] For every \(v\in H^1(\Omega_h;\Gamma_h)\),
		\begin{align}
			\|(1-Q_h)v\|_{L^2(\Omega_h;\Gamma_h)}
			&\le Ch\|v\|_{H^1(\Omega_h;\Gamma_h)}, \label{eq:Qh_approximation_db}\\
			\|Q_hv\|_{H^1(\Omega_h;\Gamma_h)}
			&\le C\|v\|_{H^1(\Omega_h;\Gamma_h)}.
			\label{eq:Qh_H1_stability_db}
		\end{align}
		
		\item[(b)] For every \(w_h\in X_h^k\),
		\begin{align}\label{eq:negative_norm_equivalence_db}
			c\|w_h\|_{H^{-1}(\Omega_h;\Gamma_h)}
			\le
			\|\widetilde A_h^{-1/2}w_h\|_{L^2(\Omega_h;\Gamma_h)}
			\le
			C\|w_h\|_{H^{-1}(\Omega_h;\Gamma_h)}.
		\end{align}
		
		\item[(c)] The operator \(L_h\) generates a \(C_0\)-group on
		\(X_h^k\times X_h^k\), and
		\begin{align}\label{assump_etLh}
			\|e^{t L_h} U_h\|_
			{H^{\theta}(\Omega_h;\Gamma_h)
				\times H^{\theta-1}(\Omega_h;\Gamma_h)}
			\le
			e^{C_L t}
			\|U_h\|_
			{H^{\theta}(\Omega_h;\Gamma_h)
				\times H^{\theta-1}(\Omega_h;\Gamma_h)}
		\end{align}
		for \(\theta=0,1\), \(t\ge0\), and
		\(U_h\in X_h^k\times X_h^k\).
		
	\item[(d)] For \(\theta=0,1\), \(\tau\ge0\), and every
		\(U_h\in X_h^k\times X_h^k\),
		\begin{align}\label{eq:phi1_stability_db}
			\|\varphi_1(\tau L_h)U_h\|_
			{H^\theta(\Omega_h;\Gamma_h)
				\times H^{\theta-1}(\Omega_h;\Gamma_h)}
			\le
			e^{C_L\tau}
			\|U_h\|_
			{H^\theta(\Omega_h;\Gamma_h)
				\times H^{\theta-1}(\Omega_h;\Gamma_h)}.
		\end{align}
	\end{itemize}
\end{lemma}
\begin{proof}
	The proof is provided in Appendix~\ref{sec:proof_of_etL-etLh}.
\end{proof}

We impose the following assumption on the nonlinear terms:

\begin{assumption}\label{assump_f_db}
	The nonlinear functions \(f_{\Omega}\) and \(f_{\Gamma}\) belong to \(C^{d}(\mathbb{R})\), where \(d=2,3\) denotes the spatial dimension, and they are assumed to satisfy the following growth conditions. There exist constants
	\begin{align}\label{eq:growth_condition_zeta}
		\zeta_{\Omega}<\left\{\begin{array}{l}
			\infty, \quad \;\; d=2,\\
			3
			,\qquad d= 3,
		\end{array}\right. \quad \text{and}\quad \zeta_{\Gamma}	< \infty, \quad \;\; d=2,3,
	\end{align}
	such that for all $u\in \R$ there hold:
	\begin{equation}\label{eq:growth_condition}
		\begin{array}{l}
			|f^{(a)}_{\Omega}(u)|\leq C \max\{\left(1+|u|\right)^{\zeta_{\Omega}-a},1\},\quad |f^{(a)}_{\Gamma}(u)|\leq C \max\{\left(1+|u|\right)^{\zeta_{\Gamma}-a},1\},
		\end{array}
	\end{equation}
	where $f^{(a)}_{\Omega}$ and $f^{(a)}_{\Gamma}$ denote the $a$-th derivatives of the nonlinear functions 
	$f_{\Omega}$ and $f_{\Gamma}$, respectively, for all $a =0, 1,\cdots, d$.
\end{assumption}

\begin{remark}\upshape
	The assumption~\ref{assump_f_db} imposes regularity and growth conditions on the nonlinearities that are compatible with low-regularity solutions and sufficient for the stability and interpolation estimates used later. Typical examples include sine-type nonlinearities and polynomial nonlinearities of admissible degree.
\end{remark}

We now state the main convergence result.
\begin{theorem}[Main result]\label{thm:convergence_db}
	Under Assumption~\ref{assump_f_db}, let \(U\) be the exact solution of \eqref{system_db}. If $U \in C([0,T]; H^{1}(\Omega; \Gamma) \times L^2(\Omega; \Gamma))$, then there exist \(h_0>0\) and \(\tau_0>0\) such that, for \(0<h\le h_0\) and \(0<\tau\le \tau_0\), the fully discrete isoparametric bulk--surface finite element scheme \eqref{eq:fully_discrete_db} of degree \(k\ge1\) satisfies
	\begin{align}\label{eq:convergence_db}
		\sup_{0 \leq n \leq T/\tau}\|U(t_n)-\L_h U^n_h\|_{L^2(\Omega;\Gamma)\times H^{-1}(\Omega;\Gamma)}\leq C\Bigl( \tau +h^{\frac{\ell_k}{\ell_k+1}}\Bigr),
	\end{align}
	with
	\begin{align}\label{eq:def_lk}
		\ell_k :=
		\begin{cases}
			k+1, & k=1,\\
			k+2, & k\ge2.
		\end{cases}
	\end{align}
	The constant $C>0$ only depends on $T$, $\sup_{t\in[0,T]}\|U(t)\|_{H^1(\Omega;\Gamma)\times L^2(\Omega;\Gamma)}$, and the constants appearing in Assumptions~\ref{assump_f_db}, but is independent of $h$ and $\tau$.
\end{theorem}

{
	\begin{remark}\upshape \label{rem:other_bc}
		Let us note that our theory remains valid also for the case of homogeneous Dirichlet or Neumann boundary conditions, by simply neglecting the boundary terms in the formulation and the analysis. 	
	\end{remark}
}

\begin{remark}\upshape
	Since we assume only bounded energy, namely \(U \in C([0,T]; H^{1}(\Omega;\Gamma)\times L^2(\Omega;\Gamma)),\) the stronger \(H^1\times L^2\)-type error estimates available for smooth solutions; see, for example, \cite{DL2022,KL2017,HL2020}; are not applicable in the present setting. For this reason, we measure the error in the weaker norm \(L^2(\Omega;\Gamma)\times H^{-1}(\Omega;\Gamma)\), in which we prove convergence of the geometrically non-conforming bulk--surface finite element scheme with lifting.
	
	Moreover, the sharp estimate for the lift adjoint \(\mathcal{L}^{V*}\) in the weak norm \(H^{-1}(\Omega;\Gamma)\) requires \(k\ge2\); see \eqref{eq:Hm1_eastimate_Lv*_kp1} in Lemma~\ref{lem:high_order_estimates}. As a consequence, in Theorem~\ref{lem:estimate_etL-etLh_db} we treat the cases \(k=1\) and \(k\ge2\) separately, which leads to the different convergence orders in \eqref{eq:convergence_db}.
	
	Although the limited regularity prevents the optimal higher-order spatial rates known for smooth solutions, our analysis still yields the rate \(h^{\frac{k+2}{k+3}}\) for \(k\ge2\) and \(h^{\frac 23}\) for \(k=1\). This shows that higher-order schemes remain advantageous even for low-regularity problems. The numerical experiments in Section~\ref{sec:numerical_experiments} support this observation.
\end{remark}

\section{Analytical ingredients of the low-regularity error analysis}\label{subsec:preliminary_db}

In this section, we develop the four main ingredients of the low-regularity error analysis used in the proof of the main theorem. We first introduce the spectral tools and the frequency-localized auxiliary scheme. We then establish weak-norm estimates for the lift and adjoint-lift operators $\L_h^{H*}$ and $\L_h^{V*}$, which allow us to compare functions on the exact and approximate geometries. Next, we derive weak-norm bounds for the difference between the continuous and discrete propagators $e^{tL}$ and $e^{tL_h}$ on low-frequency components. Finally, we prove the corresponding weak-norm estimates for the nonlinear terms under Assumption~\ref{assump_f_db}.

\subsection{Spectral truncation and the frequency-localized auxiliary scheme}
Classical finite element error estimates require more regularity than the
finite-energy assumption available here. We therefore introduce a spectral
cut-off, used only in the analysis, which regularizes the low-frequency
component while controlling the high-frequency remainder.

Since $\widetilde A=A+c_G$ is self-adjoint and strictly positive and the
embedding $V\hookrightarrow H$ is compact, $\widetilde A^{-1}$ is compact on
$H$. Hence, there exists an $H$-orthonormal basis
$\{e_\alpha\}_{\alpha\in I}$ and positive numbers
$\{\lambda_\alpha\}_{\alpha\in I}$ such that
\[
\widetilde A e_\alpha=\lambda_\alpha^2e_\alpha,
\qquad
\lambda_\alpha\to\infty.
\]


For $\ell\geq 0$, we introduce the Hilbert space
\[
\mathcal{Y}^\ell:=D(\widetilde A^{\ell/2}),
\qquad
\|w\|_{\mathcal{Y}^\ell}:=\|\widetilde A^{\ell/2}w\|_{H},
\]
and define $\mathcal{Y}^{-\ell}:=(\mathcal{Y}^\ell)^*$ as the dual space of $\mathcal{Y}^{\ell}$ with respect to the inner product in $H$. By the spectral theorem, the operator \(L\) generates a
\(C_0\)-group on
\(\mathcal Y^\ell\times\mathcal Y^{\ell-1}\) for every \(\ell\in\mathbb R\).
Moreover, by the continuous counterpart of the
energy argument in Lemma~\ref{lem:basic_results}(c), there exists a
constant \(C>0\), independent of \(\ell\), such that
\begin{align}\label{eq:stability_phi1}
	\|e^{tL}W\|_{\mathcal Y^\ell\times\mathcal Y^{\ell-1}}
	\le
	e^{C|t|}
	\|W\|_{\mathcal Y^\ell\times\mathcal Y^{\ell-1}},
	\qquad
	W\in\mathcal Y^\ell\times\mathcal Y^{\ell-1}.
\end{align}

\begin{lemma}[Relation between $\mathcal{Y}^q$ and $H^{q}(\Omega;\Gamma)$]
	\label{lem:spectral-sobolev-scale}
	Let $\Omega$ have a smooth boundary. Then, with equivalence of norms,
	\[
	\mathcal{Y}^0=H=H^0(\Omega;\Gamma),\qquad
	\mathcal{Y}^1=V=H^1(\Omega;\Gamma),\qquad
	\mathcal{Y}^{-1}=V^{*}=H^{-1}(\Omega;\Gamma).
	\]
	Moreover, for every integer $q\geq 2$,
	\[
	\mathcal{Y}^q\hookrightarrow H^q(\Omega;\Gamma),
	\qquad
	\|w\|_{H^q(\Omega;\Gamma)}
	\lesssim \|w\|_{\mathcal{Y}^q}.
	\]
\end{lemma}

\begin{proof}
	The identities for $\mathcal{Y}^0$ and $\mathcal{Y}^1$ follow from the definition of
	$\widetilde A$ and the representation theorem for closed symmetric
	coercive forms. The identity $\mathcal{Y}^{-1}=H^{-1}(\Omega;\Gamma)$ then
	follows by duality.
	
	We prove the remaining assertion inductively. Let $q\geq2$ and
	$w\in \mathcal{Y}^q$. Then
	\[
	g:=\widetilde A w\in \mathcal{Y}^{q-2}.
	\]
	Using the cases $q=0,1$ and the induction hypothesis, both $w$ and
	$g$ have the required $H^{q-2}$ bulk--surface regularity. The equation $\widetilde A w=g$ corresponds to
	\[
	-\Delta w=g_\Omega-c_Gw \quad\text{in }\Omega,
	\qquad
	\partial_n w+c_G\gamma w-\Delta_\Gamma(\gamma w)
	=g_\Gamma \quad\text{on }\Gamma.
	\]
	Hence, the elliptic regularity results for generalized Robin
	boundary conditions
	\cite[Theorems~3.3 and~3.4]{Kashiwabara-2015} yield
	\[
	\begin{aligned}
		\|w\|_{H^q(\Omega;\Gamma)}
		&\lesssim
		\|g_\Omega-c_Gw\|_{H^{q-2}(\Omega)}
		+\|g_\Gamma\|_{H^{q-2}(\Gamma)} \\
		&\lesssim
		\|g\|_{\mathcal{Y}^{q-2}}+\|w\|_{\mathcal{Y}^{q-2}}
		\lesssim \|w\|_{\mathcal{Y}^q}.
	\end{aligned}
	\]
	This completes the induction.
\end{proof}

In particular, $e_\alpha\in H^q(\Omega;\Gamma)$ for every integer
$q\geq0$, and
\begin{equation}
	\label{eq:estimate_el}
	\|e_\alpha\|_{H^q(\Omega;\Gamma)}
	\lesssim
	\|e_\alpha\|_{\mathcal Y^q}
	=
	\lambda_\alpha^q\|e_\alpha\|_H.
\end{equation}
Thus, we can define the projection operator $\Pi_{r}$ as the spectral projection of functions onto the low-frequency subspace $H_r \hookrightarrow H$, 
which is spanned by the basis functions
\[
\mathcal{B}_r = \left\{ e_\alpha \in H \;\middle|\; \alpha \in I,\; |\lambda_\alpha| \leq r \right\}.
\]
The projection operator $\Pi_r$ satisfies
\begin{align}\label{eq:projection_bd}
	\Pi_r u \in H_r, 
	\qquad 
	m( \Pi_r u, v ) = m( u, v ), 
	\quad \forall\, v \in H_r.
\end{align}
Moreover, we use the same notation $\Pi_r$ for its componentwise extension to
state vectors. Since $A=\widetilde A-c_G$, the projector $\Pi_r$
commutes with $A$, and hence with $L$ and $e^{tL}$.

The key advantage of this frequency decomposition is that it simultaneously provides smoothing and approximation. Although the exact solution is assumed to have only limited regularity, the low-frequency projection \(\Pi_r u\) is smoother because it is composed of smooth eigenfunctions, while still approximating \(u\) well in weaker norms. More precisely, we have the following Bernstein-type estimates.

\begin{lemma}\label{Bernstein}
	Let \(w\in \mathcal{Y}^{\ell}\) with \(\ell\in\mathbb{R}\). Then, for all \(m\ge 0\),
	\begin{align}
		\|\Pi_r w\|_{\mathcal{Y}^{\ell+m}}
		&\lesssim r^{m}\|w\|_{\mathcal{Y}^{\ell}},\quad\text{and}\quad
		\|(1-\Pi_r)w\|_{\mathcal{Y}^{\ell-m}}
		\lesssim r^{-m}\|w\|_{\mathcal{Y}^{\ell}}.
	\end{align}
\end{lemma}

\begin{proof}
	The result follows directly from the spectral definitions of
	$\mathcal Y^\ell$ and $\Pi_r$.
\end{proof}

\begin{remark}
	As a consequence of Lemma~\ref{lem:spectral-sobolev-scale} and Lemma~\ref{Bernstein}, for every integer \(q\geq 0\) and every \(w\in H^{\ell}(\Omega;\Gamma)\), with \(\ell=0,1\), we have
	\begin{align}\label{eq:Bernstein_low}
		\|\Pi_r w\|_{H^{\ell+q}(\Omega;\Gamma)}\lesssim \|\Pi_r w\|_{\mathcal{Y}^{\ell+q}}\lesssim r^q \| w\|_{\mathcal{Y}^{\ell}}\sim r^q \| w\|_{H^{\ell}(\Omega;\Gamma)}.
	\end{align}
	and
	\begin{align}\label{eq:Bernstein_high}
		\|(1-\Pi_r)w\|_{H^{\ell-1}(\Omega;\Gamma)}\sim\|(1-\Pi_r)w\|_{\mathcal{Y}^{\ell-1}}\lesssim r^{-1}\|w\|_{\mathcal{Y}^{\ell}}\sim r^{-1}\|w\|_{H^{\ell}(\Omega;\Gamma)}.
	\end{align}
\end{remark}


Motivated by the abstract formulation \eqref{system_db}, we now introduce the frequency-localized auxiliary scheme; see also \cite{cao2026approximating} for a related construction. This scheme is not part of the practical algorithm, but it serves as a key intermediate approximation in the convergence analysis of the fully discrete method. It is defined by
\begin{align}\label{eq:semi-discrete_db}
	U^{n+1} &= e^{\tau L} U^n + \tau \varphi_1(\tau L) \Pi_r F(\Pi_r U^n),\quad \text{with}\quad \varphi_1(z)=\frac{e^z-1}{z}\quad \text{and}\quad U^0=U(0).
\end{align}

A comparison of \eqref{eq:semi-discrete_db} with the variation-of-constants formula \eqref{eq:variation-of-constant-db} yields the decomposition

\begin{align}\label{eq:decompose_U}
	U(t_{n+1})=e^{\tau L}U(t_n)+\tau \varphi_1(\tau L)\Pi_r F(\Pi_r U(t_n))+\sum_{i=1}^{4}R_i(t_n),
\end{align}
with
\begin{align*}
	R_1(t_n)=&\int_{0}^{\tau}e^{(\tau-s)L}\left(F\big(U(t_n+s)\big)-F(e^{s L}U(t_n))\right)\d s,\\
	R_2(t_n)=&\int_{0}^{\tau}e^{(\tau-s)L}\left(F(e^{s L}U(t_n))-F(e^{s L}\Pi_r U(t_n))\right)\d s,\\
	R_3(t_n)=&\int_0^\tau e^{(\tau-s) L}\int_0^s \frac{\d}{\d \sigma}F(e^{\sigma L}\Pi_r U(t_n))\d\sigma\d s,\\
	R_4(t_n)=&\tau \varphi_1(\tau L)\left(F\big(\Pi_r U(t_n)\big)-\Pi_r F(\Pi_r U(t_n))\right),
\end{align*}
which will be estimated in Lemma~\ref{lem:consistency_error_db}.

\subsection{Geometric transfer between $(\Omega_h,\Gamma_h)$ and $(\Omega,\Gamma)$ in weak norm}

In this subsection, we study the operators \(\mathcal{L}_h^{H*}\) and \(\mathcal{L}_h^{V*}\), which play the roles of the \(L^2\)-projection and the Ritz projection, respectively, in the nonconforming setting. In particular, we derive its error estimates and stability results in the $H^{-1}$, $L^2$, and $H^1$ norms. Based on these estimates, we further deduce key bounds for the difference between the exponential operators $e^{tL}$ and $e^{tL_h}$ when applied to low-frequency components, which represents the main source of error in the finite element approximation of equation~\eqref{eq:NLW_db}.

First, from the definition of $\mathcal{L}_h^{H*}$ in \eqref{eq:def:Lh*H}, and by combining the stability of the lift operator $\mathcal{L}_h$ given by \eqref{eq:boundedness_lift_operator}, it follows directly that $\mathcal{L}_h^{H*}$ is a bounded operator from $L^2(\Omega;\Gamma)$ to $L^2(\Omega_h;\Gamma_h)$:
\begin{align}\label{eq:L2_bound_Lh*}
	\|\mathcal{L}_h^{H*}u\|_{L^2(\Omega_h;\Gamma_h)} \lesssim \|u\|_{L^2(\Omega;\Gamma)}.
\end{align}
Next, by exploiting the duality of the $L^2$ inner product and invoking the error estimates for the discrepancies of the $L^2$ inner products after lifting (see \cite{ER2013, ER2021}), we establish the following error estimates for $\mathcal{L}_h^{H*}$ in the $L^2$, $H^{-1}$, and $H^1$ norms, respectively.

\begin{lemma}\label{lem:estimate_Lh}
	Let the domain $\Omega$ be smooth, and consider the operator $\L_h \L_h^{H*}\colon L^2(\Omega; \Gamma) \to L^2(\Omega; \Gamma)$.
	Then $\L_h \L_h^{H*}$ is self-adjoint with respect to the inner product in $L^2(\Omega; \Gamma)$, and the following estimates hold for sufficiently regular $u$:
	
	\begin{enumerate}[(i)]
		\item ($L^2$-estimate)  
		For any $u \in H^\ell(\Omega; \Gamma)$ with $0 \le \ell \le k$,
		\begin{equation}\label{eq:L2_eastimate_Lh*}
			\|(1 - \L_h \L_h^{H*})u\|_{L^2(\Omega; \Gamma)} 
			\lesssim h^{\ell}\|u\|_{H^\ell(\Omega; \Gamma)}.
		\end{equation}
		
		\item ($H^{-1}$-estimate)  
		For any $u \in H^{\ell}(\Omega; \Gamma)$ with $0 \le \ell \le k$,
		\begin{equation}\label{eq:H-1_estimate}
			\|(1 - \L_h \L_h^{H*})u\|_{H^{-1}(\Omega; \Gamma)} 
			\lesssim h^{\ell+1}\|u\|_{H^{\ell}(\Omega; \Gamma)}.
		\end{equation}
		%
	\end{enumerate}
\end{lemma}

\begin{proof}
	The proof is provided in Appendix~\ref{Appendix:A}. 
\end{proof}

\begin{remark}\upshape
	Lemma~\ref{lem:estimate_Lh}, in particular estimates~\eqref{eq:L2_eastimate_Lh*} and~\eqref{eq:H-1_estimate}, shows that the operator $\mathcal{L}_h \mathcal{L}_h^{H*}$ exhibits properties very similar to those of the classical $L^2$ projection operator. 
	In particular, it yields similar error bounds in both the $L^2$ and $H^{-1}$ norms. 
\end{remark}

Utilizing the error estimates established in Lemma~\ref{lem:estimate_Lh} together with duality arguments, we derive the following stability results.

\begin{corollary}\label{cor:boundedness_properties}
	Let $\Omega$ be a smooth domain, and consider the operator $\L_h$ and its adjoint $\L_h^{H*}$. 
	Then the following boundedness properties hold:
	
	\begin{enumerate}[(i)]
		\item ($H^1$-boundedness of $\L_h^{H*}$)  
		For any $u \in H^1(\Omega; \Gamma)$,
		\begin{equation}\label{eq:bounded_H1}
			\|\L_h^{H*}u\|_{H^1(\Omega_h; \Gamma_h)} 
			\lesssim \|u\|_{H^1(\Omega; \Gamma)}.
		\end{equation}
		
		\item ($H^{-1}$-norm equivalence of $\L_h$)  
		For any $u_h \in X^k_h$,
		\begin{equation}\label{eq:equivalence_H-1}
			\|\L_h u_h\|_{H^{-1}(\Omega; \Gamma)} 
			\sim \|u_h\|_{H^{-1}(\Omega_h; \Gamma_h)}.
		\end{equation}
		
		\item ($H^{-1}$-boundedness of $\L_h^{H*}$)  
		For any $u \in H^{-1}(\Omega; \Gamma)$,
		\begin{equation}\label{eq:bounded_H-1}
			\|\L_h^{H*}u\|_{H^{-1}(\Omega_h; \Gamma_h)} 
			\lesssim \|u\|_{H^{-1}(\Omega; \Gamma)}.
		\end{equation}
	\end{enumerate}
\end{corollary}

\begin{proof}
	The proof is provided in Appendix~\ref{Appendix:B}. 
\end{proof}

Lemma~\ref{lem:estimate_Lh} shows that, for
		\(u\in H^\ell(\Omega;\Gamma)\) with \(0\le \ell\le k\), the approximation
		properties of \(\L_h\L_h^{H*}\) are consistent with the classical error
		estimates for the \(L^2\)-projection on conforming meshes. In particular,
		Lemma~\ref{lem:estimate_Lh} provides the expected \(L^2\)- and
		\(H^{-1}\)-estimates only up to the order \(\ell=k\), and therefore does
		not cover the endpoint case \(u\in H^{k+1}(\Omega;\Gamma)\).
		
		For the operator \(\L_h\L_h^{V*}\), it is well known
		(see \cite[Lemma~3.8]{ER2021}) that, for any
		\(u\in H^{\ell+1}(\Omega;\Gamma)\) with \(0\le \ell\le k\),
		\begin{align}\label{eq:H1_estimate_Lv*}
			\|(1-\L_h\L_h^{V*})u\|_{L^2(\Omega;\Gamma)}
			+h\,\|(1-\L_h\L_h^{V*})u\|_{H^1(\Omega;\Gamma)}
			\lesssim h^{\ell+1} \|u\|_{H^{\ell+1}(\Omega;\Gamma)} ,
		\end{align}
		which is consistent with the classical Ritz projection error estimate.
		Taking \(\ell=k\), this gives the endpoint \(L^2\)-estimate, together with
		the corresponding \(H^1\)-estimate, for \(\L_h\L_h^{V*}\), but it does not
		provide an endpoint \(H^{-1}\)-estimate.
		
		It remains to establish the endpoint estimates not covered by the preceding
		results: the \(L^2\)- and \(H^{-1}\)-estimates for \(\L_h\L_h^{H*}\), and
		the \(H^{-1}\)-estimate for \(\L_h\L_h^{V*}\). On conforming meshes, the
		corresponding optimal endpoint rates would be \(h^{k+1}\) in the
		\(L^2\)-norm and \(h^{k+2}\) in the \(H^{-1}\)-norm. In the present setting,
		however, the lift between \(\Omega_h\) and \(\Omega\) introduces geometric
		consistency errors, which lead to additional lower-order terms. The following
		lemma records these endpoint estimates.

\begin{lemma}\label{lem:high_order_estimates}
	Let $\Omega$ be a smooth domain, and let the lift operator $\L_h$ and its adjoints $\L_h^{H*}$ and $\L_h^{V*}$ be defined by \eqref{eq:def:Lh*H} and \eqref{eq:def:Lh*V}, respectively. Then, for any $u\in H^{k+1}(\Omega;\Gamma)$, the following estimates hold:
	\begin{enumerate}[(i)]
		\item ($L^2$-estimate for $\L_h^{H*}$) For all $k\geq 1$:
		\begin{align}\label{eq:L2_eastimate_Lh*_kp1}
			\|(1 - \L_h \L_h^{H*})u\|_{L^2(\Omega;\Gamma)} 
			\lesssim h^{k+1}\|u\|_{H^{k+1}(\Omega;\Gamma)}
			+ h^{k+\frac{1}{2}}\|u\|_{H^{1}(\Omega;\Gamma)}.
		\end{align}
		
		\item ($H^{-1}$-estimate for $\L_h^{H*}$) For all $k\geq 1$:
		\begin{align}\label{eq:Hm1_eastimate_Lh*_kp1}
			\|(1 - \L_h \L_h^{H*})u\|_{H^{-1}(\Omega;\Gamma)} 
			\lesssim h^{k+2}\|u\|_{H^{k+1}(\Omega;\Gamma)}
			+ h^{k+1}\|u\|_{H^{1}(\Omega;\Gamma)}.
		\end{align}
%
		
		\item ($H^{-1}$-estimate for $\L_h^{V*}$) For all $k\geq 2$:
		\begin{align}\label{eq:Hm1_eastimate_Lv*_kp1}
			\|(1 - \L_h \L_h^{V*})u\|_{H^{-1}(\Omega;\Gamma)} 
			\lesssim h^{k+2}\|u\|_{H^{k+1}(\Omega;\Gamma)}
			+ h^{k+\frac{1}{2}}\|u\|_{H^{1}(\Omega;\Gamma)}.
		\end{align}
	\end{enumerate}
\end{lemma}
\begin{proof}
	The proof is provided in Appendix~\ref{Appendix:estimate_kp1}. 
\end{proof}
\begin{remark}\upshape
	The second term on the right-hand side of \eqref{eq:Hm1_eastimate_Lv*_kp1} can be sharpened to \(h^{k+1}\|u\|_{H^2(\Omega;\Gamma)}\).
	This follows from a minor refinement of the proof in Appendix~\ref{Appendix:estimate_kp1}, but requires one additional order of regularity of \(u\). Since our focus is on estimates under low regularity assumptions, the form \eqref{eq:Hm1_eastimate_Lv*_kp1} is sharper in the present setting.
\end{remark}

\subsection{Weak-norm error estimates for the linear propagators}

%

 We now combine the discrete stability properties from Lemma~\ref{lem:basic_results} with the preceding lift and adjoint-lift estimates to compare the continuous and discrete linear propagators in \(L^2(\Omega;\Gamma)\times H^{-1}(\Omega;\Gamma)\). The frequency projection \(\Pi_r\) supplies the additional regularity needed for the following estimate. 

\begin{theorem}\label{lem:estimate_etL-etLh_db}
	Let \(\Omega\) be smooth, and let
	\(V(0)=(V_1(0),V_2(0))^\top \in H^1(\Omega;\Gamma)\times L^2(\Omega;\Gamma).\)
	Let \(L\) and \(L_h\) be the operators defined in \eqref{system_notation_db} and \eqref{eq:fully_discrete_notation_db}, respectively. Then, for every fixed \(T_*>0\), \(1\le r\le h^{-1}\), and
	\(t\in[0,T_*]\), the following estimate holds:
	\begin{align}\label{eq:improve_etL-etLh_db}
		&\left\|e^{t L} \Pi_r V(0) - \L_h e^{t L_h} \P^*_h \Pi_r V(0)\right\|_{L^2(\Omega;\Gamma)\times H^{-1}(\Omega;\Gamma)}\notag\\
		&\quad\lesssim \left\{\begin{array}{l}
			\left(h^{2}r^{2}+h^{\frac{3}{2}}r\right) \|V(0)\|_{H^{1}(\Omega;\Gamma)\times L^2(\Omega;\Gamma)},\qquad\qquad\quad\quad\text{for}\quad k=1,\\
			\left( h^{k+2}r^{k+2}+h^{k+\frac{1}{2}}r^k+h\right)\|V(0)\|_{H^{1}(\Omega;\Gamma)\times L^2(\Omega;\Gamma)},\quad \text{for}\quad k\geq 2,
		\end{array}\right.
	\end{align}
	where $\P^*_h=\text{diag}(\L^{H*}_h,\L^{H*}_h)$. The hidden constant may depend on \(T_*\), but is independent of \(h\) and \(r\).
\end{theorem}

\emph{Idea of the proof.}
The estimate is a weak-norm finite element error estimate for the homogeneous
wave equation applied to frequency-localized data. The cut-off \(\Pi_r\)
provides the additional smoothness needed in the projection estimates, while
the loss is measured by powers of \(r\). For \(k\ge2\), the main step is to
compare the semidiscrete solution with the Ritz-type adjoint lift
\(\mathcal L_h^{V*}v_r\). The corresponding error equation contains, as its
principal consistency term, the adjoint-lift mismatch \((\mathcal L_h^{V*}-\mathcal L_h^{H*})\widetilde A v_r\), while the remaining shifted-mass terms are lower-order and are absorbed by the Gronwall inequality. Testing with
\(\widetilde A_h^{-1}\partial_t\theta_h\) gives control of
\((\theta_h,\partial_t\theta_h)\) in \(L^2\times H^{-1}\). For \(k=1\), the sharp \(H^{-1}\)-estimate for
\(\L_h^{V*}\) is not available, and we instead integrate the error equation in
time and recover \(\partial_t\theta_h\) by duality.

\begin{proof}[Proof of \eqref{eq:improve_etL-etLh_db} for $k\geq 2$]
	Let
	\[
	V_r(t):=e^{tL}\Pi_r V(0)=(v_r(t),\partial_tv_r(t))
	\]
	be the continuous solution, and let
	\[
	V_{r,h}(t):=e^{tL_h}\P_h^*\Pi_r V(0)
	=(v_{r,h}(t),\partial_tv_{r,h}(t))
	\]
	be the corresponding semidiscrete solution.
	
	We first estimate the difference between $v_{r,h}$ and $v_r$. Define
	\begin{align}\label{eq:def_theta_h}
	\theta_h(t):=\L_h^{V*}v_r(t)-v_{r,h}(t)\in X_h^k.
	\end{align}
	Since $V_{r,h}$ solves the semidiscrete wave equation, we have
	\(
	\partial_tV_{r,h}(t)=L_hV_{r,h}(t),
	\)
	and hence, for every $w_h\in X_h^k$,
	\begin{equation}
		m_h(\partial_{tt}v_{r,h}(t),w_h)+a_h(v_{r,h}(t),w_h)=0.
	\end{equation}
	Recalling the definition of \(\widetilde{a}_h\), this can be rewritten as
	\begin{align}\label{eq:semi_wave_db}
		m_h(\partial_{tt}v_{r,h}(t),w_h)+\tilde{a}_h(v_{r,h}(t),w_h)=c_G m_h(v_{r,h}(t),w_h).
	\end{align}
	
	On the other hand, since $V_r$ solves the continuous wave equation
	\(
	\partial_tV_r(t)=LV_r(t),
	\)
	it follows from the definition of \(\tilde{a}\) that
	\[
	m(\partial_{tt}v_r(t),w)+\tilde{a}(v_r(t),w)=c_G\, m(v_r(t),w),
	\qquad\forall\, w\in H^1(\Omega;\Gamma).
	\]
	Taking $w=\L_hw_h$ and using the definitions of $\L_h^{H*}$ and $\L_h^{V*}$, we obtain
	\[
	m_h(\L_h^{H*}\partial_{tt}v_r(t),w_h)
	+
	\tilde{a}_h(\L_h^{V*}v_r(t),w_h)
	=c_G\, m_h(\L_h^{H*}v_r(t),w_h),
	\qquad\forall\, w_h\in X_h^k.
	\]
	Since $\L_h^{H*}$ and $\L_h^{V*}$ commute with $\partial_t$, and since $\partial_{tt} v_{r}(t)=-Av_r(t)$, the above identity can be rewritten as
	\begin{align*}
	&m_h\bigl(\partial_{tt}\L_h^{V*}v_r(t),w_h\bigr)
	+
	\tilde{a}_h(\L_h^{V*}v_r(t),w_h)\\
	&\quad=
	m_h\bigl((\L_h^{V*}-\L_h^{H*})\partial_{tt}v_r(t),w_h\bigr)+c_G\, m_h(\L_h^{H*}v_r(t),w_h)\\
	&\quad=-m_h\bigl((\L_h^{V*}-\L_h^{H*})(A+c_G)v_r(t),w_h\bigr)+c_G\, m_h(\L_h^{V*}v_r(t),w_h)\\
	&\quad=-m_h\bigl((\L_h^{V*}-\L_h^{H*})\widetilde{A}v_r(t),w_h\bigr)+c_G\, m_h(\L_h^{V*}v_r(t),w_h).
	\end{align*}
	Subtracting \eqref{eq:semi_wave_db} and by combining the definition of $\widetilde{A}$ in \eqref{eq:def_tilde_A} yields the error equation
	\begin{equation}\label{eq:error_eq_db}
		m_h(\partial_{tt}\theta_h(t),w_h)+\tilde{a}_h(\theta_h(t),w_h)
		=
		-m_h\bigl((\L_h^{V*}-\L_h^{H*})\widetilde{A} v_r(t),w_h\bigr)+c_G\,m_h(\theta_h(t),w_h),
	\end{equation}
	for all $w_h\in X_h^k.$
	We now choose
	\[
	w_h=\widetilde{A}_h^{-1}\partial_t \theta_h(t)\in X_h^k.
	\]
	Then, using the definition of $\widetilde{A}_h$, we get
	\[
	\tilde{a}_h(\theta_h,\widetilde{A}_h^{-1}\partial_t \theta_h)=m_h(\theta_h,\partial_t \theta_h)
	=\frac12\frac{\d}{\d t}\|\theta_h(t)\|_{L^2(\Omega_h;\Gamma_h)}^2,
	\]
	and
	\[
	m_h(\partial_{tt}\theta_h,\widetilde{A}_h^{-1}\partial_t \theta_h)
	=
	\frac12\frac{\d}{\d t}
	\|\widetilde{A}_h^{-1/2}\partial_t\theta_h(t)\|_{L^2(\Omega_h;\Gamma_h)}^2.
	\]
	Therefore, \eqref{eq:error_eq_db} gives
	\begin{align*}
	&\frac12\frac{\d}{\d t}
	\Bigl(
	\|\widetilde{A}_h^{-1/2}\partial_t\theta_h(t)\|_{L^2(\Omega_h;\Gamma_h)}^2
	+
	\|\theta_h(t)\|_{L^2(\Omega_h;\Gamma_h)}^2
	\Bigr)\\
	&\quad=
	-m_h\Bigl(
	\widetilde{A}_h^{-1/2}(\L_h^{V*}-\L_h^{H*})\widetilde{A}v_r(t),
	\,\widetilde{A}_h^{-1/2}\partial_t\theta_h(t)
	\Bigr)+c_Gm_h(\widetilde{A}_h^{-1/2}\theta_h(t),\widetilde{A}_h^{-1/2}\partial_t\theta_h(t)).
	\end{align*}
	By Cauchy--Schwarz and Young's inequality,
	\begin{align}\label{eq:energy_estimate}
		&\frac{\d}{\d t}
		\Bigl(
		\|\widetilde{A}_h^{-1/2}\partial_t\theta_h(t)\|_{L^2(\Omega_h;\Gamma_h)}^2
		+
		\|\theta_h(t)\|_{L^2(\Omega_h;\Gamma_h)}^2
		\Bigr)\notag\\
		&\quad\lesssim
		\|\widetilde{A}_h^{-1/2}(\L_h^{V*}-\L_h^{H*})\widetilde{A}v_r(t)\|_{L^2(\Omega_h;\Gamma_h)}^2
		+
		\|\widetilde{A}_h^{-1/2}\partial_t\theta_h(t)\|_{L^2(\Omega_h;\Gamma_h)}^2+\|\theta_h(t)\|_{L^2(\Omega_h;\Gamma_h)}^2.
	\end{align}
	Here we have used the norm equivalence from Lemma~\ref{lem:basic_results} (b),
	\begin{align*}
		\|\widetilde{A}_h^{-1/2}\theta_h(t)\|_{L^2(\Omega_h;\Gamma_h)}
		\sim
		\|\theta_h(t)\|_{H^{-1}(\Omega_h;\Gamma_h)}
		\lesssim
		\|\theta_h(t)\|_{L^2(\Omega_h;\Gamma_h)}.
	\end{align*}
	Hence, applying Gronwall's inequality to \eqref{eq:energy_estimate}, we obtain
	\begin{align*}
		&\|\widetilde{A}_h^{-1/2}\partial_t\theta_h(t)\|_{L^2(\Omega_h;\Gamma_h)}^2
		+
		\|\theta_h(t)\|_{L^2(\Omega_h;\Gamma_h)}^2
		\notag\\
		&\lesssim \|\widetilde{A}_h^{-1/2}\partial_t\theta_h(0)\|_{L^2(\Omega_h;\Gamma_h)}^2+\|\theta_h(0)\|_{L^2(\Omega_h;\Gamma_h)}^2+\sup_{0\le s\le t}
		\|\widetilde{A}_h^{-1/2}(\L_h^{V*}-\L_h^{H*})\widetilde{A}v_r(s)\|_{L^2(\Omega_h;\Gamma_h)}^2.
	\end{align*}
	By the equivalence of the discrete and continuous $H^{-1}$-norms in Lemma~\ref{lem:basic_results} (b), we obtain
	\begin{align}\label{eq:Gronwall_db}
		&\|\partial_t\theta_h(t)\|_{H^{-1}(\Omega_h;\Gamma_h)}^2
		+
		\|\theta_h(t)\|_{L^2(\Omega_h;\Gamma_h)}^2
		\notag\\
		&\lesssim \|\partial_t\theta_h(0)\|_{H^{-1}(\Omega_h;\Gamma_h)}^2+\|\theta_h(0)\|_{L^2(\Omega_h;\Gamma_h)}^2+\sup_{0\le s\le t}
		\|(\L_h^{V*}-\L_h^{H*})\widetilde{A}v_r(s)\|_{H^{-1}(\Omega_h;\Gamma_h)}^2.
	\end{align}
	We first estimate the last term on the right-hand side of \eqref{eq:Gronwall_db}. By the triangle inequality and the norm equivalence result \eqref{eq:equivalence_H-1}, applied to \((\L_h^{V*}-\L_h^{H*})\widetilde{A}v_r(s)\), we decompose
	\begin{align}\label{eq:Lv*-Lh*_db}
		&\|(\L_h^{V*}-\L_h^{H*})\widetilde{A}v_r(s)\|_{H^{-1}(\Omega_h;\Gamma_h)}\lesssim \|(\L_h\L_h^{V*}-\L_h\L_h^{H*})\widetilde{A}v_r(s)\|_{H^{-1}(\Omega;\Gamma)}\notag\\
		&\quad\lesssim
		\|(1-\L_h\L_h^{H*})\widetilde{A}v_r(s)\|_{H^{-1}(\Omega;\Gamma)}+
		\|(1-\L_h\L_h^{V*})\widetilde{A}v_r(s)\|_{H^{-1}(\Omega;\Gamma)}.
	\end{align}
	Applying the approximation estimates for $\L_h^{H*}$ and $\L_h^{V*}$ established in \eqref{eq:Hm1_eastimate_Lh*_kp1} and \eqref{eq:Hm1_eastimate_Lv*_kp1} yields
	\begin{align}\label{eq:estimate_Lv*-Lh*_db_1}
		\|(\L_h^{V*}-\L_h^{H*})\widetilde{A}v_r(s)\|_{H^{-1}(\Omega_h;\Gamma_h)}
		\lesssim
		h^{k+2}\|\widetilde{A}v_r(s)\|_{H^{k+1}(\Omega;\Gamma)}
		+
		h^{k+\frac12}\|\widetilde{A}v_r(s)\|_{H^1(\Omega;\Gamma)}.
	\end{align}
	Recalling the definition of \(\widetilde{A}\) in \eqref{eq:def_tilde_A}, and using the stability estimate for \(e^{sL}\) in \eqref{eq:stability_phi1}, together with the Bernstein-type inequality \eqref{eq:Bernstein_low} and the equivalence between \(\mathcal{Y}^1 \times \mathcal{Y}^0\) and \(H^1(\Omega;\Gamma)\times L^2(\Omega;\Gamma)\) established in Lemma~\ref{lem:spectral-sobolev-scale}, we obtain
	\begin{align}\label{eq:filted_regularity}
		&\|\widetilde{A}v_r(s)\|_{H^{k+1}(\Omega;\Gamma)}\lesssim\|\widetilde{A}v_r(s)\|_{\mathcal{Y}^{k+1}}\lesssim\|v_r(s)\|_{\mathcal{Y}^{k+3}}\notag\\
		&\quad\lesssim\|\Pi_r V(0)\|_{\mathcal{Y}^{k+3}\times \mathcal{Y}^{k+2}}  \lesssim r^{k+2}\| V(0)\|_{\mathcal{Y}^{1}\times \mathcal{Y}^{0}}\sim r^{k+2}\| V(0)\|_{H^{1}(\Omega;\Gamma)\times L^2(\Omega;\Gamma)},
	\end{align}
	and similarly
	\begin{align*}
		\|\widetilde{A}v_r(s)\|_{H^1(\Omega;\Gamma)}\lesssim r^2\| V(0)\|_{H^{1}(\Omega;\Gamma)\times L^2(\Omega;\Gamma)}.
	\end{align*}
	By substituting these estimates into \eqref{eq:estimate_Lv*-Lh*_db_1} we conclude
	\begin{align}\label{eq:estimate_Lv*-Lh*_db}
		\|(\L_h^{V*}-\L_h^{H*})\widetilde{A} v_r(s)\|_{H^{-1}(\Omega_h;\Gamma_h)}\lesssim \big(h^{k+2}r^{k+2}+h^{k+\frac{1}{2}}r^2\big)\| V(0)\|_{H^{1}(\Omega;\Gamma)\times L^2(\Omega;\Gamma)}.
	\end{align}
	
	Moreover noting that 
	\begin{align*}
		\partial_{t} \theta_h(0)=\L_h^{V*}\partial_t v_r(0)-\partial_t v_{r,h}(0)=(\L_h^{V*}-\L_h^{H*})\Pi_r V_2(0)\quad \text{and}\quad \theta_h(0)=(\L_h^{V*}-\L_h^{H*})\Pi_rV_1(0).
	\end{align*}
	By using the similar argument as in \eqref{eq:Lv*-Lh*_db} and \eqref{eq:estimate_Lv*-Lh*_db} and applying the estimates \eqref{eq:L2_eastimate_Lh*_kp1}--\eqref{eq:Hm1_eastimate_Lv*_kp1} together with \eqref{eq:H1_estimate_Lv*} we deduce:
	\begin{align*}
		\|\partial_t\theta_h(0)\|_{H^{-1}(\Omega_h;\Gamma_h)}&\lesssim \|(1-\L_h\L_h^{H*})\Pi_r V_2(0)\|_{H^{-1}(\Omega;\Gamma)}+
		\|(1-\L_h\L_h^{V*})\Pi_r V_2(0)\|_{H^{-1}(\Omega;\Gamma)}\\
		&\lesssim h^{k+2}\|\Pi_{r} V_2(0)\|_{H^{k+1}(\Omega;\Gamma)}+h^{k+\frac{1}{2}}\|\Pi_{r} V_2(0)\|_{H^1(\Omega;\Gamma)},
	\end{align*}
	and
	\begin{align*}
		\|\theta_h(0)\|_{L^2(\Omega_h;\Gamma_h)}&\lesssim \|(1-\L_h\L_h^{H*})\Pi_r V_1(0)\|_{L^2(\Omega;\Gamma)}+
		\|(1-\L_h\L_h^{V*})\Pi_r V_1(0)\|_{L^2(\Omega;\Gamma)}\\
		&\lesssim h^{k+1}\|\Pi_{r} V_1(0)\|_{H^{k+1}(\Omega;\Gamma)}+h^{k+\frac{1}{2}}\|\Pi_{r} V_1(0)\|_{H^{1}(\Omega;\Gamma)}.
	\end{align*}
	Noting that $V(0)=(V_1(0),V_2(0))^\top \in H^1(\Omega;\Gamma)\times L^2(\Omega;\Gamma)$, under the condition $r\leq h^{-1}$ and $k\geq 2$, by using the Bernstein-type inequality \eqref{eq:Bernstein_low} we have
	\begin{align}\label{eq:theta0_db}
		&\|\partial_t\theta_h(0)\|_{H^{-1}(\Omega_h;\Gamma_h)}+\|\theta_h(0)\|_{L^2(\Omega_h;\Gamma_h)}\notag\\
		&\quad\lesssim \big(h^{k+2}r^{k+1}+h^{k+\frac{1}{2}}r\big)\|V_2(0)\|_{L^2(\Omega;\Gamma)}+\big(h^{k+1}r^{k}+h^{k+\frac{1}{2}}\big)\|V_1(0)\|_{H^1(\Omega;\Gamma)}\notag\\
		&\quad\lesssim h\|V(0)\|_{H^1(\Omega;\Gamma)\times L^2(\Omega;\Gamma)}.
	\end{align}
	Substituting \eqref{eq:estimate_Lv*-Lh*_db} and \eqref{eq:theta0_db} into \eqref{eq:Gronwall_db}, we infer
	\begin{equation}\label{eq:theta_estimate_db}
		\|\partial_t\theta_h(t)\|_{H^{-1}(\Omega_h;\Gamma_h)}
		+
		\|\theta_h(t)\|_{L^2(\Omega_h;\Gamma_h)}
		\lesssim
		\bigl(h^{k+2}r^{k+2}+h^{k+\frac12}r^2+h\bigr)
		\|V(0)\|_{H^1(\Omega;\Gamma)\times L^2(\Omega;\Gamma)}.
	\end{equation}
	
	It remains to estimate the full error in
	$L^2(\Omega;\Gamma)\times H^{-1}(\Omega;\Gamma)$.
	By the triangle inequality,
	\begin{align}\label{eq:improve_kg2_decompose1}
		&\|e^{tL}\Pi_r V(0)-\L_he^{tL_h}\P_h^\ast\Pi_r V(0)\|_{L^2(\Omega;\Gamma)\times H^{-1}(\Omega;\Gamma)}\notag\\
		&\quad\le
		\|\partial_tv_r(t)-\L_h\partial_tv_{r,h}(t)\|_{H^{-1}(\Omega;\Gamma)}
		+
		\|v_r(t)-\L_hv_{r,h}(t)\|_{L^2(\Omega;\Gamma)}.
	\end{align}
	For the first term, we write
	\[
	\partial_tv_r-\L_h\partial_tv_{r,h}
	=
	\bigl(\partial_tv_r-\L_h\L_h^{V*}\partial_tv_r\bigr)
	+
	\L_h\bigl(\L_h^{V*}\partial_tv_r-\partial_tv_{r,h}\bigr),
	\]
	and similarly
	\[
	v_r-\L_hv_{r,h}
	=
	\bigl(v_r-\L_h\L_h^{V*}v_r\bigr)
	+
	\L_h\bigl(\L_h^{V*}v_r -v_{r,h}\bigr).
	\]
	Hence,
	\begin{align}\label{eq:improve_kg2_decompose2}
		&\|\partial_tv_r(t)-\L_h\partial_tv_{r,h}(t)\|_{H^{-1}(\Omega;\Gamma)}
		+
		\|v_r(t)-\L_hv_{r,h}(t)\|_{L^2(\Omega;\Gamma)}\notag\\
		&\quad\lesssim \|(\partial_tv_r-\L_h\L_h^{V*}\partial_tv_r)(t)\|_{H^{-1}(\Omega;\Gamma)}
		+
		\|(v_r-\L_h\L_h^{V*}v_r)(t)\|_{L^2(\Omega;\Gamma)} \notag\\
		&\qquad+\|(\L_h^{V*}\partial_tv_r-\partial_tv_{r,h})(t)\|_{H^{-1}(\Omega_h;\Gamma_h)}+\|(\L_h^{V*}v_r-v_{r,h})(t)\|_{L^2(\Omega_h;\Gamma_h)}.
	\end{align}
	Recalling the definition of \(\theta_h\) in \eqref{eq:def_theta_h}, the last two terms are controlled by \eqref{eq:theta_estimate_db}. For the first two terms on the right-hand side of \eqref{eq:improve_kg2_decompose2}, the approximation properties of \(\L_h^{V*}\) from \eqref{eq:H1_estimate_Lv*} and \eqref{eq:Hm1_eastimate_Lv*_kp1} yield
	\begin{align*}
		&\|\partial_tv_r(t)-\L_h\L_h^{V*}\partial_tv_r(t)\|_{H^{-1}(\Omega;\Gamma)}
		+
		\|v_r(t)-\L_h\L_h^{V*}v_r(t)\|_{L^2(\Omega;\Gamma)}\\
		&\quad \lesssim
		h^{k+2}\|\partial_tv_r(t)\|_{H^{k+1}(\Omega;\Gamma)}
		+
		h^{k+\frac12}\|\partial_tv_r(t)\|_{H^1(\Omega;\Gamma)}
		+
		h^{k+1}\|v_r(t)\|_{H^{k+1}(\Omega;\Gamma)}. 
	\end{align*}
	By the same argument as in \eqref{eq:filted_regularity}, using
	\eqref{eq:stability_phi1} and the Bernstein-type estimate
	\eqref{eq:Bernstein_low}, we obtain
	\begin{align*}
		\|\partial_tv_r(t)\|_{H^{k+1}(\Omega;\Gamma)}
		&\lesssim
		r^{k+1}\|V(0)\|_{H^{1}(\Omega;\Gamma)\times L^2(\Omega;\Gamma)}\\
		\|v_r(t)\|_{H^{k+1}(\Omega;\Gamma)}&\lesssim r^{k}\|V(0)\|_{H^{1}(\Omega;\Gamma)\times L^2(\Omega;\Gamma)},	
	\end{align*} 
	and similar estimates for $\|\partial_tv_r(t)\|_{H^1(\Omega;\Gamma)}$ we have
	\begin{align*}
		&\|\partial_tv_r(t)-\L_h\L_h^{V*}\partial_tv_r(t)\|_{H^{-1}(\Omega;\Gamma)}
		+
		\|v_r(t)-\L_h\L_h^{V*}v_r(t)\|_{L^2(\Omega;\Gamma)}\\
		&\quad \lesssim \big(h^{k+2}r^{k+1}+h^{k+\frac{1}{2}}r+h^{k+1}r^{k}\big)\lesssim h, \qquad \text{when}\quad r\leq h^{-1}.
	\end{align*}
	Combining this with \eqref{eq:theta_estimate_db}, \eqref{eq:improve_kg2_decompose1} and \eqref{eq:improve_kg2_decompose2}, we conclude that
	\begin{align*}
		&\bigl\|e^{tL}\Pi_r V(0)-\L_h e^{tL_h}\P_h^* \Pi_r V(0)\bigr\|_{L^2(\Omega;\Gamma)\times H^{-1}(\Omega;\Gamma)}\lesssim \bigl(h^{k+2}r^{k+2}+h^{k+\frac12}r^2+h\bigr)
		\|V(0)\|_{H^1(\Omega;\Gamma)\times L^2(\Omega;\Gamma)}.
	\end{align*}
	Since \(k\ge2\) and \(r\ge1\), we have \(r^2\le r^k\). This yields the stated estimate in \eqref{eq:improve_etL-etLh_db} and completes the proof for the case \(k\ge2\).
\end{proof}

\begin{proof}[Proof of \eqref{eq:improve_etL-etLh_db} for \(k=1\)]
	We keep the notation
	\begin{align*}
		V_r(t)&:=e^{tL}\Pi_r V(0)=\bigl(v_r(t),\partial_t v_r(t)\bigr),\\
		V_{r,h}(t)&:=e^{tL_h}\P_h^*\Pi_r V(0)
		=\bigl(v_{r,h}(t),\partial_t v_{r,h}(t)\bigr),
	\end{align*}
	where \(\P_h^*=\operatorname{diag}(\L_h^{H*},\L_h^{H*})\). We define
	\[
	\theta_h(t):=\L_h^{V*}v_r(t)-v_{r,h}(t)\in X_h^1,
	\qquad
	\rho_h(t):=(\L_h^{V*}-\L_h^{H*})v_r(t)\in X_h^1.
	\]
	As in the case \(k\ge2\), i.e. \eqref{eq:error_eq_db}, the function \(\theta_h\) satisfies
	\begin{align}\label{eq:err_eq_db_k1}
		&m_h(\partial_{tt}\theta_h(t),w_h)+\tilde{a}_h(\theta_h(t),w_h)\notag\\
		&\quad=
		-m_h\bigl((\L_h^{V*}-\L_h^{H*})(A+c_G) v_r(t),w_h\bigr)+c_G\,m_h(\theta_h(t),w_h)\notag\\
		&\quad=m_h\bigl((\L_h^{V*}-\L_h^{H*})(\partial_{tt}-c_G) v_r(t),w_h\bigr)+c_G\,m_h(\theta_h(t),w_h)\notag\\
		&\quad=m_h\bigl(\partial_{tt}\rho_h(t),w_h\bigr)+c_G\,\Big(m_h(\theta_h(t),w_h)-m_h\bigl(\rho_h(t),w_h\bigr)\Big)
	\end{align}
	for all \(w_h\in X_h^1\). Here we used the identity \(\partial_{tt}v_r(t)=-Av_r(t)\), together with the fact that \(\L_h^{H*}\) and \(\L_h^{V*}\) commute with \(\partial_t\).
	
	Define
	\[
	\Theta_h(t):=\int_0^t \theta_h(s)\,ds .
	\]
	Integrating \eqref{eq:err_eq_db_k1} from \(0\) to \(t\), we obtain
	\begin{align}
		&m_h(\partial_t\theta_h(t),w_h)-m_h\bigl(\partial_t\theta_h(0),w_h\bigr)+\tilde{a}_h(\Theta_h(t),w_h)\notag\\
		&\quad=
		m_h\bigl(\partial_t\rho_h(t)-\partial_t\rho_h(0),w_h\bigr)+c_G \,m_h(\Theta_h(t),w_h)	-c_G\int_0^tm_h\bigl(\rho_h(s),w_h\bigr)\d s
		\label{eq:int_err_eq_db_k1}
	\end{align}
	for all \(w_h\in X_h^1\). Choosing \(w_h=\theta_h(t)=\partial_t\Theta_h(t)\), we get
	\begin{align*}
		&\frac12\frac{\d}{\d t}m_h(\theta_h(t),\theta_h(t))
		+\frac12\frac{\d}{\d t}\tilde a_h(\Theta_h(t),\Theta_h(t))\\
		&\quad=
		m_h\bigl(\partial_t\rho_h(t)-\partial_t\rho_h(0),\theta_h(t)\bigr)+m_h\bigl(\partial_t\theta_h(0),\theta_h(t)\bigr)\\
		&\qquad
		+c_G \,m_h(\Theta_h(t),\theta_h(t))-c_G\int_0^tm_h\bigl(\rho_h(s),\theta_h(t)\bigr)\d s\\
		&\quad\lesssim
		\Bigl(
		\|\partial_t\rho_h(t)\|_{L^2(\Omega_h;\Gamma_h)}
		+\|\partial_t\rho_h(0)\|_{L^2(\Omega_h;\Gamma_h)}
		+\|\partial_t\theta_h(0)\|_{L^2(\Omega_h;\Gamma_h)}
		\Bigr)
		\|\theta_h(t)\|_{L^2(\Omega_h;\Gamma_h)}\\
		&\qquad+\|\theta_h(t)\|_{L^2(\Omega_h;\Gamma_h)}\|\Theta_h(t)\|_{L^2(\Omega_h;\Gamma_h)}+t\sup_{0\le s\le t}\|\rho_h(s)\|_{L^2(\Omega_h;\Gamma_h)}\|\theta_h(t)\|_{L^2(\Omega_h;\Gamma_h)}\\
		&\quad\lesssim \|\partial_t\rho_h(t)\|^2_{L^2(\Omega_h;\Gamma_h)}
		+\|\partial_t\rho_h(0)\|^2_{L^2(\Omega_h;\Gamma_h)}
		+\|\partial_t\theta_h(0)\|^2_{L^2(\Omega_h;\Gamma_h)}+\sup_{0\le s\le t}\|\rho_h(s)\|^2_{L^2(\Omega_h;\Gamma_h)}\\
		&\qquad+\|\theta_h(t)\|_{L^2(\Omega_h;\Gamma_h)}^2+\|\Theta_h(t)\|^2_{H^1(\Omega_h;\Gamma_h)}.
	\end{align*}
	Here, in the last inequality, we used \(t\le T\), the Cauchy--Schwarz inequality, Young's inequality, and the embedding
	\(H^1(\Omega_h;\Gamma_h)\hookrightarrow L^2(\Omega_h;\Gamma_h)\).

	Using the equivalence of \(m_h\) and \(\tilde a_h\) with the
	\(L^2(\Omega_h;\Gamma_h)\)- and \(H^1(\Omega_h;\Gamma_h)\)-norms, respectively,
	and integrating the resulting energy inequality, we infer
	\begin{align}
		&\|\theta_h(t)\|^2_{L^2(\Omega_h;\Gamma_h)}
		+\|\Theta_h(t)\|^2_{H^1(\Omega_h;\Gamma_h)}\lesssim \|\theta_h(0)\|^2_{L^2(\Omega_h;\Gamma_h)}
		+\|\partial_t\rho_h(0)\|^2_{L^2(\Omega_h;\Gamma_h)}\notag\\
		&\qquad+\|\partial_t\theta_h(0)\|^2_{L^2(\Omega_h;\Gamma_h)}+\sup_{0\le s\le t}\|\partial_t\rho_h(s)\|^2_{L^2(\Omega_h;\Gamma_h)}+\sup_{0\le s\le t}\|\rho_h(s)\|^2_{L^2(\Omega_h;\Gamma_h)},
		\label{eq:theta_Theta_est_db_k1}
	\end{align}
	where we have used \(\Theta_h(0)=0\).
	
	We now estimate the terms on the right-hand side of \eqref{eq:theta_Theta_est_db_k1}. First, by the triangle inequality and boundedness result \eqref{eq:boundedness_lift_operator},
	\[
	\|\partial_t\rho_h(s)\|_{L^2(\Omega_h;\Gamma_h)}
	=
	\|(\L_h^{V*}-\L_h^{H*})\partial_t v_r(s)\|_{L^2(\Omega_h;\Gamma_h)}
	\lesssim
	\|(\L_h\L_h^{V*}-\L_h\L_h^{H*})\partial_t v_r(s)\|_{L^2(\Omega;\Gamma)}.
	\]
	Hence,
	\begin{align*}
		\|\partial_t\rho_h(s)\|_{L^2(\Omega_h;\Gamma_h)}
		&\lesssim
		\|(1-\L_h\L_h^{V*})\partial_t v_r(s)\|_{L^2(\Omega;\Gamma)}
		+
		\|(1-\L_h\L_h^{H*})\partial_t v_r(s)\|_{L^2(\Omega;\Gamma)}.
	\end{align*}
By the \(L^2\)-error estimates \eqref{eq:H1_estimate_Lv*} and \eqref{eq:L2_eastimate_Lh*_kp1} with \(k=1\), we obtain
\begin{align}
	\|\partial_t\rho_h(s)\|_{L^2(\Omega_h;\Gamma_h)}
	&\lesssim
	h^2\|\partial_t v_r(s)\|_{H^2(\Omega;\Gamma)}
	+h^{\frac32}\|\partial_t v_r(s)\|_{H^1(\Omega;\Gamma)} .
	\label{eq:dt_rho_est_db_k1}
\end{align}
Similarly, applying the same estimates to \(\rho_h\) instead of \(\partial_t\rho_h\), we have
\begin{align}
	\|\rho_h(s)\|_{L^2(\Omega_h;\Gamma_h)}
	&\lesssim
	h^2\|v_r(s)\|_{H^2(\Omega;\Gamma)}
	+h^{\frac32}\|v_r(s)\|_{H^1(\Omega;\Gamma)} .
	\label{eq:rho_est_db_k1}
\end{align}
Moreover, applying these estimates to the initial data gives
\begin{align}
	\|\partial_t\theta_h(0)\|_{L^2(\Omega_h;\Gamma_h)}
	&=
	\|(\L_h^{V*}-\L_h^{H*})\Pi_r V_2(0)\|_{L^2(\Omega_h;\Gamma_h)}
	\notag\\
	&\lesssim
	h^2\|\Pi_r V_2(0)\|_{H^2(\Omega;\Gamma)}
	+h^{\frac32}\|\Pi_r V_2(0)\|_{H^1(\Omega;\Gamma)} ,
	\label{eq:dt_theta0_est_db_k1}
\end{align}
and
\begin{align}
	\|\theta_h(0)\|_{L^2(\Omega_h;\Gamma_h)}
	&=
	\|(\L_h^{V*}-\L_h^{H*})\Pi_r V_1(0)\|_{L^2(\Omega_h;\Gamma_h)}
	\notag\\
	&\lesssim
	h^2\|\Pi_r V_1(0)\|_{H^2(\Omega;\Gamma)}
	+h^{\frac32}\|\Pi_r V_1(0)\|_{H^1(\Omega;\Gamma)} .
	\label{eq:theta0_est_db_k1}
\end{align}
	
	Using the boundedness of \(e^{tL}\) in \eqref{eq:stability_phi1}, the condition \(V(0)=(V_1(0),V_2(0))^\top \in H^1(\Omega;\Gamma)\times L^2(\Omega;\Gamma)\) and Bernstein-type inequality \eqref{eq:Bernstein_low} for the spectral localization of \(\Pi_r\), we have 
	\begin{align}\label{eq:H1_reg_db_k1_v1}
		\|\Pi_r V_1(0)\|_{H^1(\Omega;\Gamma)}+\|v_r(s)\|_{H^1(\Omega;\Gamma)}\lesssim \|V(0)\|_{H^1(\Omega;\Gamma)\times L^2(\Omega;\Gamma)}
	\end{align}
	and
	\begin{align}
		&\|\partial_t v_r(s)\|_{H^2(\Omega;\Gamma)}
		+\|\Pi_r V_2(0)\|_{H^2(\Omega;\Gamma)}
		\lesssim
		r^2\|V(0)\|_{H^1(\Omega;\Gamma)\times L^2(\Omega;\Gamma)},
		\label{eq:H2_reg_db_k1}
		\\[3mm]
		&\|\partial_t v_r(s)\|_{H^1(\Omega;\Gamma)}+\|v_r(s)\|_{H^2(\Omega;\Gamma)}
		+\|\Pi_r V_2(0)\|_{H^1(\Omega;\Gamma)}
		+\|\Pi_r V_1(0)\|_{H^2(\Omega;\Gamma)}\notag\\
		&\quad\lesssim
		r\|V(0)\|_{H^1(\Omega;\Gamma)\times L^2(\Omega;\Gamma)},
		\label{eq:H1_reg_db_k1}
	\end{align}
	Substituting \eqref{eq:H1_reg_db_k1_v1}--\eqref{eq:H1_reg_db_k1} into
	\eqref{eq:dt_rho_est_db_k1}--\eqref{eq:theta0_est_db_k1}, we obtain
	\begin{align}
		&\|\theta_h(0)\|_{L^2(\Omega_h;\Gamma_h)}
		+\|\partial_t\rho_h(0)\|_{L^2(\Omega_h;\Gamma_h)}
		+\|\partial_t\theta_h(0)\|_{L^2(\Omega_h;\Gamma_h)}
		+\sup_{0\le s\le t}\|\partial_t\rho_h(s)\|_{L^2(\Omega_h;\Gamma_h)}\notag\\
		&\quad+\sup_{0\le s\le t}\|\rho_h(s)\|_{L^2(\Omega_h;\Gamma_h)}\lesssim
		\bigl(h^2r^2+h^{\frac32}r\bigr)
		\|V(0)\|_{H^1(\Omega;\Gamma)\times L^2(\Omega;\Gamma)}.
		\label{eq:rhs_theta_est_db_k1}
	\end{align}
	Therefore, by \eqref{eq:theta_Theta_est_db_k1},
	\begin{equation}
		\|\theta_h(t)\|_{L^2(\Omega_h;\Gamma_h)}
		+\|\Theta_h(t)\|_{H^1(\Omega_h;\Gamma_h)}
		\lesssim
		\bigl(h^2r^2+h^{\frac32}r\bigr)
		\|V(0)\|_{H^1(\Omega;\Gamma)\times L^2(\Omega;\Gamma)}.
		\label{eq:theta_Theta_final_db_k1}
	\end{equation}
	
Next, we estimate \(\partial_t\theta_h(t)\) in \(H^{-1}(\Omega_h;\Gamma_h)\). By duality and the norm equivalence established in \eqref{eq:Hminus1_Hminus1h_db}, we obtain
\begin{align}\label{eq:duality}
	\|\partial_t\theta_h(t)\|_{H^{-1}(\Omega_h;\Gamma_h)}
	&\lesssim
	\sup_{\substack{w_h\in X^1_h\atop \|w_h\|_{H^1(\Omega_h;\Gamma_h)}=1}}
	m_h(\partial_t\theta_h(t),w_h).
\end{align}
Using \eqref{eq:int_err_eq_db_k1}, Lemma~\ref{lem:basic_results} (b), the Cauchy--Schwarz inequality, and the embedding
\(H^1(\Omega_h;\Gamma_h)\hookrightarrow L^2(\Omega_h;\Gamma_h)\), we obtain
\begin{align*}
	\big|m_h(\partial_t\theta_h(t),w_h)\big|
	\lesssim
	\Big(&
	\|\Theta_h(t)\|_{H^1(\Omega_h;\Gamma_h)}
	+\|\partial_t\theta_h(0)\|_{L^2(\Omega_h;\Gamma_h)}
	+\|\partial_t\rho_h(t)\|_{L^2(\Omega_h;\Gamma_h)}
	\\
	&\quad
	+\|\partial_t\rho_h(0)\|_{L^2(\Omega_h;\Gamma_h)}
	+\sup_{0\le s\le t}\|\rho_h(s)\|_{L^2(\Omega_h;\Gamma_h)}
	\Big)
	\|w_h\|_{H^1(\Omega_h;\Gamma_h)} .
\end{align*}
Combining this estimate with \eqref{eq:theta_Theta_final_db_k1} and \eqref{eq:rhs_theta_est_db_k1}, we get
\begin{align*}
	\big|m_h(\partial_t\theta_h(t),w_h)\big|
	\lesssim
	\bigl(h^2r^2+h^{\frac32}r\bigr)
	\|V(0)\|_{H^1(\Omega;\Gamma)\times L^2(\Omega;\Gamma)}
	\|w_h\|_{H^1(\Omega_h;\Gamma_h)} .
\end{align*}
Therefore, by the duality estimate \eqref{eq:duality}, we conclude that
\begin{equation}
	\|\partial_t\theta_h(t)\|_{H^{-1}(\Omega_h;\Gamma_h)}
	\lesssim
	\bigl(h^2r^2+h^{\frac32}r\bigr)
	\|V(0)\|_{H^1(\Omega;\Gamma)\times L^2(\Omega;\Gamma)} .
	\label{eq:dttheta_final_db_k1}
\end{equation}
	
	As in the proof for the case \(k\ge2\), we now write
	\begin{align}
		&\|e^{tL}\Pi_rV(0)-\L_h e^{tL_h}\P_h^*\Pi_rV(0)\|_{L^2(\Omega;\Gamma)\times H^{-1}(\Omega;\Gamma)}
		\notag\\
		&\qquad\le
		\|\partial_t v_r(t)-\L_h\L_h^{V*}\partial_t v_r(t)\|_{H^{-1}(\Omega;\Gamma)}
		+\|v_r(t)-\L_h\L_h^{V*}v_r(t)\|_{L^2(\Omega;\Gamma)}
		\notag\\
		&\qquad\quad
		+\|\partial_t\theta_h(t)\|_{H^{-1}(\Omega_h;\Gamma_h)}
		+\|\theta_h(t)\|_{L^2(\Omega_h;\Gamma_h)}.
		\label{eq:split_total_error_db_k1}
	\end{align}
	The last two terms are controlled by \eqref{eq:theta_Theta_final_db_k1} and \eqref{eq:dttheta_final_db_k1}. For the first two terms, applying \eqref{eq:H1_estimate_Lv*} with \(k=1\) and the Bernstein-type inequality \eqref{eq:Bernstein_low}, we obtain
	\begin{align*}
		\|\partial_t v_r(t)-\L_h\L_h^{V*}\partial_t v_r(t)\|_{H^{-1}(\Omega;\Gamma)}&\lesssim \|\partial_t v_r(t)-\L_h\L_h^{V*}\partial_t v_r(t)\|_{L^2(\Omega;\Gamma)}\\
		&\lesssim h^2\|\partial_t v_r(t)\|_{H^2(\Omega;\Gamma)}
		\lesssim h^2r^2\|V(0)\|_{H^1(\Omega;\Gamma)\times L^2(\Omega;\Gamma)},\\[2mm]
		\|v_r(t)-\L_h\L_h^{V*}v_r(t)\|_{L^2(\Omega;\Gamma)}
		&\lesssim h^2\|v_r(t)\|_{H^2(\Omega;\Gamma)}
		\lesssim h^2r\|V(0)\|_{H^1(\Omega;\Gamma)\times L^2(\Omega;\Gamma)}.
	\end{align*}
	Substituting these bounds into \eqref{eq:split_total_error_db_k1}, we conclude that
	\[
	\|e^{tL}\Pi_rV(0)-\L_h e^{tL_h}\P_h^*\Pi_rV(0)\|_{L^2(\Omega;\Gamma)\times H^{-1}(\Omega;\Gamma)}
	\lesssim
	\bigl(h^2r^2+h^{\frac32}r\bigr)
	\|V(0)\|_{H^1(\Omega;\Gamma)\times L^2(\Omega;\Gamma)}.
	\]
	This proves \eqref{eq:improve_etL-etLh_db} for \(k=1\).
\end{proof}

\subsection{Weak-norm estimates for the nonlinear term}
Previously, we provided a series of estimates for linear operators. Next, we will further discuss the estimates for the nonlinear terms in equation \eqref{system_db}.
\begin{lemma}\label{lem:nonlinear_term_estimates}
	Suppose that $f(u)=(f_{\Omega}(u),f_{\Gamma}(u))$, where the nonlinear functions $f_\Omega$ and $f_{\Gamma}$ satisfy the growth conditions \eqref{eq:growth_condition_zeta} and \eqref{eq:growth_condition}. Then, the following estimates hold:
	\begin{align}
		\|f(u)\|_{L^2(\Omega;\Gamma)}&\leq C\big(\|u\|_{H^1(\Omega;\Gamma)}\big),\label{eq:fu_L2}\\[2mm]
		\|f(u)\|_{H^{-1}(\Omega;\Gamma)}&\leq C\big(\|u\|_{H^1(\Omega;\Gamma)}\big)\big(1+\|u\|_{L^2(\Omega;\Gamma)}\big),\label{eq:fu_H-1}\\[2mm]
		\|f(u)-f(v)\|_{H^{-1}(\Omega;\Gamma)}&\leq C\Big(\|u\|_{H^1(\Omega;\Gamma)},\|v\|_{H^1(\Omega;\Gamma)} \Big)\|u-v\|_{L^2(\Omega;\Gamma)},\label{eq:fu-fv}\\[2mm]
		\|f^{\prime}(u)v\|_{H^{-1}(\Omega;\Gamma)}&\leq C\big(\|u\|_{H^1(\Omega;\Gamma)}\big)\|v\|_{L^2(\Omega;\Gamma)},\label{eq:fuv}
	\end{align}
	where $C\big(\|u\|_{H^1(\Omega;\Gamma)}\big)$ and $C\big(\|u\|_{H^1(\Omega;\Gamma)},\|v\|_{H^1(\Omega;\Gamma)} \big)$ denote constants that depend on $\|u\|_{H^1(\Omega;\Gamma)}$ and both $\|u\|_{H^1(\Omega;\Gamma)}$ and $\|v\|_{H^1(\Omega;\Gamma)}$, respectively.
\end{lemma}

\begin{proof}
	The proof is provided in Appendix~\ref{Appendix:C}. 
\end{proof}

We next derive the interpolation error estimate for the nonlinear term. In the present framework, the regularity of both the exact and numerical solutions is below that required for a direct application of the standard interpolation theory. We therefore again make use of the smoother frequency-projected function \(\Pi_r u\), and employ the associated nonlinear term as an intermediate quantity. Following the proof of Lemma~\ref{lem:nonlinear_term_estimates}, and combining the piecewise smoothness estimate for the lift mapping \(G\) from \cite{ER2013} with the standard interpolation estimate in \cite[Theorem~4.4.20]{brenner2008mathematical}, we obtain the following \(H^{-1}(\Omega;\Gamma)\)-error bound.

\begin{lemma}\label{lem:high_order_nonlinear_term_estimates}
		Let $f(\Pi_r u) = \big(f_{\Omega}(\Pi_r u),\, f_{\Gamma}(\Pi_r u)\big)$, 
		where the nonlinear functions $f_{\Omega}$ and $f_{\Gamma}$ satisfy the growth conditions 
		\eqref{eq:growth_condition_zeta} and \eqref{eq:growth_condition}, respectively. Here, \( I_h \) is the bulk--surface interpolation operator defined in \eqref{eq:def_Ih}.
		Then the following estimate holds:
		\begin{align}\label{eq:fu_H_ell}
			\|\L_h^{H*}f(\Pi_r u)- I_h f(\Pi_{r}u)\|_{H^{-1}(\Omega_h;\Gamma_h)}
			\;\le\;
			C\!\left(\|u\|_{H^1(\Omega;\Gamma)}\right) (r^{d-1}h^{d}+h^k),
		\end{align}
		for $k\geq d-1$, and
		\begin{align}\label{eq:fu_H_ell_1}
			\|\L_h^{H*}f(\Pi_r u)-I_h f(\Pi_{r}u)\|_{H^{-1}(\Omega_h;\Gamma_h)}
			\;\le\;
			C\!\left(\|u\|_{H^1(\Omega;\Gamma)}\right) (r^{2}h^{2}+h),
		\end{align}
		for $k=1$, $d=3$, where $d=2,3$ is the spatial dimension. 
		The constant 
		$C\!\left(\|u\|_{H^1(\Omega;\Gamma)}\right)$ 
		depends only on $\|u\|_{H^1(\Omega;\Gamma)}$, 
		and is independent of both $r$ and $h$.
\end{lemma}

\begin{proof}
	The proof is provided in Appendix~\ref{Appendix:C}. 
\end{proof}

We now demonstrate the stability of the interpolation operator when applied to the nonlinear functions.

\begin{lemma}\label{lem:interpolation}
	Let $f_h(u_h) = I_h f(\L_h u_h)$ be defined as in \eqref{eq:def_fh}. Assume that the nonlinear functions
	\(f_{\Omega}\) and \(f_{\Gamma}\) satisfy the growth conditions
	\eqref{eq:growth_condition_zeta} and \eqref{eq:growth_condition}, respectively. Then, for all 
	$u_h, v_h \in X^k_h$, the following estimates hold:
	\begin{align}
		\|f_h(u_h)-f_h(v_h)\|_{H^{-1}(\Omega_h;\Gamma_h)}&\leq C(\|u_h\|_{H^1(\Omega_h;\Gamma_h)}, \|v_h\|_{H^1(\Omega_h;\Gamma_h)})\cdot \|u_h-v_h\|_{L^2(\Omega_h;\Gamma_h)},\label{eq:Ihfuh-Ihfvh_1}\\
		\|f_h(u_h)-f_h(v_h)\|_{L^2(\Omega_h;\Gamma_h)}&\leq C(\|u_h\|_{H^1(\Omega_h;\Gamma_h)}, \|v_h\|_{H^1(\Omega_h;\Gamma_h)})\cdot \|u_h-v_h\|^{\alpha}_{L^2(\Omega_h;\Gamma_h)},\label{eq:Ihfuh-Ihfvh_2}
	\end{align}
	where the constant $\alpha > 0$ depends only on 
	$\zeta_{\Gamma}$, $\zeta_{\Omega}$, and the dimension $d$. 
	Specifically, $\alpha$ can be chosen as
	\begin{align}\label{eq:choose_alpha}
		\alpha=\left\{\begin{array}{l}
			{\displaystyle \frac{1}{2}},\qquad\qquad\qquad\qquad\qquad\quad\;\;\; \text{for}\quad d=2,\\[2mm]
			{\displaystyle \min\left\{\frac{3-\zeta_{\Omega}}{2},\; \frac{1}{2}\right\}},\quad\qquad\qquad\, \text{for}\quad d=3.
		\end{array}\right.
	\end{align}
\end{lemma}

\begin{proof}
	The proof is provided in Appendix~\ref{Appendix:C}. 
\end{proof}

\begin{remark}\upshape
	In three dimensions, the exponent \(\alpha\) in
	\eqref{eq:choose_alpha} depends on the growth exponent \(\zeta_\Omega\) and tends to zero as \(\zeta_\Omega\uparrow3\). Consequently, the
	smallness thresholds \(h_0\) and \(\tau_0\) in
	Theorem~\ref{thm:convergence_db} may deteriorate as
	\(\zeta_\Omega\) approaches the critical value \(3\).
	The convergence result is therefore uniform for every fixed
	\(\zeta_\Omega<3\), but not uniformly with respect to the limit
	\(\zeta_\Omega\uparrow3\).
\end{remark}

\begin{remark}\label{rem:interpolation}\upshape
	We note that, by arguments analogous to those used in the proof of 
	\eqref{eq:Ihfuh-Ihfvh_1} and \eqref{eq:Ihfuh-Ihfvh_2}, 
	it follows directly that
	\begin{align}\label{eq:Ihfuh}
		\|f_h(u_h)\|_{L^2(\Omega_h;\Gamma_h)}
		\le C\big(\|u_h\|_{H^1(\Omega_h;\Gamma_h)}\big),
	\end{align}
	under the same assumptions as in Lemma~\ref{lem:interpolation}.
\end{remark}

\section{Convergence of the fully discrete scheme~\eqref{eq:fully_discrete_db}}\label{sec:proof}

In this section, we prove the convergence result stated in Theorem~\ref{thm:convergence_db}, building on the preliminary estimates established in the previous sections.
We begin by introducing the local consistency error
\[
\mathcal{R}^n := \sum_{i=1}^4 R_i(t_n),
\]
where the terms \(R_i(t_n)\) arise from the consistency error decomposition in \eqref{eq:decompose_U}. By iterating \eqref{eq:decompose_U}, we obtain
\begin{align}\label{eq:iterate_U_db}
	U(t_{n+1})
	=
	e^{t_{n+1}L}U(0)
	+
	\sum_{j=0}^{n}\tau e^{(t_{n+1}-t_{j+1})L}\varphi_1(\tau L)\Pi_r F(\Pi_r U(t_j))
	+
	\sum_{j=0}^n e^{(t_{n+1}-t_{j+1})L}\mathcal{R}^j.
\end{align}
Similarly, the fully discrete scheme \eqref{eq:fully_discrete_db} admits the representation
\begin{align}\label{eq:iterate_Uh_db}
	\L_h U^{n+1}_h
	=
	\L_h e^{t_{n+1}L_h}U^0_h
	+
	\sum_{j=0}^n \tau \L_h e^{(t_{n+1}-t_{j+1})L_h}\varphi_1(\tau L_h)F_h(U^j_h).
\end{align}
The convergence analysis is based on comparing the fully discrete representation \eqref{eq:iterate_Uh_db} with the iterated form \eqref{eq:iterate_U_db} of the frequency-localized auxiliary scheme \eqref{eq:semi-discrete_db}. To this end, we first estimate the semidiscrete consistency error \(\mathcal{R}^n\), and then derive bounds for the approximation of the nonlinear term \(e^{sL}\varphi_1(\tau L)\Pi_r F(\Pi_r U(t_n)).\)

\subsection{Consistency error estimate for the semi-discrete scheme \eqref{eq:semi-discrete_db}}\label{sec:consistency}
In this subsection, we establish the following consistency error estimate:
\begin{lemma}\label{lem:consistency_error_db}
	If the solution \( U \) of the nonlinear wave equation with dynamic boundary conditions \eqref{system_db} 
	in the smooth domain \( \Omega \) possesses the regularity 
	$U \in C([0,T]; H^{1}(\Omega; \Gamma) \times L^{2}(\Omega; \Gamma)),$
	and if the nonlinear term \( F \) is given by \eqref{system_notation_db} 
	with nonlinear functions \( f(u) = (f_{\Omega}(u), f_{\Gamma}(u)) \) 
	satisfying the growth conditions \eqref{eq:growth_condition}, 
	then the following result holds:
	\begin{align*}
		&\|R_1(t_n)\|_{L^2(\Omega;\Gamma)\times H^{-1}(\Omega;\Gamma)}+\|R_3(t_n)\|_{L^2(\Omega;\Gamma)\times H^{-1}(\Omega;\Gamma)}\lesssim\tau^2,\\
		&\|R_2(t_n)\|_{L^2(\Omega;\Gamma)\times H^{-1}(\Omega;\Gamma)}+\|R_4(t_n)\|_{L^2(\Omega;\Gamma)\times H^{-1}(\Omega;\Gamma)}\lesssim \tau r^{-1}.
	\end{align*}
\end{lemma}

\begin{proof}
We first observe, by Taylor's theorem, that
	\begin{align*}
		&\left\|F\big(U(t_n+s)\big)-F(e^{s L}U(t_n))\right\|_{L^2(\Omega;\Gamma)\times H^{-1}(\Omega;\Gamma)}\\
		&\quad=\left\|\int_0^1F^{\prime}\big(\theta U(t_n+s)+(1-\theta)e^{sL}U(t_n)\big)\d \theta \Big(U(t_n+s)-e^{s L}U(t_n)\Big)\right\|_{L^2(\Omega;\Gamma)\times H^{-1}(\Omega;\Gamma)}.
	\end{align*}
	Moreover, the variation-of-constants formula \eqref{eq:variation-of-constant-db} yields
	\begin{align*}
		U(t_n+s)-e^{s L}U(t_n)=\int_{0}^{s}e^{(s-\sigma)L}F\big(U(t_n+\sigma)\big)\d \sigma.
	\end{align*}
	Consequently, we obtain
	\begin{align*}
		&\|R_1(t_n)\|_{L^2(\Omega;\Gamma)\times H^{-1}(\Omega;\Gamma)}\\
		&\lesssim \tau^2 \sup_{\theta,s,\sigma}\left\|F^{\prime}\big(\theta U(t_n+s)+(1-\theta)e^{sL}U(t_n)\big)e^{(s-\sigma)L}F\big(U(t_n+\sigma)\right\|_{L^2(\Omega;\Gamma)\times H^{-1}(\Omega;\Gamma)}.
	\end{align*}
	We denote $U_{\theta}=(U_{\theta,1},U_{\theta,2})^{\top}:=\theta U(t_n+s)+(1-\theta)e^{sL}U(t_n)$, and note that
	\begin{align*}
		F^{\prime}(U_\theta)=\begin{pmatrix}
			0 & 0 \\ f^{\prime}(U_{\theta,1}) & 0
		\end{pmatrix}.
	\end{align*}
	Then, applying the boundedness of \(e^{tL}\) from \eqref{eq:stability_phi1} together with the nonlinear estimates \eqref{eq:fu_H-1} and \eqref{eq:fuv}, we obtain
	\begin{align*}
		&\left\|F^{\prime}\big(\theta U(t_n+s)+(1-\theta)e^{sL}U(t_n)\big)e^{(s-\sigma)L}F\big(U(t_n+\sigma)\right\|_{L^2(\Omega;\Gamma)\times H^{-1}(\Omega;\Gamma)}\\
		&=\left\|f^{\prime}(U_{\theta,1})\Big(e^{(s-\sigma)L}F\big(U(t_n+\sigma)\Big)_1\right\|_{H^{-1}(\Omega;\Gamma)}\\
		&\lesssim C(\|U_{\theta,1}\|_{H^1(\Omega;\Gamma)})\cdot\left\|\Big(e^{(s-\sigma)L}F\big(U(t_n+\sigma)\Big)_1\right\|_{L^2(\Omega;\Gamma)}\\
		&\lesssim C(\|U_{\theta,1}\|_{H^1(\Omega;\Gamma)})\cdot\left\|F\big(U(t_n+\sigma)\right\|_{L^2(\Omega;\Gamma)\times H^{-1}(\Omega;\Gamma)}\\
		&=C(\|U_{\theta,1}\|_{H^1(\Omega;\Gamma)})\cdot \|f(u(t_n+\sigma))\|_{H^{-1}(\Omega;\Gamma)}\\
		&\lesssim C\big(\|U_{\theta,1}\|_{H^1(\Omega;\Gamma)}, \|u(t_n+\sigma)\|_{H^1(\Omega;\Gamma)}\big)\cdot\big(\|u(t_n+\sigma)\|_{L^2(\Omega;\Gamma)}+1\big).
	\end{align*}
	Here, $\big(e^{(s-\sigma)L}F(U(t_n+\sigma))\big)_1$ denotes the first component of the vector
	$e^{(s-\sigma)L}F(U(t_n+\sigma))$.
	
	From the definition of $U_{\theta}$, it follows that
	\begin{align*}
		\|U_{\theta,1}\|_{H^1(\Omega;\Gamma)}&\leq \|U_\theta\|_{H^1(\Omega;\Gamma)\times L^2(\Omega;\Gamma)}\\
		&\lesssim \|U(t_n+s)\|_{H^1(\Omega;\Gamma)\times L^2(\Omega;\Gamma)}+\|U(t_n)\|_{H^1(\Omega;\Gamma)\times L^2(\Omega;\Gamma)}.
	\end{align*}
	Therefore,
	\begin{align}\label{eq:estimate_R1_db}
		\|R_1(t_n)\|_{L^2(\Omega;\Gamma)\times H^{-1}(\Omega;\Gamma)}\lesssim \tau^2 C(\|U\|_{C([0,T]; H^{1}(\Omega; \Gamma) \times L^{2}(\Omega; \Gamma))})\lesssim \tau^2,
	\end{align}
	under the regularity condition  $U \in C([0,T]; H^{1}(\Omega; \Gamma) \times L^{2}(\Omega; \Gamma))$.
	
	For the term $R_3$, we employ the identity 
	\begin{align}\label{eq:d-eFe}
		\frac{\d}{\d \sigma}F(e^{\sigma L}\Pi_r U(t_n))&=\begin{pmatrix}
			0& 0\\
			f^{\prime}\big(\Pi_r\tilde{u}(t_n+\sigma) )& 0
		\end{pmatrix}\begin{pmatrix}
			0& 1\\
			-A & 0
		\end{pmatrix}\begin{pmatrix}
			\Pi_r\tilde{u}(t_n+\sigma)\\
			\Pi_r\tilde{v}(t_n+\sigma)
		\end{pmatrix}\notag\\
		&=\begin{pmatrix}
			0\\ f^{\prime}\big(\Pi_r\tilde{u}(t_n+\sigma)\big)\cdot\Pi_r\tilde{v}(t_n+\sigma)
		\end{pmatrix}
	\end{align}
	where \((\tilde{u}(t_n+\sigma), \tilde{v}(t_n+\sigma))^\top := e^{\sigma L} (u(t_n), v(t_n))^\top\), together with the estimate \eqref{eq:fuv} to derive
	
	%
\begin{align*}
	\|R_3(t_n)\|_{L^2(\Omega;\Gamma)\times H^{-1}(\Omega;\Gamma)}&\leq \tau^2 \sup_{\sigma}\left\|\frac{\d }{\d \sigma}F\left(e^{\sigma L}\Pi_{r} U(t_n)\right)\right\|_{L^2(\Omega;\Gamma)\times H^{-1}(\Omega;\Gamma)}\\
	&=\tau^2 \sup_{\sigma}\left\|f^{\prime}(\Pi_r \tilde{u}(t_n+\sigma))\Pi_{r} \tilde{v}(t_n+\sigma)\right\|_{ H^{-1}(\Omega;\Gamma)}\\
	&\lesssim \tau^2 \sup_{\sigma}C(\|\Pi_r \tilde{u}(t_n+\sigma)\|_{H^1(\Omega;\Gamma)})\cdot \|\Pi_{r} \tilde{v}(t_n+\sigma)\|_{L^2(\Omega;\Gamma)}.
\end{align*}
Since $U \in C([0,T]; H^{1}(\Omega; \Gamma) \times L^{2}(\Omega; \Gamma))$, we have
\begin{align*}
	\|\Pi_r \tilde{u}(t_n+\sigma)\|_{H^1(\Omega;\Gamma)}+\|\Pi_{r} \tilde{v}(t_n+\sigma)\|_{L^2(\Omega;\Gamma)}\lesssim 1,
\end{align*}
and thus
\begin{align}\label{eq:estimate_R3_db}
	\|R_3(t_n)\|_{L^2(\Omega;\Gamma)\times H^{-1}(\Omega;\Gamma)}\lesssim \tau^2.
\end{align}

For the terms $R_2$ and $R_4$, we apply the boundedness of the operators 
$e^{(\tau-s)L}$ and $\varphi_1(\tau L)$ from \ref{eq:stability_phi1}, along with the Bernstein-type inequality \eqref{eq:Bernstein_high} and 
the nonlinear estimates \eqref{eq:fu-fv} and \eqref{eq:fu_L2}. This yields
\begin{align*}
	&\|R_2(t_n)\|_{L^2(\Omega;\Gamma)\times H^{-1}(\Omega;\Gamma)}+\|R_4(t_n)\|_{L^2(\Omega;\Gamma)\times H^{-1}(\Omega;\Gamma)}\\
	&\lesssim \tau \sup_{s} \left\|f(\tilde{u}(t_n+s))-f(\Pi_{r}\tilde{u}(t_n+s))\right\|_{H^{-1}(\Omega;\Gamma)}\\
	&\quad+\tau\sup_{s} \left\|f(\Pi_{r}\tilde{u}(t_n+s))-\Pi_{r} f(\Pi_{r}\tilde{u}(t_n+s))\right\|_{H^{-1}(\Omega;\Gamma)}\\
	&\lesssim \tau \sup_{s}C\Big(\|\tilde{u}(t_n+s)\|_{H^{1}(\Omega;\Gamma)}, \|\Pi_r\tilde{u}(t_n+s)\|_{H^{1}(\Omega;\Gamma)}\Big)\cdot \|\tilde{u}(t_n+s)-\Pi_{r}\tilde{u}(t_n+s)\|_{L^2(\Omega;\Gamma)}\\
	&\quad +\tau r^{-1} \sup_{s}\|f(\Pi_r \tilde{u}(t_n+s))\|_{L^2(\Omega;\Gamma)}\\
	&\lesssim \tau r^{-1} \sup_{s}C\Big(\|\tilde{u}(t_n+s)\|_{H^{1}(\Omega;\Gamma)}, \|\Pi_r\tilde{u}(t_n+s)\|_{H^{1}(\Omega;\Gamma)}\Big).
\end{align*}
Hence, under the condition 
$U \in C([0,T]; H^{1}(\Omega;\Gamma)\times L^{2}(\Omega;\Gamma))$, we deduce that
\begin{align}
	\|R_2(t_n)\|_{L^2(\Omega;\Gamma)\times H^{-1}(\Omega;\Gamma)}+\|R_4(t_n)\|_{L^2(\Omega;\Gamma)\times H^{-1}(\Omega;\Gamma)}\lesssim \tau r^{-1}.
\end{align}
Combining this inequality with \eqref{eq:estimate_R1_db} and \eqref{eq:estimate_R3_db}, 
we establish the desired consistency error estimate of the lemma.
\end{proof}

\subsection{Error estimate for the approximation of the nonlinear term}

We now turn to the estimation of the remaining discrepancy between the fully discrete scheme and the original continuous problem. More precisely, for each \(s\in[0,T]\), we estimate the difference between
\[
e^{sL}\varphi_1(\tau L)\Pi_{r} F(\Pi_{r}U(t_n))
\quad \text{and} \quad
\mathcal{L}_h e^{sL_h}\varphi_1(\tau L_h) F_h(U^{n}_h).
\]
This term will be controlled by combining the lemmas established in Section~\ref{subsec:preliminary_db}. We state the corresponding result in the following lemma.

\begin{lemma}\label{lem:nonlinear_error}
	Under the same assumptions as in Theorem~\ref{thm:convergence_db}, and for $r \leq h^{-1}$, we have
	\begin{align}\label{eq:nonlinear_error}
		&\left\|e^{sL}\varphi_1(\tau L)\Pi_{r} F(\Pi_{r}U(t_n))-\L_h e^{sL_h}\varphi_1(\tau L_h) F_h(U^{n}_h) \right\|_{L^2(\Omega;\Gamma)\times H^{-1}(\Omega;\Gamma)}\notag\\
		&\leq C\Big(\|U(t_n)\|_{H^1\times L^2}, \|U^n_h\|_{H^1\times L^2}\Big)\cdot\Big(h^{\ell_k}r^{\ell_k}+h^{k+\frac{1}{2}}r^{k}+h+r^{-1}\Big)\notag\\
		&\quad+C\Big(\|U(t_n)\|_{H^1\times L^2}, \|U^n_h\|_{H^1\times L^2}\Big)\cdot \left\|U(t_n)-\L_h U^n_h \right\|_{L^2(\Omega;\Gamma)\times H^{-1}(\Omega;\Gamma)},
	\end{align}
	where \(\ell_k\) is defined in \eqref{eq:def_lk}, and
	\(C\!\left(\|U(t_n)\|_{H^1\times L^2}, \|U^n_h\|_{H^1\times L^2}\right)\)
	denotes a constant depending only on
	\(\|U(t_n)\|_{H^1\times L^2}\) and \(\|U^n_h\|_{H^1\times L^2}\), and is independent of \(\tau\), \(h\), and \(r\).
\end{lemma}

\begin{proof}
	We decompose the difference between 
	$e^{sL}\varphi_1(\tau L)\Pi_{r} F(\Pi_{r}U(t_n))$ 
	and $\mathcal{L}_h e^{sL_h}\varphi_1(\tau L_h) F_h(U^{n}_h)$ as follows:
	\begin{align}\label{eq:def_N1N2N3N4}
		&e^{sL}\varphi_1(\tau L)\Pi_{r} F(\Pi_{r}U(t_n))-\L_h e^{sL_h}\varphi_1(\tau L_h) F_h(U^{n}_h)\notag\\
		&=e^{sL}\varphi_1(\tau L)\Pi_{r} F(\Pi_{r}U(t_n))-\L_h e^{sL_h}\varphi_1(\tau L_h) \P^*_h \Pi_{r} F(\Pi_{r}U(t_n))\notag\\
		&\quad+\L_h e^{sL_h}\varphi_1(\tau L_h) \P^*_h \Pi_{r} F(\Pi_{r}U(t_n))-\L_h e^{sL_h}\varphi_1(\tau L_h) \P^*_h  F(\Pi_{r}U(t_n))\notag\\
		&\quad +\L_h e^{sL_h}\varphi_1(\tau L_h) \P^*_h  F(\Pi_{r}U(t_n))-\L_h e^{sL_h}\varphi_1(\tau L_h) I_h  F(\Pi_{r}U(t_n))\notag\\
		&\quad +\L_h e^{sL_h}\varphi_1(\tau L_h) I_h  F(\Pi_{r}U(t_n))-\L_h e^{sL_h}\varphi_1(\tau L_h)  F_h(U^n_h)\notag\\
		&:=N_1+N_2+N_3+N_4,
	\end{align}
	where $\P^*_h=\text{diag}(\L^{H*}_h,\L^{H*}_h)$. The four terms have the following interpretations: \(N_1\) represents the error caused by replacing the continuous operator \(e^{sL}\varphi_1(\tau L)\) with its discrete counterpart \(e^{sL_h}\varphi_1(\tau L_h)\); \(N_2\) is the error induced by the frequency truncation of the nonlinear term; \(N_3\) measures the discrepancy between the adjoint lift operator \(\mathcal{L}_h^{H*}\) and the interpolation operator \(I_h\); and \(N_4\) contains the nonlinear stability term associated with \(I_h\).
	
	We begin with \(N_1\). Note that
	\begin{align}
		e^{sL}\varphi_1(\tau L)\Pi_{r} F(\Pi_{r}U(t_n))=&\frac{1}{\tau}\int_0^{\tau}e^{(s+\sigma)L}\Pi_{r} F(\Pi_{r}U(t_n))\d \sigma,\label{eq:expand_phi1}\\
		\L_h e^{sL_h}\varphi_1(\tau L_h) \P^*_h \Pi_{r} F(\Pi_{r}U(t_n))=&\frac{1}{\tau}\int_0^{\tau}\L_h e^{(s+\sigma)L_h} \P^*_h \Pi_{r} F(\Pi_{r}U(t_n))\d \sigma.\label{eq:expand_phi1_h}
	\end{align}
	Since \(F(U)=(0,f(u))^\top\), by combining representations \eqref{eq:expand_phi1} and \eqref{eq:expand_phi1_h} with the estimate between two different exponential operators in Theorem~\ref{lem:estimate_etL-etLh_db} and the nonlinear term estimate \eqref{eq:fu_L2} in Lemma~\ref{lem:nonlinear_term_estimates}, we obtain
	\begin{align}\label{eq:estimate_N1}
		&\|N_1\|_{L^2(\Omega;\Gamma)\times H^{-1}(\Omega;\Gamma)}\notag\\
		&\leq \sup_{\sigma\in[0,\tau]}\left\|e^{(s+\sigma)L}\Pi_{r} F(\Pi_{r}U(t_n))-\L_h e^{(s+\sigma)L_h} \P^*_h \Pi_{r} F(\Pi_{r}U(t_n))\right\|_{L^2(\Omega;\Gamma)\times H^{-1}(\Omega;\Gamma)}\notag\\
		&\lesssim (h^{\ell_k}r^{\ell_k}+h^{k+\frac{1}{2}}r^{k}+h)\left\|\Pi_{r} F(\Pi_{r}U(t_n))\right\|_{H^1(\Omega;\Gamma)\times L^{2}(\Omega;\Gamma)}\notag\\
		&=(h^{\ell_k}r^{\ell_k}+h^{k+\frac{1}{2}}r^k+h)\left\|\Pi_{r} f(\Pi_{r}u(t_n))\right\|_{L^2(\Omega;\Gamma)}\notag\\
		&\leq C\Big(\left\|U(t_n)\right\|_{H^1(\Omega;\Gamma)\times L^{2}(\Omega;\Gamma)} \Big)\cdot (h^{\ell_k}r^{\ell_k}+h^{k+\frac{1}{2}}r^k+h).
	\end{align}
	For the remaining terms, \( N_2 \), \( N_3 \), and \( N_4 \), using the boundedness of 
	\(\mathcal{L}_h\) from
	\(L^2(\Omega_h;\Gamma_h)\times H^{-1}(\Omega_h;\Gamma_h)\)
	to
	\(L^2(\Omega;\Gamma)\times H^{-1}(\Omega;\Gamma)\) (see~\eqref{eq:boundedness_lift_operator} and \eqref{eq:equivalence_H-1} in Corollary~\ref{cor:boundedness_properties}),
	together with the uniform boundedness of
	\(e^{sL_h}\varphi_1(\tau L_h)\) on
	\(L^2(\Omega_h;\Gamma_h)\times H^{-1}(\Omega_h;\Gamma_h)\) (see Lemma~\ref{lem:basic_results}(c)--(d)), we infer that
	\begin{align}\label{eq:estimate_N2N3N4}
		&\|N_2\|_{L^2(\Omega;\Gamma)\times H^{-1}(\Omega;\Gamma)}+\|N_3\|_{L^2(\Omega;\Gamma)\times H^{-1}(\Omega;\Gamma)}+\|N_4\|_{L^2(\Omega;\Gamma)\times H^{-1}(\Omega;\Gamma)}\notag\\
		&\lesssim \|\P^*_h \Pi_{r} F(\Pi_{r}U(t_n))-\P^*_h  F(\Pi_{r}U(t_n))\|_{L^2(\Omega_h;\Gamma_h)\times H^{-1}(\Omega_h;\Gamma_h)}\notag\\
		&\quad+\|\P^*_h F(\Pi_{r}U(t_n))-I_h F(\Pi_{r}U(t_n))\|_{L^2(\Omega_h;\Gamma_h)\times H^{-1}(\Omega_h;\Gamma_h)}\notag\\
		&\quad+\|I_h  F(\Pi_{r}U(t_n))- F_h(U^n_h)\|_{L^2(\Omega_h;\Gamma_h)\times H^{-1}(\Omega_h;\Gamma_h)}\notag\\
		&:=\|\widetilde{N}_2\|_{L^2(\Omega_h;\Gamma_h)\times H^{-1}(\Omega_h;\Gamma_h)}+\|\widetilde{N}_3\|_{L^2(\Omega_h;\Gamma_h)\times H^{-1}(\Omega_h;\Gamma_h)}+\|\widetilde{N}_4\|_{L^2(\Omega_h;\Gamma_h)\times H^{-1}(\Omega_h;\Gamma_h)}.
	\end{align}
	For the term $\widetilde{N}_2$, we apply the boundedness of 
	$\mathcal{L}^{H*}_h$ in both the $L^2$ and $H^{-1}$ norms 
	(see \eqref{eq:L2_bound_Lh*} and \eqref{eq:bounded_H-1} in Corollary~\ref{cor:boundedness_properties}), Bernstein-type inequality \eqref{eq:Bernstein_high},
	together with the nonlinear estimate for $f$ in \eqref{eq:fu_L2}. 
	This yields
	\begin{align}\label{eq:estimate_N2}
		&\|\widetilde{N}_2\|_{L^2(\Omega_h;\Gamma_h)\times H^{-1}(\Omega_h;\Gamma_h)}\lesssim\|\Pi_{r} F(\Pi_{r}U(t_n))- F(\Pi_{r}U(t_n))\|_{L^2(\Omega;\Gamma)\times H^{-1}(\Omega;\Gamma)}\notag\\
		&\quad=\|\Pi_{r} f(\Pi_{r}u(t_n))- f(\Pi_{r}u(t_n))\|_{H^{-1}(\Omega;\Gamma)}\notag\\
		&\quad\lesssim r^{-1}\|f(\Pi_r u(t_n))\|_{L^2(\Omega;\Gamma)}\lesssim r^{-1}\cdot C\Big(\left\|U(t_n)\right\|_{H^1(\Omega;\Gamma)\times L^{2}(\Omega;\Gamma)} \Big).
	\end{align}
For $\widetilde{N}_3$, by the definitions of the nonlinear term $F$ and the operator $\P_h^*$, we have
	\begin{align}\label{eq:tilde_N3}
		\|\widetilde{N}_3\|_{L^2(\Omega_h;\Gamma_h)\times H^{-1}(\Omega_h;\Gamma_h)}=\|\L_h^{H*} f(\Pi_{r}u(t_n))-I_h f(\Pi_{r}u(t_n))\|_{H^{-1}(\Omega_h;\Gamma_h)}.
	\end{align}
	We distinguish two cases.
	
	First, if $k=1$ and $d=3$, then by \eqref{eq:fu_H_ell_1} in Lemma~\ref{lem:high_order_nonlinear_term_estimates},
		\begin{align}\label{eq:interpolation_1}
			\|\L_h^{H*} f(\Pi_{r}u(t_n))-I_h f(\Pi_{r}u(t_n))\|_{H^{-1}(\Omega_h;\Gamma_h)}&\leq C\Big(\left\|U(t_n)\right\|_{H^1(\Omega;\Gamma)\times L^{2}(\Omega;\Gamma)} \Big)\cdot (h^2 r^2+h)\notag\\
			&\lesssim h^{\ell_k} r^{\ell_k}+h.
		\end{align}

		For all other cases, we have $k\ge d-1$, and therefore \eqref{eq:fu_H_ell} in Lemma~\ref{lem:high_order_nonlinear_term_estimates} yields
		\begin{align}\label{eq:interpolation_2}
		\|\L_h^{H*} f(\Pi_{r}u(t_n))-I_h f(\Pi_{r}u(t_n))\|_{H^{-1}(\Omega_h;\Gamma_h)}&\leq C\Big(\left\|U(t_n)\right\|_{H^1(\Omega;\Gamma)\times L^{2}(\Omega;\Gamma)} \Big)\cdot (h^{d} r^{d-1}+h)\notag\\
		&\lesssim h,
		\end{align}
		where in the last step we used the condition $r\le h^{-1}$. 
		
		Combining \eqref{eq:tilde_N3}, \eqref{eq:interpolation_1}, and \eqref{eq:interpolation_2}, we conclude that in all cases
		\begin{align}\label{eq:estimate_N3}
			\|\widetilde{N}_3\|_{L^2(\Omega_h;\Gamma_h)\times H^{-1}(\Omega_h;\Gamma_h)}\leq (h+h^{\ell_k}r^{\ell_k})\cdot C\Big(\|U(t_n)\|_{H^1(\Omega;\Gamma)\times L^2(\Omega;\Gamma)}\Big).
	\end{align}
	
	For the term \(\widetilde{N}_4\) in \eqref{eq:estimate_N2N3N4}, recall that \(I_h\) denotes the interpolation operator defined in \eqref{eq:def_Ih}. Since \(\Pi_r u(t_n)\) is trace-compatible and \(G\) preserves the interpolation nodes, the functions \(\mathcal L_h I_h \Pi_r u(t_n)\) and \(\Pi_r u(t_n)\) agree at all nodal points. Consequently, their bulk and surface nonlinear nodal interpolants coincide. Applying \(Q_h\) to the resulting, possibly trace-incompatible, bulk--surface pair therefore yields
		\begin{align*}
				I_h  F(\Pi_{r}U(t_n))=F_h(I_h\Pi_{r}U(t_n)).
		\end{align*} 
	Applying Lemma~\ref{lem:interpolation} we deduce:
	\begin{align}\label{eq:decompose_N4}
		&\|\widetilde{N}_4\|_{L^2(\Omega_h;\Gamma_h)\times H^{-1}(\Omega_h;\Gamma_h)}=\|F_h(I_h\Pi_{r}U(t_n))-  F_h(U^n_h)\|_{L^2(\Omega_h;\Gamma_h)\times H^{-1}(\Omega_h;\Gamma_h)}\notag\\[2mm]
		&=\|f_h(I_h\Pi_{r}u(t_n))- f_h(u^n_h)\|_{H^{-1}(\Omega_h;\Gamma_h)}\notag\\
		&\leq C\Big(\|I_h\Pi_{r}u(t_n)\|_{H^1(\Omega_h;\Gamma_h)},\|u^n_h\|_{H^1(\Omega_h;\Gamma_h)}\Big)\cdot \|I_h\Pi_{r}u(t_n)-u^n_h\|_{L^2(\Omega_h;\Gamma_h)}.
	\end{align}
	Using Proposition~5.4 of \cite{ER2013}, which gives the \(H^1(\Omega;\Gamma)\)-approximation property of \(\mathcal{L}_h I_h\), together with the boundedness of \(\mathcal{L}_h\) and the Bernstein-type inequality \eqref{eq:Bernstein_low}, we obtain
	\begin{align*}
		\|I_h\Pi_{r}u(t_n)\|_{H^1(\Omega_h;\Gamma_h)}&\lesssim \|\L_h I_h\Pi_{r}u(t_n)\|_{H^1(\Omega;\Gamma)}\\
		&\leq \|\Pi_{r}u(t_n)\|_{H^1(\Omega;\Gamma)}+\|\Pi_{r}u(t_n)-\L_h I_h\Pi_{r}u(t_n)\|_{H^1(\Omega;\Gamma)}\\
		&\lesssim \|\Pi_{r}u(t_n)\|_{H^1(\Omega;\Gamma)}+h\|\Pi_{r}u(t_n)\|_{H^2(\Omega;\Gamma)}\\
		&\lesssim (1+hr)\|\Pi_{r}u(t_n)\|_{H^1(\Omega;\Gamma)}.
	\end{align*}
	When $r\le h^{-1}$, this gives
	\begin{align}\label{eq:estimate_N4_1}
		\|I_h\Pi_{r}u(t_n)\|_{H^1(\Omega_h;\Gamma_h)}\lesssim\|\Pi_{r}u(t_n)\|_{H^1(\Omega;\Gamma)}.
	\end{align}
	Similarly
	\begin{align}\label{eq:estimate_N4_2}
			&\|I_h\Pi_{r}u(t_n)-u^n_h\|_{L^2(\Omega_h;\Gamma_h)}\lesssim \|\L_h I_h\Pi_{r}u(t_n)-\L_h u^n_h\|_{L^2(\Omega;\Gamma)}\notag\\
			&\quad\leq \|\Pi_{r}u(t_n)-\L_h I_h\Pi_{r}u(t_n)\|_{L^2(\Omega;\Gamma)}+\|u(t_n)-\Pi_{r}u(t_n)\|_{L^2(\Omega;\Gamma)}\notag\\
			&\qquad+\|u(t_n)-\L_h u^n_h\|_{L^2(\Omega;\Gamma)}\notag\\
			&\quad\lesssim h^{k+1}\|\Pi_{r}u(t_n)\|_{H^{k+1}(\Omega;\Gamma)}+r^{-1}\|u(t_n)\|_{H^1(\Omega;\Gamma)}+\|u(t_n)-\L_h u^n_h\|_{L^2(\Omega;\Gamma)}\notag\\
			&\quad \lesssim \left(h^{k+1} r^k +r^{-1}\right)\|u(t_n)\|_{H^1(\Omega;\Gamma)}+\|U(t_n)-\L_h U^n_h\|_{L^2(\Omega;\Gamma)\times H^{-1}(\Omega;\Gamma)}.
	\end{align}
	Substituting \eqref{eq:estimate_N4_1} and \eqref{eq:estimate_N4_2} into
	\eqref{eq:decompose_N4}, and using the condition \(r \le h^{-1}\), we obtain
	\begin{align}\label{eq:estimate_N4}
			&\|\widetilde{N}_4\|_{L^2(\Omega_h;\Gamma_h)\times H^{-1}(\Omega_h;\Gamma_h)}\notag\\
			&\leq C\Big(\|u(t_n)\|_{H^1(\Omega;\Gamma)},\|u^n_h\|_{H^1(\Omega_h;\Gamma_h)}\Big)\cdot \left(h +r^{-1}+\|U(t_n)-\L_h U^n_h\|_{L^2(\Omega;\Gamma)\times H^{-1}(\Omega;\Gamma)}\right).
	\end{align}
	Finally, by combining \eqref{eq:def_N1N2N3N4}, \eqref{eq:estimate_N1}, \eqref{eq:estimate_N2N3N4}, \eqref{eq:estimate_N2}, \eqref{eq:estimate_N3} and \eqref{eq:estimate_N4}, we derive the desired bound~\eqref{eq:nonlinear_error}, 
	which concludes the proof of the lemma.
\end{proof}

\subsection{The proof of Theorem~\ref{thm:convergence_db}}
	Subtracting \eqref{eq:iterate_Uh_db} from \eqref{eq:iterate_U_db} yields
	\begin{align}\label{eq:U-Uh_db}
		&U(t_{n+1})-\L_h U^{n+1}_h=e^{t_{n+1}L}U(0)-\L_h e^{t_{n+1} L_h}U_h^0\notag\\
		&+\sum_{j=0}^{n}\tau \left(e^{(t_{n+1}-t_{j+1})L}\varphi_1(\tau L)\Pi_r F(\Pi_r U(t_j))-\L_h e^{(t_{n+1}-t_{j+1}) L_h}\varphi_1(\tau L_h)F_h(U^j_h)\right)\notag\\
		&+\sum_{j=0}^n e^{(t_{n+1}-t_{j+1})L} \mathcal{R}^j.
	\end{align}
	For the first term on the right-hand side of \eqref{eq:U-Uh_db}, 
	note that $U^0_h = \P^*_h U^0 = (\L^{H*}_h u^0, \L^{H*}_h v^0)^{\top}$.  
	Applying Theorem~\ref{lem:estimate_etL-etLh_db} and using the boundedness of the operators 
	$\L_h$, $e^{t_{n+1}L_h}$, $e^{t_{n+1}L}$, and $\L^{*H}_h$ in the 
	$L^2$- and $H^{-1}$-norms, we obtain
	\begin{align}
		&\left\|e^{t_{n+1}L}U(0)-\L_h e^{t_{n+1} L_h}U_h^0\right\|_{L^2(\Omega;\Gamma)\times H^{-1}(\Omega;\Gamma)}\notag\\
		&\quad\leq \left\|e^{t_{n+1}L}U(0)-e^{t_{n+1}L}\Pi_r U(0)\right\|_{L^2(\Omega;\Gamma)\times H^{-1}(\Omega;\Gamma)}\notag\\
		&\qquad+\left\|e^{t_{n+1}L}\Pi_r U(0)-\L_h e^{t_{n+1} L_h}\P^*_h \Pi_r U(0) \right\|_{L^2(\Omega;\Gamma)\times H^{-1}(\Omega;\Gamma)}\notag\\
		&\qquad+\left\|\L_h e^{t_{n+1} L_h}\P^*_h \Pi_r U(0)-\L_h e^{t_{n+1} L_h} \P^*_h U(0) \right\|_{L^2(\Omega;\Gamma)\times H^{-1}(\Omega;\Gamma)}\notag\\
		&\quad \lesssim \|U(0)-\Pi_r U(0)\|_{L^2(\Omega;\Gamma)\times H^{-1}(\Omega;\Gamma)}+(h^{\ell_k}r^{\ell_k}+h^{k+\frac{1}{2}}r^k+h)\|U(0) \|_{H^1(\Omega;\Gamma)\times L^2(\Omega;\Gamma)}\notag\\
		&\quad\lesssim C\Big(\|U(0) \|_{H^1(\Omega;\Gamma)\times L^2(\Omega;\Gamma)}\Big)\cdot\left(r^{-1}+h^{\ell_k}r^{\ell_k}+h^{k+\frac{1}{2}}r^k+h\right).
	\end{align}
	The second and third terms on the right-hand side of \eqref{eq:U-Uh_db} 
	can be bounded directly using Lemma~\ref{lem:nonlinear_error} 
	and Lemma~\ref{lem:consistency_error_db}.  
	Under the assumptions of Theorem~\ref{thm:convergence_db}, we know that 
	$U \in C([0,T];H^1(\Omega;\Gamma)\times L^2(\Omega;\Gamma))$, namely,
	\begin{align*}
		\sup_{t\in[0,T]}\|U(t) \|_{H^1(\Omega;\Gamma)\times L^2(\Omega;\Gamma)}\lesssim 1.
	\end{align*}
	Thus we have
	\begin{align}\label{eq:gronwall_1}
		&\left\|U(t_{n+1})-\L_h U^{n+1}_h\right\|_{L^2(\Omega;\Gamma)\times H^{-1}(\Omega;\Gamma)}\notag\\
		&\leq C_0 \cdot \left(r^{-1}+h^{\ell_k}r^{\ell_k}+h^{k+\frac{1}{2}}r^k+h\right)\notag\\
		&\quad+\sum_{j=0}^{n}\tau \cdot C\Big(\|U^j_h\|_{H^1(\Omega_h;\Gamma_h)\times L^2(\Omega_h;\Gamma_h)}\Big)\cdot\Big(r^{-1}+h^{\ell_k}r^{\ell_k}+h^{k+\frac{1}{2}}r^k+h\Big)\notag\\
		&\quad +\sum_{j=0}^{n}\tau \cdot C\Big(\|U^j_h\|_{H^1(\Omega_h;\Gamma_h)\times L^2(\Omega_h;\Gamma_h)}\Big)\cdot\left\|U(t_{j})-\L_h U^{j}_h\right\|_{L^2(\Omega;\Gamma)\times H^{-1}(\Omega;\Gamma)}\notag\\
		&\quad +\sum_{j=0}^{n} \tau\cdot  C_0 \cdot (\tau+r^{-1}),
	\end{align}
	where $C_0$ is a constant independent of $\tau$, $h$, $r$, and the bound of 
	the numerical solution, and where 
	$C(\|U^j_h\|_{H^1(\Omega_h;\Gamma_h)\times L^2(\Omega_h;\Gamma_h)})$
	denotes a constant depending only on 
	$\|U^j_h\|_{H^1(\Omega_h;\Gamma_h)\times L^2(\Omega_h;\Gamma_h)}$.
	
	Since some of the constants in \eqref{eq:gronwall_1} depend on the 
	$H^1\times L^2$-norm of the numerical solution, 
	it is essential to guarantee the boundedness of the numerical solution in this functional space 
	before applying Gronwall’s inequality to \eqref{eq:gronwall_1} 
	and thereby establishing the error estimate stated in Theorem~\ref{thm:convergence_db}. 
	
	Taking the $H^1\times L^2$-norm on both sides of \eqref{eq:fully_discrete_db} 
	and using the boundedness of the operators $e^{t_{n+1} L_h}$ and $\varphi_1(\tau L_h)$ from Lemma~\ref{lem:basic_results} (c)--(d), we obtain
	\begin{align}\label{eq:U_H1}
		&\|U^{n+1}_h\|_{H^1(\Omega_h;\Gamma_h)\times L^2(\Omega_h;\Gamma_h)}\leq\left\| e^{t_{n+1} L_h}U^0_h\right\|_{H^1(\Omega_h;\Gamma_h)\times L^2(\Omega_h;\Gamma_h)}\notag\\
		&\quad\qquad\qquad\qquad\qquad\qquad\qquad+\sum_{j=0}^n\tau \left\| e^{(t_{n+1}-t_{j+1}) L_h}\varphi_1(\tau L_h) F_h(U^j_h)\right\|_{H^1(\Omega_h;\Gamma_h)\times L^2(\Omega_h;\Gamma_h)}\notag\\
		&\lesssim \left\| U^0_h\right\|_{H^1(\Omega_h;\Gamma_h)\times L^2(\Omega_h;\Gamma_h)}+\sum_{j=0}^n\tau \left\| F_h(U^j_h)\right\|_{H^1(\Omega_h;\Gamma_h)\times L^2(\Omega_h;\Gamma_h)}.
	\end{align}
	For the second term on the right-hand side of \eqref{eq:U_H1}, 
	we insert the interpolation of the nonlinear term $\Pi_r U(t_j)$.  
	By applying the estimates in Lemma~\ref{lem:interpolation} and Remark~\ref{rem:interpolation} 
	for the interpolation operator $I_h$, we have
\begin{align}\label{eq:IhFhUh}
	&\left\|F_h(U^j_h)\right\|_{H^1(\Omega_h;\Gamma_h)\times L^2(\Omega_h;\Gamma_h)}\leq \left\|f_h(I_h \Pi_r u(t_j))\right\|_{L^2(\Omega_h;\Gamma_h)}+\left\|f_h(I_h \Pi_r u(t_j))- f_h(u^j_h)\right\|_{L^2(\Omega_h;\Gamma_h)}\notag\\[2mm]
	&\quad\lesssim  C\Big(\left\|I_h \Pi_r u(t_j)\right\|_{H^1(\Omega_h;\Gamma_h)}\Big)\notag\\
	&\qquad+C\Big(\left\|I_h \Pi_r u(t_j)\right\|_{H^1(\Omega_h;\Gamma_h)},\|u^j_h\|_{H^1(\Omega_h;\Gamma_h)}\Big)\cdot \left\|I_h \Pi_r u(t_j)-u^j_h\right\|^\alpha_{L^2(\Omega_h;\Gamma_h)}.
\end{align}
Using the interpolation error estimate from Proposition~5.4 of \cite{ER2013} together with the Bernstein-type inequality \eqref{eq:Bernstein_low}, we obtain
\begin{align*}
	\left\|\L_h I_h \Pi_r u(t_j)-\Pi_r u(t_j)\right\|_{H^1(\Omega;\Gamma)}\lesssim h \left\|\Pi_r u(t_j)\right\|_{H^2(\Omega;\Gamma)}\lesssim hr\left\|\Pi_r u(t_j)\right\|_{H^1(\Omega;\Gamma)},
\end{align*}
we deduce that, under the condition $r \leq h^{-1}$ and with $u(t_j)\in H^1(\Omega;\Gamma)$,
\begin{align*}
	\left\|I_h \Pi_r u(t_j)\right\|_{H^1(\Omega_h;\Gamma_h)}\lesssim \left\|\L_h I_h \Pi_r u(t_j)\right\|_{H^1(\Omega;\Gamma)}\lesssim \left\|\Pi_r u(t_j)\right\|_{H^1(\Omega;\Gamma)}+hr\left\|\Pi_r u(t_j)\right\|_{H^1(\Omega;\Gamma)}\lesssim 1,
\end{align*}
Substituting this bound and the estimate \eqref{eq:estimate_N4_2} 
for $\|I_h \Pi_r u(t_j)-u^j_h\|$ into \eqref{eq:IhFhUh} yields
\begin{align*}
	&\left\|F_h(U^j_h)\right\|_{H^1(\Omega_h;\Gamma_h)\times L^2(\Omega_h;\Gamma_h)}\notag\\
	&\leq C_1 +C\Big(\|U^j_h\|_{H^1(\Omega_h;\Gamma_h)\times L^2(\Omega_h;\Gamma_h)}\Big)\cdot \left(h+r^{-1}+\left\|U(t_{j})-\L_h U^{j}_h\right\|_{L^2(\Omega;\Gamma)\times H^{-1}(\Omega;\Gamma)}\right)^{\alpha},
\end{align*}
where we have used the condition $r\leq h^{-1}$ and constant $C_1$ is independent of $\tau$, $h$, $r$, and the bound of the numerical solution.
	
	Substituting this estimate into \eqref{eq:U_H1} and noting that
	\begin{align*}
		\left\| U^0_h\right\|_{H^1(\Omega_h;\Gamma_h)\times L^2(\Omega_h;\Gamma_h)}=\left\|\P^*_h U^0\right\|_{H^1(\Omega_h;\Gamma_h)\times L^2(\Omega_h;\Gamma_h)}\lesssim \|U^0\|_{H^1(\Omega;\Gamma)\times L^2(\Omega;\Gamma)}\lesssim 1,
	\end{align*}
	we deduce
	\begin{align}\label{eq:gronwall_2}
		&\|U^{n+1}_h\|_{H^1(\Omega_h;\Gamma_h)\times L^2(\Omega_h;\Gamma_h)}\notag\\
		&\leq C_0+\sum_{j=0}^n \tau C\Big(\|U^j_h\|_{H^1(\Omega_h;\Gamma_h)\times L^2(\Omega_h;\Gamma_h)}\Big)\cdot \left(r^{-1}+\left\|U(t_{j})-\L_h U^{j}_h\right\|_{L^2(\Omega;\Gamma)\times H^{-1}(\Omega;\Gamma)}\right)^{\alpha},
	\end{align}
	where $C_0$ denotes a constant independent of $\tau$, $h$, $r$, and the numerical solution bound.
	
	To balance the different error contributions, we choose \(r\) such that
	\begin{align*}
		r^{-1}=h^{\ell_k}r^{\ell_k},\qquad\text{that is}\qquad r=h^{-\frac{\ell_k}{\ell_k+1}}.
	\end{align*}
	By the definition of \(\ell_k\) in \eqref{eq:def_lk}, one readily verifies that \(-(k+\frac{1}{2})\frac{\ell_k+1}{\ell_k}+k<-1\) and \(-\frac{\ell_k+1}{\ell_k}<-1\). Hence,
	\begin{align*}
		h^{k+\frac12}r^k+h \lesssim r^{-1}=h^{\ell_k}r^{\ell_k}=h^{\frac{\ell_k}{\ell_k+1}}
	\end{align*}
	for all \(k\ge 1\) and \(h\leq 1\).
	
Let
\[
\delta:=h^{\frac{\ell_k}{\ell_k+1}}, \qquad
a^n=\left\|U(t_n)-\mathcal L_h U_h^n\right\|_{L^2(\Omega;\Gamma)\times H^{-1}(\Omega;\Gamma)},
\qquad
b^n=\|U_h^n\|_{H^1(\Omega_h;\Gamma_h)\times L^2(\Omega_h;\Gamma_h)}.
\]
Then \eqref{eq:gronwall_1} and \eqref{eq:gronwall_2} can be rewritten as
\begin{align}\label{ineq_ab}
	a^{n+1}
	\le A+\sum_{j=0}^n \tau\, C(b^j)(\delta+a^j),\qquad
	b^{n+1}
	\le B+\sum_{j=0}^n \tau\, C(b^j)(\delta+a^j)^\alpha,
\end{align}
with
\[
A=C_0(\tau+\delta), \qquad B=C_0.
\]
Moreover,
\[
a^0=\left\|U(0)-\mathcal L_h P_h^*U(0)\right\|_{L^2\times H^{-1}}\le C_0h,
\qquad
b^0=\|P_h^*U(0)\|_{H^1\times L^2}\le C_0.
\]
Since \(\frac{\ell_k}{\ell_k+1}<1\), for \(0<h<1\) we have \(h\le \delta\), and therefore
\[
a^0\le  C_0\delta \le A, \qquad b^0\le B.
\]
Possibly enlarging the function \(C\), we may assume that \(C:[0,\infty)\to [0,\infty)\) is nondecreasing. Set
\[
\widetilde C:=\sup_{0\le s\le C_0+1}C(s).
\]
Since \(\alpha>0\) and \(A,\delta\to 0\) as \(h,\tau\to 0\), there exist \(h_0>0\) and \(\tau_0>0\) such that, for all \(0<h<h_0\) and \(0<\tau<\tau_0\),
\[
\widetilde C T\bigl(\delta+(A+\widetilde C T\delta)e^{T\widetilde C}\bigr)^\alpha\le 1.
\]

Therefore, Lemma~4.3 in \cite{cao2026approximating} applies to \eqref{ineq_ab}, and for all \(0\le n\le T/\tau\),
\[
b^n\le B+1=C_0+1,\qquad
a^n\le (A+\widetilde C T\delta)e^{n\tau\widetilde C}.
\]
Since \(n\tau\le T\), we conclude that
\begin{align*}
	\sup_{0\leq n\leq T/\tau}\|U(t_n)-\L_h U^n_h\|_{L^2(\Omega;\Gamma)\times H^{-1}(\Omega;\Gamma)}\leq C_0 e^{T\widetilde{C}}\cdot \tau +\big(C_0+\widetilde{C}T\big)e^{T\widetilde{C}} \cdot h^{\frac{\ell_k}{\ell_k+1}}.
\end{align*}
This completes the proof of Theorem~\ref{thm:convergence_db}.

\section{Numerical experiments}\label{sec:numerical_experiments}

In this section, we present a numerical experiment on a peanut-shaped domain, see Figure~\ref{fig:peanut_domain}. The domain is centered around the origin and has maximal width of $\approx1.8$ in x-direction and $\approx0.8$ in y-direction, while the narrowest part is $\approx0.4$. Note that this domain allows for smooth eigenfunctions without being convex. 
Since we have already discussed convergence on the square and an equilateral triangle in \cite{cao2026approximating}, and since other smooth domains do not provide further insight, we restrict ourselves to this one domain.
As pointed out in Remark~\ref{rem:other_bc}, our theory extends to the homogeneous Dirichlet problem,
and we thus include them in our experiments.

\begin{figure}
	\includegraphics[scale=0.3]{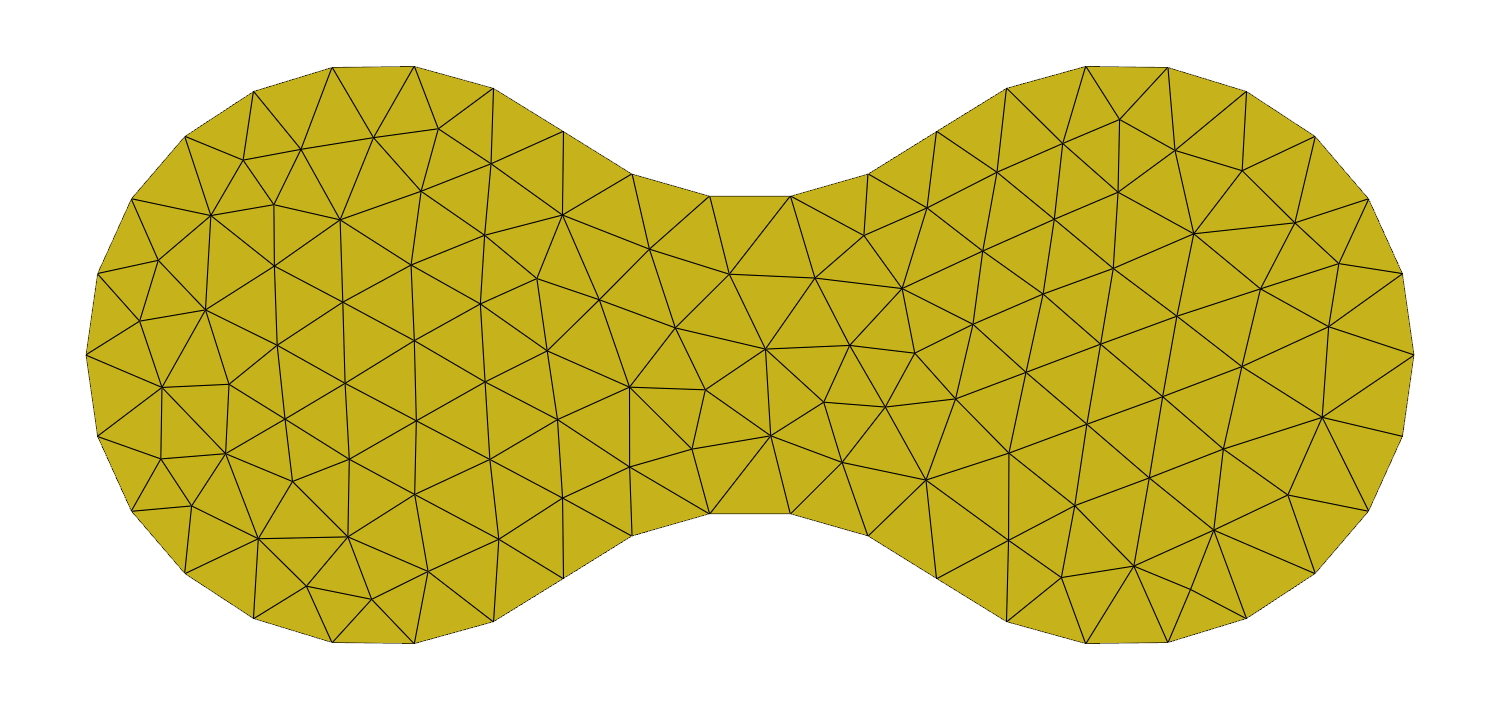}
	\caption{Plot of the peanut-shaped domain triangulated with \texttt{Gmsh} and mesh width $h \approx \frac18$.}
	\label{fig:peanut_domain}
\end{figure}

\subsection{Implementation}

We first discuss the implementation of the algorithm 
in Python. We first generate the meshes in
\texttt{Gmsh} \cite{geuzaine2009gmsh},
where we allow for a polynomial approximation of the boundary leading to the computational domain $\Omega_h$.
For the assembly of the mass and stiffness matrices $\massmatrix$ and $\stiffmatrix$, we use the \texttt{DOLFINx} environment \cite{DOLFINx23}, by setting
\begin{equation}
	(\massmatrix)_{i,j} = \int_{\Omega_h} \varphi_i \varphi_j \, dx
	+
	\int_{\Gamma_h} \varphi_i \varphi_j \, ds,
	\quad
	(\stiffmatrix)_{i,j} = \int_{\Omega_h} \nabla \varphi_i \cdot  \nabla \varphi_j \, dx 
	+
	 \int_{\Gamma_h} \nabla_{\Gamma_h} \varphi_i \cdot  \nabla_{\Gamma_h} \varphi_j \, ds .
\end{equation}
These objects are stored using the \texttt{PETSc for Python} library \cite{petsc_2,petsc_1}, which allows for many options in the solution of the upcoming linear systems. 
For the assembly of the Dirichlet problem, the boundary integrals are omitted and corresponding boundary degrees of freedom are eliminated.
We denote by $(\solutionvec^{n},\solutiondtvec^n)$
the coefficient vectors of $( u_h^n, v_h^n ) \approx ( u(t_n), u'(t_n) )$,
and note that the exponential integrator in \eqref{eq:fully_discrete_db} is equivalent to the exact solution $(\solutionvec,\solutionvec')$  at time $t_{n+1} = t_n + \tau$ of
\begin{equation} \label{eq:sys_for_rat_kry}
 \massmatrix \solutionvec''(t) = - \stiffmatrix \solutionvec(t)  + \massmatrix \loadvectorn, 
	\quad 
	t \in [t_n, t_n +\tau],
	\qquad
	\solutionvec(t_n) = \solutionvec^n,  
	\quad 
	\solutionvec'(t_n) = \solutiondtvec^n, 
\end{equation}
%
where $\loadvectorn$ denotes the vector corresponding to the function $I_h f(\L_h u_h^n)$ from \eqref{eq:def_fh},
and the new approximations are defined as $\solutionvec^{n+1} = \solutionvec(t_{n+1})$
and 
$\solutiondtvec^{n+1} = \solutionvec'(t_{n+1})$.
The solution of \eqref{eq:sys_for_rat_kry} is computed using a rational Krylov approximation as suggested in 
\cite{GriH08} and \cite{HocPSTW15}. 
For writing and reading of the meshes and functions
we use \texttt{ADIOS4DOLFINx}
\cite{ADIOS4DOLFINx2024}.

	For the error computation we use the
	$\massmatrix$-inner product for the $L^2(\Omega_h;\Gamma_h)$-
	and	$L^2(\Omega_h)$-norm and the $(\stiffmatrix+\massmatrix)^{-1}$-inner product for the $H^{-1}(\Omega_h;\Gamma_h)$ and $H^{-1}(\Omega_h)$-norm. By Lemma~\ref{lem:basic_results}(b), for \(c_G=1\), the latter norms
	are uniformly equivalent to the corresponding \(H^{-1}\)-norms.

The code corresponding to the experiments in this section is made publicly available at
\begin{equation*}
	\text{\url{https://github.com/BenjaminDoerich/ExpoFem-LowRegWave}}.
\end{equation*}

\subsection{Dynamical and Dirichlet boundary conditions}

\begin{figure}
	\includegraphics[scale=0.3]{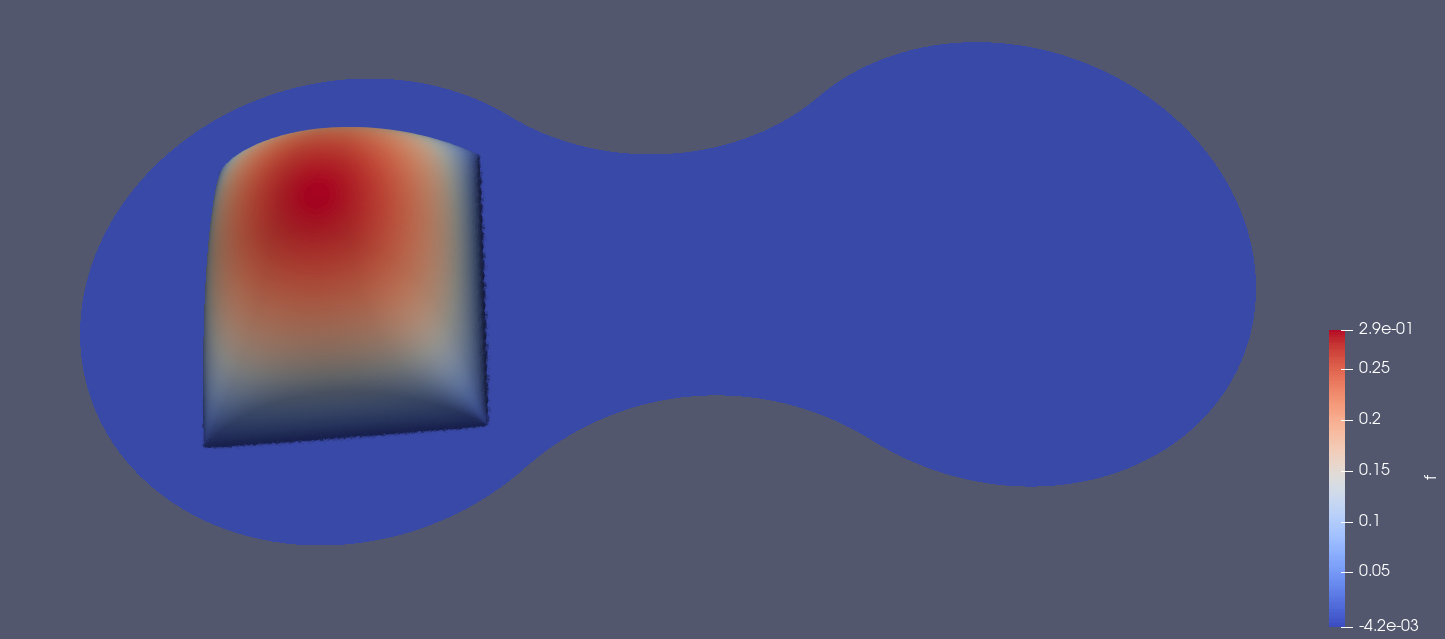}
	\caption{Initial datum defined in \eqref{eq:init_bump} 
	computed with the $L^2$-projection and cubic elements.}
	\label{fig:bump_inits}
\end{figure}

For both types of boundary conditions, we are on the time interval $[0,0.8]$ and we use the same initial conditions $(u^0,v^0)$. The initial displacement, shown in Figure~\ref{fig:bump_inits}, is the compactly supported bump
\begin{equation} \label{eq:init_bump}
u^{0}(x,y)=\left\{ 
	\begin{array}{ll}
		{\displaystyle H \sqrt{r_x^2-(x-x_0)^2} \sqrt{r_y^2 - (y - y_0)^2}},\quad & {\displaystyle (x,y) \in Q },\\[1mm]
		{\displaystyle 0},\quad &\text{else, }
	\end{array}
	\right. 
\end{equation}
with box $Q = ( x_0 - r_x , x_0 + r_x   ) \times ( y_0 - r_y , y_0 + r_y   ) $ and parameters $H \approx 2.9$, $(x_0,y_0) = (-0.5,0)$ and $(r_x,r_y) \approx (0.2, 0.23)$. The square-root behaviour of \(u^0\) along the sides of \(Q\) implies that \(u^0\in H^{1-\varepsilon}(\Omega)\) for every \(\varepsilon>0\), while \(u^0\notin H^1(\Omega)\). Since \(Q\Subset\Omega\), no additional boundary singularity is introduced. Thus this is a borderline rough displacement, slightly below but arbitrarily close to the finite-energy space. Although this datum is slightly outside the assumptions of Theorem~\ref{thm:convergence_db}, it provides a meaningful borderline test of the predicted low-regularity convergence rates. 
The initial velocity is given by
$v^0 (x,y) = \frac{1}{10}(x^2-1)$ and as nonlinearities we choose
$f_\Omega (u) = u^2$ and $f_\Gamma (u) = u^3$.
If not stated otherwise, we employ the scaling of the spatial and temporal parameters according to
\begin{equation} \label{eq:relation_h_tau}
h \approx \frac{1}{T} \taufac  \tau^{\frac{\ell_k+1}{\ell_k}}
\end{equation} 
with $\ell_k$ defined in \eqref{eq:def_lk}
to obtain a similar error contribution in Theorem~\ref{thm:convergence_db}, and restrict ourselves to polynomials of order $k=1,2,3$.
We first compute reference solutions with some small step size $\tauref$ and fine mesh parameter $\hspaceref$, see Table~\ref{tab:exp1_data}, and note that $\taufac$ is the same for the reference solution as well as for the other approximations. 
\begin{table}[h!]
\centering
\begin{tabular}{||c c c c c||} 
 \hline
 $k$ & $\taufac$ & $\tauref$ & $\hspaceref$ &  $\text{dofs}_{\text{ref}}$ \\ [0.5ex] 
 \hline\hline
 1 & 3.75 & $\sim  3\cdot 10^{-3}$ & $\sim 1\cdot 10^{-3}$ &  778,098 \\
 \hline
 2 & 2.5 & $\sim  3\cdot 10^{-3}$  & $\sim 2\cdot 10^{-3}$  & 928,393 \\
 \hline
 3 & 2.5 & $\sim  3\cdot 10^{-3}$ & $\sim 3\cdot 10^{-3}$ & 1,284,916 \\
 \hline
 \end{tabular}
\caption{Choice of the parameters for the experiment corresponding to the initial datum defined in \eqref{eq:init_bump}.}
 \label{tab:exp1_data}
\end{table}

The error $E(T)$ is computed in the
$L^2(\Omega_h;\Gamma_h) \times H^{-1}(\Omega_h;\Gamma_h)$-norm
at the final time $T = 0.8$ and scaled by the $L^2(\Omega_h;\Gamma_h) \times H^{-1}(\Omega_h;\Gamma_h)$-norm of the reference solution.
For the Dirichlet case $L^2(\Omega_h;\Gamma_h)$ is replaced by $L^2(\Omega_h)$, and $H^{-1}(\Omega_h;\Gamma_h)$ by $H^{-1}(\Omega_h)$.
  A plot of the reference solution at the final time is shown in Figure~\ref{fig:bump_sols}.
\begin{figure}
	\begin{subfigure}{\textwidth}
	\centering
		\includegraphics[scale=0.25]{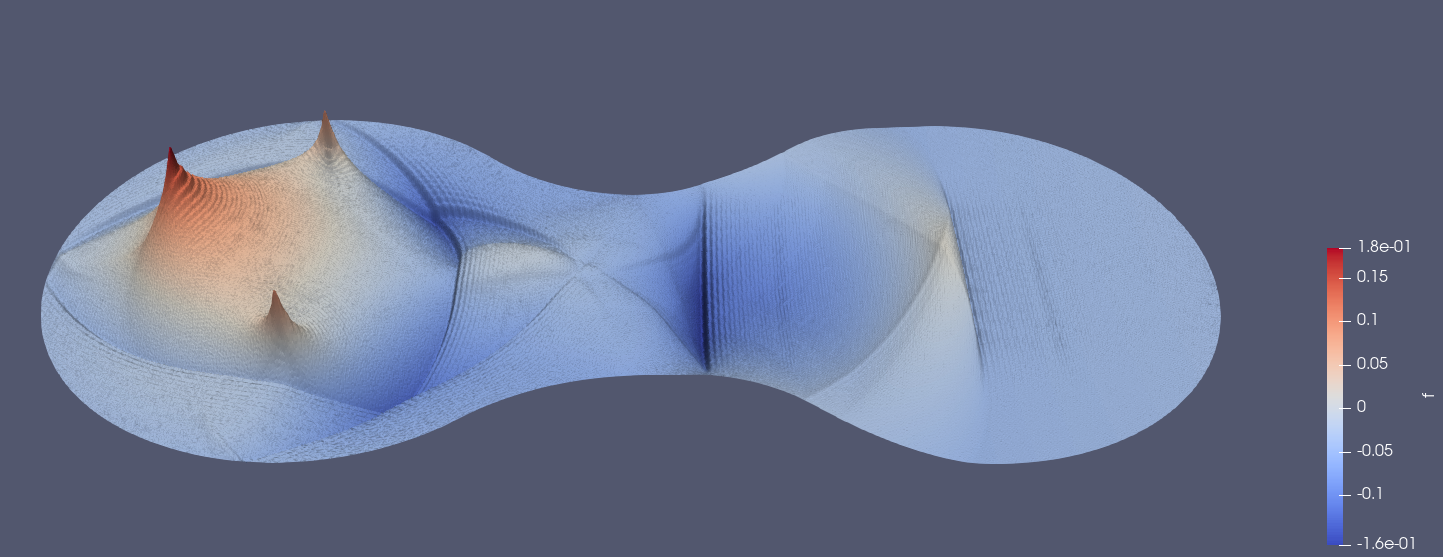}
\end{subfigure}%
	\par\bigskip
	\begin{subfigure}{\textwidth}
	\centering
	\includegraphics[scale=0.25]{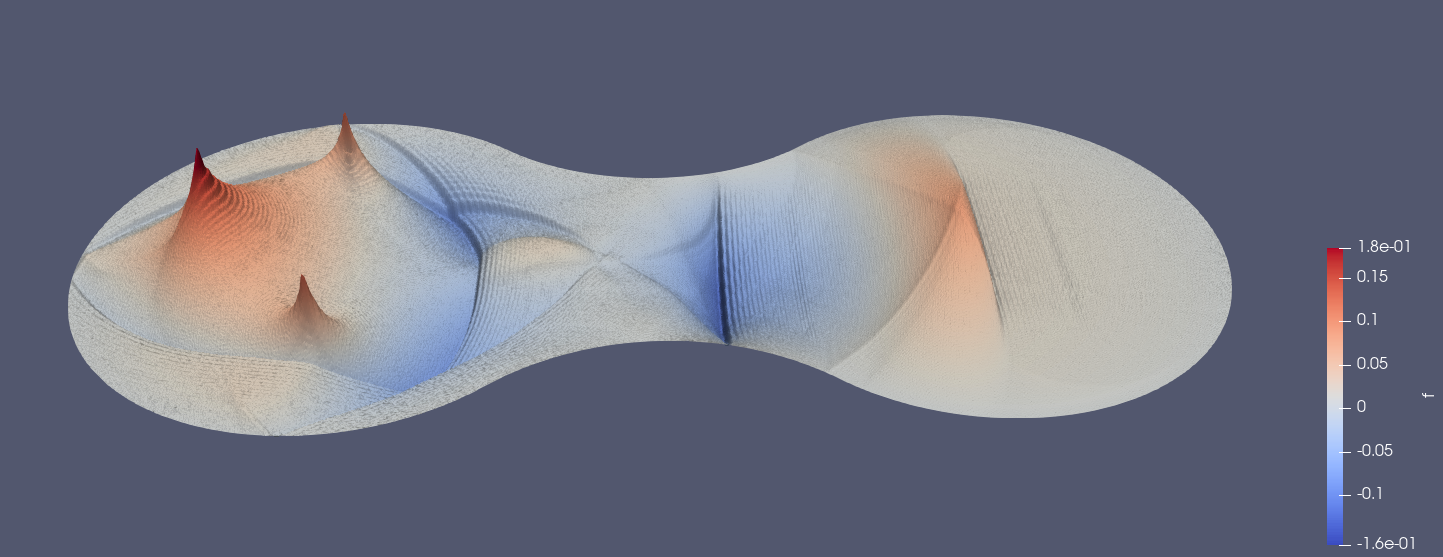}
	\end{subfigure}%
	\caption{Approximation by the reference solution computed for 
	initial datum defined in \eqref{eq:init_bump}	
	with cubic elements and the data from Table~\ref{tab:exp1_data}
	at time $T=0.8$ for 
	dynamical boundary conditions (top) 
	and
	Dirichlet boundary conditions (bottom). }
	\label{fig:bump_sols}
\end{figure}

For both the dynamical boundary and the Dirichlet case, we performed the same simulations, i.e. we used the same meshes and time step sizes. In the upper left plot of 
Figure~\ref{fig:exp_dyn_2x2grid} (for the dynamical boundary case) and Figure~\ref{fig:exp_Dir_2x2grid}
 (for the Dirichlet case), we computed the error under the scaling \eqref{eq:relation_h_tau} and included reference lines that show the expected order $\tau$. In the upper right plot, we further plotted the error over the degrees of freedom in order to show also the computational advantage of higher order elements.
 In the second part of the experiments, we study the convergence with respect to $h$ by choosing some fixed $\tau_0 = 4 \tauref$ and decreasing values of $h$. 
We observe in the lower left plot that the spatial error aligns with the expected order $h^{\frac{\ell_k}{\ell_k+1}}$.
The plots on the lower right, again show the advantage of the higher order elements.

\subsection*{Acknowledgments}

We
would like to thank Tim Buchholz for his valuable input on the \texttt{DOLFINx} 
and \texttt{PETSc} implementation of the finite element solver.

\begin{figure}
\textbf{Case of dynamic boundary conditions}\\[0.5em]
  \centering
  \begin{subfigure}{0.48\textwidth}
    \centering
	\begingroup%
	\StrCount{num_exp/peanut_bump_dyn_100/peanut_bump_dyn_100_joint_tau}{/}[\matches]%
	\StrBefore[\matches]{num_exp/peanut_bump_dyn_100/peanut_bump_dyn_100_joint_tau}{/}[\datapath]%
	\includestandalone[]{num_exp/peanut_bump_dyn_100/peanut_bump_dyn_100_joint_tau}%
	\endgroup%

    \label{fig:exp_dyn_joint-tau}
  \end{subfigure}%
  \hfill
  \begin{subfigure}{0.48\textwidth}
    \centering
	\begingroup%
	\StrCount{num_exp/peanut_bump_dyn_100/peanut_bump_dyn_100_joint_dofs}{/}[\matches]%
	\StrBefore[\matches]{num_exp/peanut_bump_dyn_100/peanut_bump_dyn_100_joint_dofs}{/}[\datapath]%
\providecommand{\logLogSlopeTriangle}[5]
{

    \pgfplotsextra
    {
        \pgfkeysgetvalue{/pgfplots/xmin}{\xmin}
        \pgfkeysgetvalue{/pgfplots/xmax}{\xmax}
        \pgfkeysgetvalue{/pgfplots/ymin}{\ymin}
        \pgfkeysgetvalue{/pgfplots/ymax}{\ymax}

        \pgfmathsetmacro{\xArel}{#1}
        \pgfmathsetmacro{\yArel}{#3}
        \pgfmathsetmacro{\xBrel}{#1-#2}
        \pgfmathsetmacro{\yBrel}{\yArel}
        \pgfmathsetmacro{\xCrel}{\xArel}

        \pgfmathsetmacro{\lnxB}{\xmin*(1-(#1-#2))+\xmax*(#1-#2)} 
        \pgfmathsetmacro{\lnxA}{\xmin*(1-#1)+\xmax*#1} 
        \pgfmathsetmacro{\lnyA}{\ymin*(1-#3)+\ymax*#3} 
        \pgfmathsetmacro{\lnyC}{\lnyA+#4*(\lnxA-\lnxB)}
        \pgfmathsetmacro{\yCrel}{\lnyC-\ymin)/(\ymax-\ymin)} 

        \coordinate (A) at (rel axis cs:\xArel,\yArel);
        \coordinate (B) at (rel axis cs:\xBrel,\yBrel);
        \coordinate (C) at (rel axis cs:\xCrel,\yCrel);

        \draw[#5]   (A)-- node[pos=0.5,anchor=north] {1}
                    (B)-- 
                    (C)-- node[pos=0.5,anchor=west] {#4}
                    cycle;
    }
}

    \begin{tikzpicture}
      \begin{axis}[
          width=0.925\linewidth,
          height=0.9\linewidth,
          grid=major,
          grid style={dashed,gray!30},
          xlabel=dofs,
          xmode=log,
          ymode=log,
          ymax=1.1,
          legend style={at={(0.7,0.21)},anchor=north}, 
          legend columns = 2
          x tick label style={rotate=0,anchor=north}, 
          legend columns=-1,
          legend to name=Legendforall
        ]
        \addplot+[mark=*, color=blue]
        table[x=dofs_1,y=error_tot_1,col sep=comma] 
        {\datapath/data/ref_sol_test_num_17dyn_shift_dyn_h_0.002967359050445104_tau_0.003125_all_pol_deg_joint_conv_relative_time.txt};
        \addlegendentry{\(k=1\)}

        
        \addplot+[mark=*, color=red]
        table[x=dofs_2,y=error_tot_2,col sep=comma] 
        {\datapath/data/ref_sol_test_num_17dyn_shift_dyn_h_0.002967359050445104_tau_0.003125_all_pol_deg_joint_conv_relative_time.txt};
        \addlegendentry{\(k=2\)}

        
        \addplot+[mark=*, color=cyan,  mark options={draw=cyan,fill=cyan}]
		table[x=dofs_3,y=error_tot_3,col sep=comma] 
		{\datapath/data/ref_sol_test_num_17dyn_shift_dyn_h_0.002967359050445104_tau_0.003125_all_pol_deg_joint_conv_relative_time.txt};
		\addlegendentry{\(k=3\)}




      \end{axis}
    \end{tikzpicture}
%
	\endgroup%

    \label{fig:exp_dyn_joint-dofs}
  \end{subfigure}
  \par\bigskip
  \begin{subfigure}{0.48\textwidth}
    \centering
	\begingroup%
	\StrCount{num_exp/peanut_bump_dyn_100/peanut_bump_dyn_100_tau}{/}[\matches]%
	\StrBefore[\matches]{num_exp/peanut_bump_dyn_100/peanut_bump_dyn_100_tau}{/}[\datapath]%
		\providecommand{\logLogSlopeTriangle}[5]
{
	
	\pgfplotsextra
	{
		\pgfkeysgetvalue{/pgfplots/xmin}{\xmin}
		\pgfkeysgetvalue{/pgfplots/xmax}{\xmax}
		\pgfkeysgetvalue{/pgfplots/ymin}{\ymin}
		\pgfkeysgetvalue{/pgfplots/ymax}{\ymax}
		
		\pgfmathsetmacro{\xArel}{#1}
		\pgfmathsetmacro{\yArel}{#3}
		\pgfmathsetmacro{\xBrel}{#1-#2}
		\pgfmathsetmacro{\yBrel}{\yArel}
		\pgfmathsetmacro{\xCrel}{\xArel}
		
		\pgfmathsetmacro{\lnxB}{\xmin*(1-(#1-#2))+\xmax*(#1-#2)} 
		\pgfmathsetmacro{\lnxA}{\xmin*(1-#1)+\xmax*#1} 
		\pgfmathsetmacro{\lnyA}{\ymin*(1-#3)+\ymax*#3} 
		\pgfmathsetmacro{\lnyC}{\lnyA+#4*(\lnxA-\lnxB)}
		\pgfmathsetmacro{\yCrel}{\lnyC-\ymin)/(\ymax-\ymin)} 
		
		\coordinate (A) at (rel axis cs:\xArel,\yArel);
		\coordinate (B) at (rel axis cs:\xBrel,\yBrel);
		\coordinate (C) at (rel axis cs:\xCrel,\yCrel);
		
	
	\draw[#5]   (A)-- node[midway, auto, font=\scriptsize] {1}
	(B)-- 
	(C)-- node[midway, auto, font=\scriptsize] {#4}
	cycle;
}
}

\providecommand{\logLogSlopeTriangleFlip}[5]
{

\pgfplotsextra
{
	\pgfkeysgetvalue{/pgfplots/xmin}{\xmin}
	\pgfkeysgetvalue{/pgfplots/xmax}{\xmax}
	\pgfkeysgetvalue{/pgfplots/ymin}{\ymin}
	\pgfkeysgetvalue{/pgfplots/ymax}{\ymax}
	
	\pgfmathsetmacro{\xArel}{#1}
	\pgfmathsetmacro{\yArel}{#3}
	\pgfmathsetmacro{\xBrel}{#1-#2}
	\pgfmathsetmacro{\yBrel}{\yArel}
	\pgfmathsetmacro{\xCrel}{\xArel}
	
	\pgfmathsetmacro{\lnxB}{\xmin*(1-(#1-#2))+\xmax*(#1-#2)} 
	\pgfmathsetmacro{\lnxA}{\xmin*(1-#1)+\xmax*#1} 
	\pgfmathsetmacro{\lnyA}{\ymin*(1-#3)+\ymax*#3} 
	\pgfmathsetmacro{\lnyC}{\lnyA+#4*(\lnxA-\lnxB)}
	\pgfmathsetmacro{\yCrel}{\lnyC-\ymin)/(\ymax-\ymin)} 
	
	\coordinate (A) at (rel axis cs:\xBrel,\yCrel); 
	\coordinate (B) at (rel axis cs:\xBrel,\yBrel); 
	\coordinate (C) at (rel axis cs:\xCrel,\yCrel); 
	
	\draw[#5]
	(A) -- node[midway, auto, swap, font=\scriptsize] {$1$} (B)
	-- (C)
	-- node[midway, auto , swap,font=\scriptsize] {#4}
	cycle;
}
}

    \begin{tikzpicture}
      \begin{axis}[
          width=0.925\linewidth,
          height=0.9\linewidth,
          grid=major,
          grid style={dashed,gray!30},
          xlabel=$h$,
          ylabel={$E(T = 0.8)$},
          ylabel style={yshift= 0.5em},    
          xmode=log,
          ymode=log,
          ymax=1.1,
          legend style={at={(0.7,0.21)},anchor=north}, 
          legend columns = 2
          x tick label style={rotate=0,anchor=north}, 
          legend columns=-1,
          legend to name=Legendforall
        ]
        \addplot+[mark=*, color=blue]
        table[x=has_1,y=error_tot_1,col sep=comma] 
        {\datapath/data/ref_sol_test_num_17dyn_shift_dyn_h_0.002967359050445104_tau_0.003125_all_pol_deg_h_conv_relative_time.txt};
        \addlegendentry{\(k=1\)}

        
        \addplot+[mark=*, color=red]
        table[x=has_2,y=error_tot_2,col sep=comma] 
        {\datapath/data/ref_sol_test_num_17dyn_shift_dyn_h_0.002967359050445104_tau_0.003125_all_pol_deg_h_conv_relative_time.txt};
        \addlegendentry{\(k=2\)}

        
        \addplot+[mark=*, color=cyan,  mark options={draw=cyan,fill=cyan}]
		table[x=has_3,y=error_tot_3,col sep=comma] 
		{\datapath/data/ref_sol_test_num_17dyn_shift_dyn_h_0.002967359050445104_tau_0.003125_all_pol_deg_h_conv_relative_time.txt};
		\addlegendentry{\(k=3\)}

        \addplot+[mark=None, dashed, color=blue]
        table[x=has_1,y expr=\thisrow{line_over_curve_has_1}*0.9,col sep=comma] 
        {\datapath/data/ref_sol_test_num_17dyn_shift_dyn_h_0.002967359050445104_tau_0.003125_all_pol_deg_h_conv_relative_time.txt};
   		\logLogSlopeTriangleFlip{0.28}{0.13}{0.34}{0.67}{blue, thick};
                
		\addplot+[mark=None, dashed, color=red]
		table[x=has_2,y expr=\thisrow{line_over_curve_has_2}*0.9,col sep=comma] 
		{\datapath/data/ref_sol_test_num_17dyn_shift_dyn_h_0.002967359050445104_tau_0.003125_all_pol_deg_h_conv_relative_time.txt};
		\logLogSlopeTriangleFlip{0.28}{0.13}{0.17}{0.8}{red, thick};

 	    \addplot+[mark=None, dashed, color=cyan]
		table[x=has_3,y expr=\thisrow{line_over_curve_has_3}*0.9,col sep=comma] 
		{\datapath/data/ref_sol_test_num_17dyn_shift_dyn_h_0.002967359050445104_tau_0.003125_all_pol_deg_h_conv_relative_time.txt};
		\logLogSlopeTriangle{0.4}{0.13}{0.11}{0.83}{cyan, thick};

  \end{axis}
    \end{tikzpicture}
%
	\endgroup%

    \label{fig:exp_dyn_onlyh-tau}
  \end{subfigure}%
  \hfill
  \begin{subfigure}{0.48\textwidth}
    \centering
	\begingroup%
	\StrCount{num_exp/peanut_bump_dyn_100/peanut_bump_dyn_100_dofs}{/}[\matches]%
	\StrBefore[\matches]{num_exp/peanut_bump_dyn_100/peanut_bump_dyn_100_dofs}{/}[\datapath]%
\providecommand{\logLogSlopeTriangle}[5]
{

    \pgfplotsextra
    {
        \pgfkeysgetvalue{/pgfplots/xmin}{\xmin}
        \pgfkeysgetvalue{/pgfplots/xmax}{\xmax}
        \pgfkeysgetvalue{/pgfplots/ymin}{\ymin}
        \pgfkeysgetvalue{/pgfplots/ymax}{\ymax}

        \pgfmathsetmacro{\xArel}{#1}
        \pgfmathsetmacro{\yArel}{#3}
        \pgfmathsetmacro{\xBrel}{#1-#2}
        \pgfmathsetmacro{\yBrel}{\yArel}
        \pgfmathsetmacro{\xCrel}{\xArel}

        \pgfmathsetmacro{\lnxB}{\xmin*(1-(#1-#2))+\xmax*(#1-#2)} 
        \pgfmathsetmacro{\lnxA}{\xmin*(1-#1)+\xmax*#1} 
        \pgfmathsetmacro{\lnyA}{\ymin*(1-#3)+\ymax*#3} 
        \pgfmathsetmacro{\lnyC}{\lnyA+#4*(\lnxA-\lnxB)}
        \pgfmathsetmacro{\yCrel}{\lnyC-\ymin)/(\ymax-\ymin)} 

        \coordinate (A) at (rel axis cs:\xArel,\yArel);
        \coordinate (B) at (rel axis cs:\xBrel,\yBrel);
        \coordinate (C) at (rel axis cs:\xCrel,\yCrel);

        \draw[#5]   (A)-- node[pos=0.5,anchor=north] {1}
                    (B)-- 
                    (C)-- node[pos=0.5,anchor=west] {#4}
                    cycle;
    }
}

    \begin{tikzpicture}
      \begin{axis}[
          width=0.925\linewidth,
          height=0.9\linewidth,      
          grid=major,
          grid style={dashed,gray!30},
          xlabel=dofs,
          xmode=log,
          ymode=log,
          ymax=1.1,
          legend style={at={(0.7,0.21)},anchor=north}, 
          legend columns = 2
          x tick label style={rotate=0,anchor=north}, 
          legend columns=-1,
          legend to name=Legendforall
        ]
       \addplot+[mark=*, color=blue]
        table[x=dofs_1,y=error_tot_1,col sep=comma] 
        {\datapath/data/ref_sol_test_num_17dyn_shift_dyn_h_0.002967359050445104_tau_0.003125_all_pol_deg_h_conv_relative_time.txt};
        \addlegendentry{\(k=1\)}


       \addplot+[mark=*, color=red]
        table[x=dofs_2,y=error_tot_2,col sep=comma] 
        {\datapath/data/ref_sol_test_num_17dyn_shift_dyn_h_0.002967359050445104_tau_0.003125_all_pol_deg_h_conv_relative_time.txt};
        \addlegendentry{\(k=2\)}

        
        \addplot+[mark=*, color=cyan,  mark options={draw=cyan,fill=cyan}]
		table[x=dofs_3,y=error_tot_3,col sep=comma] 
		{\datapath/data/ref_sol_test_num_17dyn_shift_dyn_h_0.002967359050445104_tau_0.003125_all_pol_deg_h_conv_relative_time.txt};
		\addlegendentry{\(k=3\)}


		

      \end{axis}
    \end{tikzpicture}
%
	\endgroup%

    \label{fig:exp_dyn_onlyh-dofs}
  \end{subfigure}
  \par\bigskip
  \ref{Legendforall}
  \caption{Error plots for the initial data \eqref{eq:init_bump}. Top row: Equilibrated error $E(T=0.8)$ via \eqref{eq:relation_h_tau} for $k=1,2,3$. 
 In the left plot, the dashed lines indicate order~$\tau$.
  Bottom row: Convergence in $h$ for fixed $\tau$. The dashed lines indicate the expected order of convergence $h^{\frac{\ell_k}{\ell_k+1}}$ from Theorem~\ref{thm:convergence_db}.}
  \label{fig:exp_dyn_2x2grid}
\end{figure}
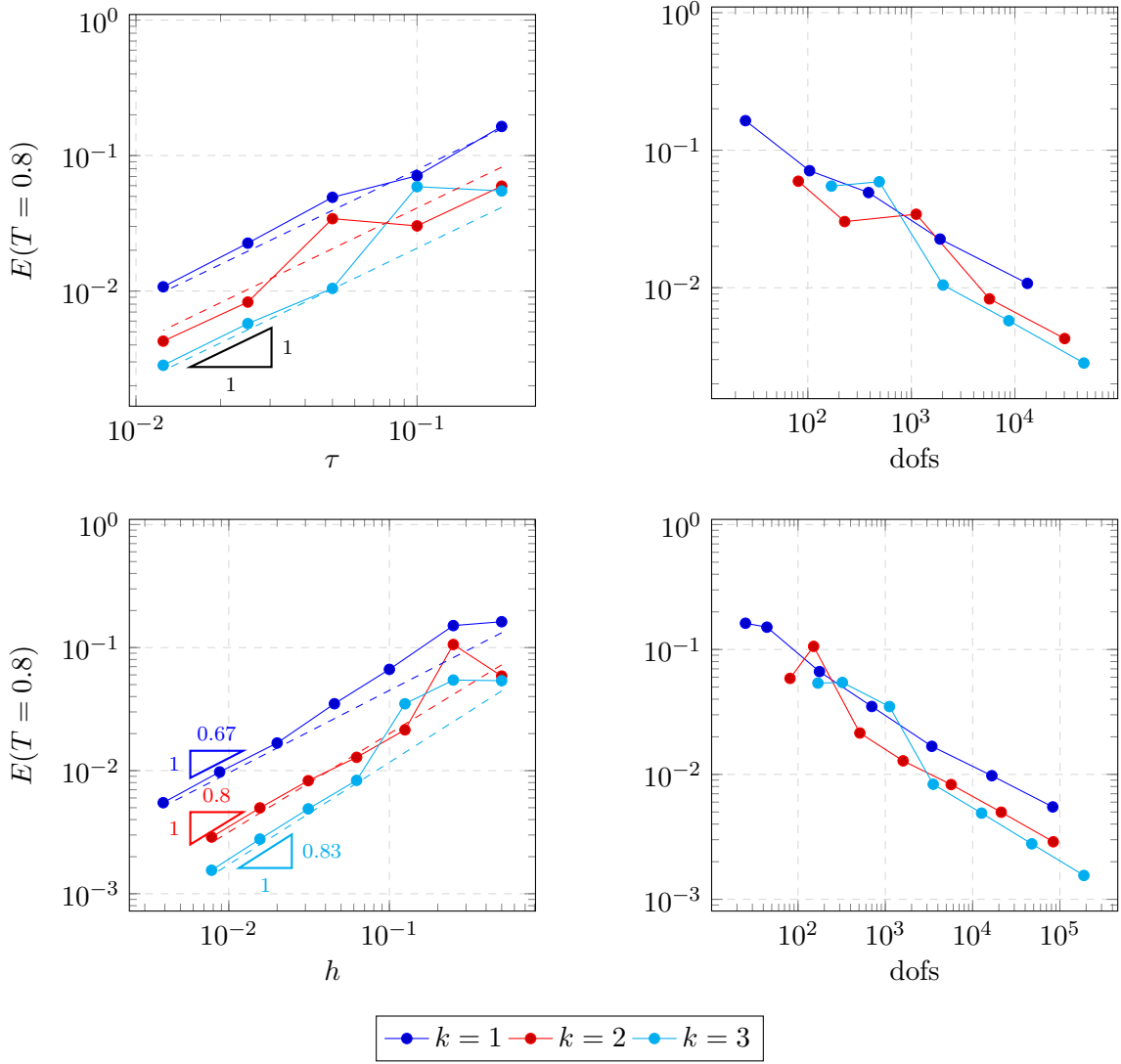

\begin{figure}
\textbf{Case of Dirichlet boundary conditions}\\[0.5em]
  \centering
  \begin{subfigure}{0.48\textwidth}
    \centering
	\begingroup%
	\StrCount{num_exp/peanut_bump_Dir_100/peanut_bump_Dir_100_joint_tau}{/}[\matches]%
	\StrBefore[\matches]{num_exp/peanut_bump_Dir_100/peanut_bump_Dir_100_joint_tau}{/}[\datapath]%
	\includestandalone[]{num_exp/peanut_bump_Dir_100/peanut_bump_Dir_100_joint_tau}%
	\endgroup%

    \label{fig:exp_Dir_joint-tau}
  \end{subfigure}%
  \hfill
  \begin{subfigure}{0.48\textwidth}
    \centering
	\begingroup%
	\StrCount{num_exp/peanut_bump_Dir_100/peanut_bump_Dir_100_joint_dofs}{/}[\matches]%
	\StrBefore[\matches]{num_exp/peanut_bump_Dir_100/peanut_bump_Dir_100_joint_dofs}{/}[\datapath]%
\providecommand{\logLogSlopeTriangle}[5]
{

    \pgfplotsextra
    {
        \pgfkeysgetvalue{/pgfplots/xmin}{\xmin}
        \pgfkeysgetvalue{/pgfplots/xmax}{\xmax}
        \pgfkeysgetvalue{/pgfplots/ymin}{\ymin}
        \pgfkeysgetvalue{/pgfplots/ymax}{\ymax}

        \pgfmathsetmacro{\xArel}{#1}
        \pgfmathsetmacro{\yArel}{#3}
        \pgfmathsetmacro{\xBrel}{#1-#2}
        \pgfmathsetmacro{\yBrel}{\yArel}
        \pgfmathsetmacro{\xCrel}{\xArel}

        \pgfmathsetmacro{\lnxB}{\xmin*(1-(#1-#2))+\xmax*(#1-#2)} 
        \pgfmathsetmacro{\lnxA}{\xmin*(1-#1)+\xmax*#1} 
        \pgfmathsetmacro{\lnyA}{\ymin*(1-#3)+\ymax*#3} 
        \pgfmathsetmacro{\lnyC}{\lnyA+#4*(\lnxA-\lnxB)}
        \pgfmathsetmacro{\yCrel}{\lnyC-\ymin)/(\ymax-\ymin)} 

        \coordinate (A) at (rel axis cs:\xArel,\yArel);
        \coordinate (B) at (rel axis cs:\xBrel,\yBrel);
        \coordinate (C) at (rel axis cs:\xCrel,\yCrel);

        \draw[#5]   (A)-- node[pos=0.5,anchor=north] {1}
                    (B)-- 
                    (C)-- node[pos=0.5,anchor=west] {#4}
                    cycle;
    }
}

    \begin{tikzpicture}
      \begin{axis}[
          width=0.925\linewidth,
          height=0.9\linewidth,
          grid=major,
          grid style={dashed,gray!30},
          xlabel=dofs,
          xmode=log,
          ymode=log,
          ymax=1.1,
          legend style={at={(0.7,0.21)},anchor=north}, 
          legend columns = 2
          x tick label style={rotate=0,anchor=north}, 
          legend columns=-1,
          legend to name=Legendforall
        ]
        
                \addplot+[mark=*, color=blue]
        table[x=dofs_1,y=error_tot_1,col sep=comma] 
        {\datapath/data/ref_sol_test_num_17dyn_shift_Dir_h_0.002967359050445104_tau_0.003125_all_pol_deg_joint_conv_relative_time.txt};
        \addlegendentry{\(k=1\)}

        
        \addplot+[mark=*, color=red]
        table[x=dofs_2,y=error_tot_2,col sep=comma] 
        {\datapath/data/ref_sol_test_num_17dyn_shift_Dir_h_0.002967359050445104_tau_0.003125_all_pol_deg_joint_conv_relative_time.txt};
        \addlegendentry{\(k=2\)}

        
        \addplot+[mark=*, color=cyan,  mark options={draw=cyan,fill=cyan}]
        table[x=dofs_3,y=error_tot_3,col sep=comma] 
        {\datapath/data/ref_sol_test_num_17dyn_shift_Dir_h_0.002967359050445104_tau_0.003125_all_pol_deg_joint_conv_relative_time.txt};
        \addlegendentry{\(k=3\)}





      \end{axis}
    \end{tikzpicture}
%
	\endgroup%

    \label{fig:exp_Dir_joint-dofs}
  \end{subfigure}
  \par\bigskip
  \begin{subfigure}{0.48\textwidth}
    \centering
	\begingroup%
	\StrCount{num_exp/peanut_bump_Dir_100/peanut_bump_Dir_100_tau}{/}[\matches]%
	\StrBefore[\matches]{num_exp/peanut_bump_Dir_100/peanut_bump_Dir_100_tau}{/}[\datapath]%
		\providecommand{\logLogSlopeTriangle}[5]
{
	
	\pgfplotsextra
	{
		\pgfkeysgetvalue{/pgfplots/xmin}{\xmin}
		\pgfkeysgetvalue{/pgfplots/xmax}{\xmax}
		\pgfkeysgetvalue{/pgfplots/ymin}{\ymin}
		\pgfkeysgetvalue{/pgfplots/ymax}{\ymax}
		
		\pgfmathsetmacro{\xArel}{#1}
		\pgfmathsetmacro{\yArel}{#3}
		\pgfmathsetmacro{\xBrel}{#1-#2}
		\pgfmathsetmacro{\yBrel}{\yArel}
		\pgfmathsetmacro{\xCrel}{\xArel}
		
		\pgfmathsetmacro{\lnxB}{\xmin*(1-(#1-#2))+\xmax*(#1-#2)} 
		\pgfmathsetmacro{\lnxA}{\xmin*(1-#1)+\xmax*#1} 
		\pgfmathsetmacro{\lnyA}{\ymin*(1-#3)+\ymax*#3} 
		\pgfmathsetmacro{\lnyC}{\lnyA+#4*(\lnxA-\lnxB)}
		\pgfmathsetmacro{\yCrel}{\lnyC-\ymin)/(\ymax-\ymin)} 
		
		\coordinate (A) at (rel axis cs:\xArel,\yArel);
		\coordinate (B) at (rel axis cs:\xBrel,\yBrel);
		\coordinate (C) at (rel axis cs:\xCrel,\yCrel);
		
	
	\draw[#5]   (A)-- node[midway, auto, font=\scriptsize] {1}
	(B)-- 
	(C)-- node[midway, auto, font=\scriptsize] {#4}
	cycle;
}
}

\providecommand{\logLogSlopeTriangleFlip}[5]
{

\pgfplotsextra
{
	\pgfkeysgetvalue{/pgfplots/xmin}{\xmin}
	\pgfkeysgetvalue{/pgfplots/xmax}{\xmax}
	\pgfkeysgetvalue{/pgfplots/ymin}{\ymin}
	\pgfkeysgetvalue{/pgfplots/ymax}{\ymax}
	
	\pgfmathsetmacro{\xArel}{#1}
	\pgfmathsetmacro{\yArel}{#3}
	\pgfmathsetmacro{\xBrel}{#1-#2}
	\pgfmathsetmacro{\yBrel}{\yArel}
	\pgfmathsetmacro{\xCrel}{\xArel}
	
	\pgfmathsetmacro{\lnxB}{\xmin*(1-(#1-#2))+\xmax*(#1-#2)} 
	\pgfmathsetmacro{\lnxA}{\xmin*(1-#1)+\xmax*#1} 
	\pgfmathsetmacro{\lnyA}{\ymin*(1-#3)+\ymax*#3} 
	\pgfmathsetmacro{\lnyC}{\lnyA+#4*(\lnxA-\lnxB)}
	\pgfmathsetmacro{\yCrel}{\lnyC-\ymin)/(\ymax-\ymin)} 
	
	\coordinate (A) at (rel axis cs:\xBrel,\yCrel); 
	\coordinate (B) at (rel axis cs:\xBrel,\yBrel); 
	\coordinate (C) at (rel axis cs:\xCrel,\yCrel); 
	
	\draw[#5]
	(A) -- node[midway, auto, swap, font=\scriptsize] {$1$} (B)
	-- (C)
	-- node[midway, auto , swap,font=\scriptsize] {#4}
	cycle;
}
}

    \begin{tikzpicture}
      \begin{axis}[
          width=0.925\linewidth,
          height=0.9\linewidth,
          grid=major,
          grid style={dashed,gray!30},
          xlabel=$h$,
          ylabel={$E(T = 0.8)$},
          ylabel style={yshift= 0.5em},    
          xmode=log,
          ymode=log,
          ymax=1.1,
          legend style={at={(0.7,0.21)},anchor=north}, 
          legend columns = 2
          x tick label style={rotate=0,anchor=north}, 
          legend columns=-1,
          legend to name=Legendforall
        ]
        \addplot+[mark=*, color=blue]
        table[x=has_1,y=error_tot_1,col sep=comma] 
        {\datapath/data/ref_sol_test_num_17dyn_shift_Dir_h_0.002967359050445104_tau_0.003125_all_pol_deg_h_conv_relative_time.txt};
        \addlegendentry{\(k=1\)}

        
       \addplot+[mark=*, color=red]
        table[x=has_2,y=error_tot_2,col sep=comma] 
        {\datapath/data/ref_sol_test_num_17dyn_shift_Dir_h_0.002967359050445104_tau_0.003125_all_pol_deg_h_conv_relative_time.txt};
        \addlegendentry{\(k=2\)}

        
        \addplot+[mark=*, color=cyan,  mark options={draw=cyan,fill=cyan}]
		table[x=has_3,y=error_tot_3,col sep=comma] 
		{\datapath/data/ref_sol_test_num_17dyn_shift_Dir_h_0.002967359050445104_tau_0.003125_all_pol_deg_h_conv_relative_time.txt};
		\addlegendentry{\(k=3\)}

       \addplot+[mark=None, dashed, color=blue]
        table[x=has_1,y expr=\thisrow{line_over_curve_has_1}*0.9,col sep=comma] 
        {\datapath/data/ref_sol_test_num_17dyn_shift_Dir_h_0.002967359050445104_tau_0.003125_all_pol_deg_h_conv_relative_time.txt};
       	\logLogSlopeTriangleFlip{0.28}{0.13}{0.41}{0.67}{blue, thick};

		\addplot+[mark=None, dashed, color=red]
		table[x=has_2,y expr=\thisrow{line_over_curve_has_2}*0.9,col sep=comma] 
		{\datapath/data/ref_sol_test_num_17dyn_shift_Dir_h_0.002967359050445104_tau_0.003125_all_pol_deg_h_conv_relative_time.txt};
		\logLogSlopeTriangle{0.28}{0.13}{0.21}{0.8}{red, thick};
				
 	   \addplot+[mark=None, dashed, color=cyan]
		table[x=has_3,y expr=\thisrow{line_over_curve_has_3}*0.9,col sep=comma] 
		{\datapath/data/ref_sol_test_num_17dyn_shift_Dir_h_0.002967359050445104_tau_0.003125_all_pol_deg_h_conv_relative_time.txt};
		\logLogSlopeTriangle{0.38}{0.13}{0.1}{0.83}{cyan, thick};
		
  \end{axis}
    \end{tikzpicture}
%
	\endgroup%

    \label{fig:exp_Dir_onlyh-tau}
  \end{subfigure}%
  \hfill
  \begin{subfigure}{0.48\textwidth}
    \centering
	\begingroup%
	\StrCount{num_exp/peanut_bump_Dir_100/peanut_bump_Dir_100_dofs}{/}[\matches]%
	\StrBefore[\matches]{num_exp/peanut_bump_Dir_100/peanut_bump_Dir_100_dofs}{/}[\datapath]%
\providecommand{\logLogSlopeTriangle}[5]
{

    \pgfplotsextra
    {
        \pgfkeysgetvalue{/pgfplots/xmin}{\xmin}
        \pgfkeysgetvalue{/pgfplots/xmax}{\xmax}
        \pgfkeysgetvalue{/pgfplots/ymin}{\ymin}
        \pgfkeysgetvalue{/pgfplots/ymax}{\ymax}

        \pgfmathsetmacro{\xArel}{#1}
        \pgfmathsetmacro{\yArel}{#3}
        \pgfmathsetmacro{\xBrel}{#1-#2}
        \pgfmathsetmacro{\yBrel}{\yArel}
        \pgfmathsetmacro{\xCrel}{\xArel}

        \pgfmathsetmacro{\lnxB}{\xmin*(1-(#1-#2))+\xmax*(#1-#2)} 
        \pgfmathsetmacro{\lnxA}{\xmin*(1-#1)+\xmax*#1} 
        \pgfmathsetmacro{\lnyA}{\ymin*(1-#3)+\ymax*#3} 
        \pgfmathsetmacro{\lnyC}{\lnyA+#4*(\lnxA-\lnxB)}
        \pgfmathsetmacro{\yCrel}{\lnyC-\ymin)/(\ymax-\ymin)} 

        \coordinate (A) at (rel axis cs:\xArel,\yArel);
        \coordinate (B) at (rel axis cs:\xBrel,\yBrel);
        \coordinate (C) at (rel axis cs:\xCrel,\yCrel);

        \draw[#5]   (A)-- node[pos=0.5,anchor=north] {1}
                    (B)-- 
                    (C)-- node[pos=0.5,anchor=west] {#4}
                    cycle;
    }
}

    \begin{tikzpicture}
      \begin{axis}[
          width=0.925\linewidth,
          height=0.9\linewidth,      
          grid=major,
          grid style={dashed,gray!30},
          xlabel=dofs,
          xmode=log,
          ymode=log,
          ymax=1.1,
          legend style={at={(0.7,0.21)},anchor=north}, 
          legend columns = 2
          x tick label style={rotate=0,anchor=north}, 
          legend columns=-1,
          legend to name=Legendforall
        ]
        \addplot+[mark=*, color=blue]
         table[x=dofs_1,y=error_tot_1,col sep=comma] 
        {\datapath/data/ref_sol_test_num_17dyn_shift_Dir_h_0.002967359050445104_tau_0.003125_all_pol_deg_h_conv_relative_time.txt};
        \addlegendentry{\(k=1\)}

        
        \addplot+[mark=*, color=red]
        table[x=dofs_2,y=error_tot_2,col sep=comma] 
        {\datapath/data/ref_sol_test_num_17dyn_shift_Dir_h_0.002967359050445104_tau_0.003125_all_pol_deg_h_conv_relative_time.txt};
        \addlegendentry{\(k=2\)}

        
        \addplot+[mark=*, color=cyan,  mark options={draw=cyan,fill=cyan}]
		table[x=dofs_3,y=error_tot_3,col sep=comma] 
		{\datapath/data/ref_sol_test_num_17dyn_shift_Dir_h_0.002967359050445104_tau_0.003125_all_pol_deg_h_conv_relative_time.txt};
		\addlegendentry{\(k=3\)}


		

      \end{axis}
    \end{tikzpicture}
%
	\endgroup%

    \label{fig:exp_Dir_onlyh-dofs}
  \end{subfigure}
  \par\bigskip
  \ref{Legendforall}
  \caption{Error plots for the initial data \eqref{eq:init_bump}. Top row: Equilibrated error $E(T=0.8)$ via \eqref{eq:relation_h_tau} for $k=1,2,3$. 
 In the left plot, the dashed lines indicate order~$\tau$.
  Bottom row: Convergence in $h$ for fixed $\tau$. The dashed lines indicate the expected order of convergence $h^{\frac{\ell_k}{\ell_k+1}}$ from Theorem~\ref{thm:convergence_db}.}
  \label{fig:exp_Dir_2x2grid}
\end{figure}

\appendix
\begin{center}
	\textbf{\large{Appendix}}
\end{center}
 \section{Proof of basic discrete stability properties}\label{sec:proof_of_etL-etLh}
\begin{proof}[Proof of Lemma~\ref{lem:basic_results}(a)]

We first recall the existence of a trace-compatible Scott--Zhang
quasi-interpolation operator \(J_h:H^1(\Omega_h;\Gamma_h)\rightarrow X_h^k\) satisfying
\begin{align}\label{eq:coupled-SZ}
	\|J_hv\|_{H^1(\Omega_h;\Gamma_h)}
	+h^{-1}\|v-J_hv\|_{L^2(\Omega_h;\Gamma_h)}
	\le C\|v\|_{H^1(\Omega_h;\Gamma_h)}.
\end{align}
We use the boundary-preserving Scott--Zhang construction from
\cite[Section~2]{scott1990finite}. More precisely, for an interior
degree of freedom, the averaging functional is chosen on a bulk
simplex, whereas for a boundary degree of freedom it is chosen, after
pullback, on a boundary face of the corresponding reference element.
The same averaging functional is used for the shared bulk--surface
boundary degree of freedom. Since the trace of the bulk nodal basis
coincides with the nodal basis of the induced surface finite element
space, it follows that
\[
\gamma_hJ_hv=J_{h,\Gamma}(\gamma_hv),
\]
where \(J_{h,\Gamma}\) is the Scott--Zhang operator on \(\Gamma_h\).
The standard local stability and approximation estimates then yield
\eqref{eq:coupled-SZ}; see \cite{scott1990finite}. The uniformity of
the constants on the curved isoparametric meshes follows from shape
regularity and the standard pullback estimates for the element and
boundary-face maps; see, e.g.,
\cite[Appendix~A]{chaumont2025pollution}.

Since \(Q_h\) is the \(m_h\)-orthogonal projection, it is contractive in \(L^2(\Omega_h;\Gamma_h)\). Moreover, its best-approximation property and \eqref{eq:coupled-SZ} imply
\begin{align}\label{eq:combined-Qh-error}
	\|v-Q_hv\|_{L^2(\Omega_h;\Gamma_h)}
	\le
	\|v-J_hv\|_{L^2(\Omega_h;\Gamma_h)}
	\le
	Ch\|v\|_{H^1(\Omega_h;\Gamma_h)}.
\end{align}
Furthermore, since \(J_hv\in X_h^k\), we have \(Q_hJ_hv=J_hv\).
Therefore, the inverse estimate, the \(L^2\)-contractivity of \(Q_h\), and \eqref{eq:coupled-SZ} yield
\begin{align*}
	\|Q_hv\|_{H^1(\Omega_h;\Gamma_h)}
	&\le
	\|J_hv\|_{H^1(\Omega_h;\Gamma_h)}
	+
	\|Q_h(v-J_hv)\|_{H^1(\Omega_h;\Gamma_h)}
	\\
	&\lesssim
	\|J_hv\|_{H^1(\Omega_h;\Gamma_h)}
	+
	h^{-1}
	\|Q_h(v-J_hv)\|_{L^2(\Omega_h;\Gamma_h)}
	\\
	&\le
	\|J_hv\|_{H^1(\Omega_h;\Gamma_h)}
	+
	h^{-1}
	\|v-J_hv\|_{L^2(\Omega_h;\Gamma_h)}
	\\
	&\lesssim
	\|v\|_{H^1(\Omega_h;\Gamma_h)}.
\end{align*}
Together with \eqref{eq:combined-Qh-error}, this proves part~(a).
\end{proof}

\begin{proof}[Proof of Lemma~\ref{lem:basic_results}(b)]
	For \(w_h\in X_h^k\), define the discrete negative norm by
	\[
	\|w_h\|_{-1,h}
	:=
	\sup_{\substack{v_h\in X_h^k\\ \|v_h\|_{H^1(\Omega_h;\Gamma_h)}=1}}
	m_h(w_h,v_h).
	\]
	We first prove that
	\begin{equation}\label{eq:Ah-half-discrete-Hminus1_db}
		\|\widetilde{A}_h^{-1/2}w_h\|_{L^2(\Omega_h;\Gamma_h)}
		\sim
		\|w_h\|_{-1,h}.
	\end{equation}
	Let \(z_h:=\widetilde{A}_h^{-1}w_h\in X_h^k\). Then, for every \(v_h\in X_h^k\), we have \(m_h(w_h,v_h)=m_h(\widetilde{A}_h z_h,v_h)=\tilde a_h(z_h,v_h)\). Hence, by the continuity of \(\tilde a_h\),
	\[
	\|w_h\|_{-1,h}
	=
	\sup_{\substack{v_h\in X_h^k\\ \|v_h\|_{H^1(\Omega_h;\Gamma_h)}=1}}
	\tilde a_h(z_h,v_h)
	\lesssim \|z_h\|_{H^1(\Omega_h;\Gamma_h)}.
	\]
	Conversely, choosing \(v_h=z_h/\|z_h\|_{H^1(\Omega_h;\Gamma_h)}\) and using the coercivity of \(\tilde a_h\), we obtain
	\[
	\|w_h\|_{-1,h}
	\gtrsim
	\frac{\tilde a_h(z_h,z_h)}{\|z_h\|_{H^1(\Omega_h;\Gamma_h)}}
	\gtrsim
	\|z_h\|_{H^1(\Omega_h;\Gamma_h)}.
	\]
	Therefore, by \eqref{eq:def_tilde_ah}:
	\begin{align}\label{eq:relation_wh}
		\|w_h\|_{-1,h}
		\sim
		\|z_h\|_{H^1(\Omega_h;\Gamma_h)}
		\sim
		\sqrt{\tilde a_h(z_h,z_h)}.
	\end{align}
	
	Moreover, since \(\widetilde{A}_h\) is symmetric positive definite on \(X_h^k\),
	\[
	\|\widetilde{A}_h^{-1/2}w_h\|_{L^2(\Omega_h;\Gamma_h)}^2
	=
	m_h(\widetilde{A}_h^{-1/2}w_h,\widetilde{A}_h^{-1/2}w_h)
	=
	m_h(\widetilde{A}_h^{-1}w_h,w_h)
	=
	m_h(z_h,\widetilde{A}_h z_h)
	=
	\tilde a_h(z_h,z_h).
	\]
	Combining this with \eqref{eq:relation_wh}, we obtain \eqref{eq:Ah-half-discrete-Hminus1_db}.
	
	Next, since \(X_h^k\subset H^1(\Omega_h;\Gamma_h)\), it follows immediately that
	\begin{align}\label{eq:relation_wh_1}
		\|w_h\|_{-1,h}
		=
		\sup_{\substack{v_h\in X_h^k\\ \|v_h\|_{H^1(\Omega_h;\Gamma_h)}=1}}
		m_h(w_h,v_h)
		\le
		\sup_{\substack{v\in H^1(\Omega_h;\Gamma_h)\\ \|v\|_{H^1(\Omega_h;\Gamma_h)}=1}}
		m_h(w_h,v)
		=
		\|w_h\|_{H^{-1}(\Omega_h;\Gamma_h)}.
	\end{align}
	
	To prove the reverse inequality, let \(v\in H^1(\Omega_h;\Gamma_h)\) satisfy \(\|v\|_{H^1(\Omega_h;\Gamma_h)}=1\). Since \(w_h\in X_h^k\) and \(Q_h\) is the \(L^2\)-orthogonal projection defined by \eqref{eq:def_Qh}, the \(H^1\)-stability of \(Q_h\) on quasi-uniform meshes implies that
	\[
	m_h(w_h,v)=m_h(w_h,Q_h v)
	\le
	\|w_h\|_{-1,h}\,\|Q_hv\|_{H^1(\Omega_h;\Gamma_h)}
	\lesssim \|w_h\|_{-1,h}\,\|v\|_{H^1(\Omega_h;\Gamma_h)}.
	\]
	Taking the supremum over all \(v\in H^1(\Omega_h;\Gamma_h)\) with \(\|v\|_{H^1(\Omega_h;\Gamma_h)}=1\), we conclude that
	\(
	\|w_h\|_{H^{-1}(\Omega_h;\Gamma_h)}
	\lesssim \|w_h\|_{-1,h}.
	\)
	Together with \eqref{eq:relation_wh_1}, this yields
	\begin{align}\label{eq:Hminus1_Hminus1h_db}
		\|w_h\|_{H^{-1}(\Omega_h;\Gamma_h)}\sim \|w_h\|_{-1,h}.
	\end{align}
	Combining this with \eqref{eq:Ah-half-discrete-Hminus1_db}, we complete the proof.
\end{proof}

\begin{proof}[Proof of Lemma~\ref{lem:basic_results}(c)]
	Since \(X_h^k\times X_h^k\) is finite-dimensional,
	\(L_h\) is a bounded linear operator and therefore generates a
	\(C_0\)-group.
	Let \(\widetilde{U}_h(t)= \big(\tilde{u}_h(t),\tilde{v}_h(t)\big)^{\top}=e^{tL_h}U_h(0).\)
	Then
	\begin{align}\label{eq:discrete_linear_wave}
		\partial_t\tilde{u}_h=\tilde{v}_h,
		\qquad
		\partial_t\tilde{v}_h=-A_h\tilde{u}_h.		
	\end{align}
	
	We first prove the estimate \eqref{assump_etLh} for \(\theta=1\). Define
	\begin{align*}
		\mathcal{E}_{1}(t):=
		\widetilde a_h(\tilde u_h(t),\tilde u_h(t))
		+
		m_h(\tilde v_h(t),\tilde v_h(t)).
	\end{align*}
	Using \eqref{eq:discrete_linear_wave}, the definition of \(A_h\),
	and the symmetry of \(a_h\) and \(m_h\), we obtain
	\begin{align*}
		\frac{\d}{\d t}\mathcal{E}_{1}(t)
		&=
		2\widetilde a_h(\tilde u_h,\tilde v_h)
		+
		2m_h(\tilde v_h,-A_h\tilde u_h)
		\\
		&=
		2a_h(\tilde u_h,\tilde v_h)
		+
		2c_Gm_h(\tilde u_h,\tilde v_h)
		-
		2m_h(A_h\tilde u_h,\tilde v_h)=
		2c_Gm_h(\tilde u_h,\tilde v_h).
	\end{align*}
	Since \(\widetilde a_h(\tilde u_h,\tilde u_h)
	\ge c_Gm_h(\tilde u_h,\tilde u_h),\)
	the Cauchy--Schwarz and Young inequalities imply
	\[
	\frac{\d}{\d t}\mathcal{E}_{1}(t)
	\le C \mathcal{E}_{1}(t),
	\]
	where \(C\) is independent of \(h\). Therefore, Gronwall's inequality gives
	\begin{align}
		\mathcal{E}_{1}(t)\le e^{Ct}\mathcal{E}_{1}(0).
	\end{align}
	By the uniform equivalence of \(\widetilde a_h(\cdot,\cdot)\)
	with the \(H^1(\Omega_h;\Gamma_h)\)-norm, this yields
	\begin{align}\label{eq:etLh_1}
		\|e^{tL_h}U_h(0)\|_
		{H^1(\Omega_h;\Gamma_h)\times L^2(\Omega_h;\Gamma_h)}
		\le
		e^{C_Lt}
		\|U_h(0)\|_
		{H^1(\Omega_h;\Gamma_h)\times L^2(\Omega_h;\Gamma_h)}.
	\end{align}
	
	We next prove the estimate for \(\theta=0\). Since \(\widetilde A_h\) is invertible, define the weak energy
	\begin{align*}
		\mathcal{E}_{0}(t)
		:=
		m_h(\tilde u_h(t),\tilde u_h(t))
		+
		m_h(\widetilde A_h^{-1}\tilde v_h(t),\tilde v_h(t)).
	\end{align*}
	Equivalently,
	\begin{align*}
		\mathcal{E}_{0}(t)
		=
		\|\tilde u_h(t)\|_{L^2(\Omega_h;\Gamma_h)}^2
		+
		\|\widetilde A_h^{-1/2}\tilde v_h(t)\|_
		{L^2(\Omega_h;\Gamma_h)}^2.
	\end{align*}
	Using \(A_h=\widetilde A_h-c_G\), we calculate
	\begin{align*}
		\frac{\d}{\d t}\mathcal{E}_{0}(t)
		&=
		2m_h(\tilde u_h,\tilde v_h)
		-
		2m_h(\widetilde A_h^{-1}\tilde v_h,A_h\tilde u_h)
		\\
		&=
		2m_h(\tilde u_h,\tilde v_h)
		-
		2m_h(\widetilde A_h^{-1}\tilde v_h,
		(\widetilde A_h-c_G)\tilde u_h)=
		2c_Gm_h(\widetilde A_h^{-1}\tilde v_h,\tilde u_h).
	\end{align*}
	The lower spectral bound of \(\widetilde A_h\) implies
	\begin{align*}
		\|\widetilde A_h^{-1}\tilde v_h\|_{L^2(\Omega_h;\Gamma_h)}
		\le
		c_G^{-1/2}
		\|\widetilde A_h^{-1/2}\tilde v_h\|_
		{L^2(\Omega_h;\Gamma_h)}.
	\end{align*}
	Consequently,
	\begin{align*}
		\left|
		\frac{\d}{\d t}\mathcal{E}_{0}(t)
		\right|
		&\le
		2c_G^{1/2}
		\|\tilde u_h(t)\|_{L^2(\Omega_h;\Gamma_h)}
		\|\widetilde A_h^{-1/2}\tilde v_h(t)\|_
		{L^2(\Omega_h;\Gamma_h)}\le C \mathcal{E}_{0}(t),
	\end{align*}
	with \(C\) independent of \(h\). Another application of Gronwall's inequality gives
	\[
	\mathcal{E}_{0}(t)\le e^{Ct}\mathcal{E}_{0}(0).
	\]
	By part~(b),
	\[
	\mathcal{E}_{0}(t)^{1/2}
	\sim
	\|\widetilde U_h(t)\|_
	{L^2(\Omega_h;\Gamma_h)\times
		H^{-1}(\Omega_h;\Gamma_h)},
	\]
	uniformly in \(h\). Hence,
	\begin{align}\label{eq:etLh_0}
		\|e^{tL_h}U_h(0)\|_
		{L^2(\Omega_h;\Gamma_h)\times
			H^{-1}(\Omega_h;\Gamma_h)} 
		\le
		e^{C_Lt}
		\|U_h(0)\|_
		{L^2(\Omega_h;\Gamma_h)\times
			H^{-1}(\Omega_h;\Gamma_h)}.
	\end{align}
	Combining \eqref{eq:etLh_1} and \eqref{eq:etLh_0} proves the assertion.
\end{proof}

\begin{proof}[Proof of Lemma~\ref{lem:basic_results}(d)]
	For \(\tau>0\), the representation
	\[
	\varphi_1(\tau L_h)
	=
	\frac1\tau\int_0^\tau e^{\sigma L_h}\,\d\sigma
	\]
	together with part~(c) yields \eqref{eq:phi1_stability_db}; the case
	\(\tau=0\) follows from \(\varphi_1(0)=\mathrm{Id}\).
\end{proof}

\section{Geometric transfer and projection estimates}
In this appendix, we collect the estimates for the lift and adjoint-lift
operators that are used in the proof of the main error bounds. In particular,
we derive approximation estimates in the \(L^2(\Omega;\Gamma)\)-norm and in the
weak \(H^{-1}(\Omega;\Gamma)\)-norm. As a first step, we record the geometric
consistency estimates for the lifted bilinear forms. These quantify the defect
between the discrete bilinear forms on the computational domain
\((\Omega_h;\Gamma_h)\) and their continuous counterparts on the exact domain
\((\Omega,\Gamma)\) after transfer by the lift operator, and they will be used
repeatedly throughout this appendix.

\begin{lemma}[Geometric consistency of the lifted bilinear forms]
	\label{lem:geom_consistency}
	Let \(\mathcal L_h\) be the lift operator associated with the isoparametric
	discretization of degree \(k\). Then, for all
	\(\eta_h,\xi_h\in L^2(\Omega_h;\Gamma_h)\),
	\begin{align}
		\bigl|
		m_h(\eta_h,\xi_h)-m(\mathcal L_h\eta_h,\mathcal L_h\xi_h)
		\bigr|
		&\lesssim
		h^k
		\|\mathcal L_h\eta_h\|_{L^2(B_h^{\mathcal L})}
		\|\mathcal L_h\xi_h\|_{L^2(B_h^{\mathcal L})}
		\notag\\
		&\quad
		+
		h^{k+1}
		\|\mathcal L_h\eta_h\|_{L^2(\Omega;\Gamma)}
		\|\mathcal L_h\xi_h\|_{L^2(\Omega;\Gamma)}.
		\label{eq:geom_consistency_m}
	\end{align}
	Moreover, for all \(\eta_h,\xi_h\in H^1(\Omega_h;\Gamma_h)\),
	\begin{align}
		\bigl|
		\tilde a_h(\eta_h,\xi_h)-\tilde a(\mathcal L_h\eta_h,\mathcal L_h\xi_h)
		\bigr|
		&\lesssim
		h^k
		\|\mathcal L_h\eta_h\|_{H^1(B_h^{\mathcal L})}
		\|\mathcal L_h\xi_h\|_{H^1(B_h^{\mathcal L})}
		\notag\\
		&\quad
		+
		h^{k+1}
		\|\mathcal L_h\eta_h\|_{H^1(\Omega;\Gamma)}
		\|\mathcal L_h\xi_h\|_{H^1(\Omega;\Gamma)}.
		\label{eq:geom_consistency_a}
	\end{align}
	Here, \(B_h^{\L}\) denotes the layer of lifted elements with a boundary face.
\end{lemma}
\begin{proof}
	The stated geometric consistency estimates follow from
	\cite[Lemma~10.13]{ER2021} for bulk--surface problems. More precisely, they are
	obtained by combining the geometric approximation estimates for the bulk domain
	and the surface given in \cite[Lemmas~8.24 and 9.24]{ER2021}; see also
	\cite[Lemma~5.6]{kovacs2018high}.
\end{proof}
To estimate the terms over the boundary layer \(B_h^{\mathcal L}\) in
\eqref{eq:geom_consistency_m} and \eqref{eq:geom_consistency_a}, we use the
following important boundary-layer estimate.

\begin{lemma}[\cite{ER2013}, Lemma~6.3]\label{lem:boundary_layer_estimate}
	For all \(\eta\in H^1(\Omega)\), there holds
	\begin{align}\label{eq:boundary_layer_estimate}
		\|\eta\|_{L^2(B_h^{\mathcal L})}\lesssim h^{1/2}\|\eta\|_{H^1(\Omega)}.
	\end{align}
\end{lemma}

We next derive the following projection estimate, which will be a key
ingredient in the proof of the error bound for the lifted adjoint operator
\(\mathcal L_h^{H*}\).

\begin{lemma}
	\label{lem:L2_projection_error} 
	Let \(Q_h : L^2(\Omega_h;\Gamma_h)\to X_h^k\)
	be the \(L^2\)-projection defined by \eqref{eq:def_Qh}.
	Then, for every \(u\in H^{k+1}(\Omega;\Gamma)\) and every integer \(0\le \ell\le k+1\),
	there holds
	\begin{align*}
		\|(1-Q_h)\mathcal L_h^{-1}u\|_{L^2(\Omega_h;\Gamma_h)}
		\lesssim h^\ell \|u\|_{H^\ell(\Omega;\Gamma)}.
	\end{align*}
\end{lemma}

\begin{proof}
	We first consider the cases \(\ell=0,1\). By
	\eqref{eq:boundedness_lift_operator}, the inverse lift satisfies
	\(\mathcal L_h^{-1}u\in H^1(\Omega_h;\Gamma_h)\). Hence, the standard
	approximation estimate for the \(L^2\)-projection \(Q_h\) yields
	\begin{align*}
		\|(1-Q_h)\mathcal L_h^{-1}u\|_{L^2(\Omega_h;\Gamma_h)}
		\lesssim h^\ell \|\mathcal L_h^{-1}u\|_{H^\ell(\Omega_h;\Gamma_h)}
		\lesssim h^\ell \|u\|_{H^\ell(\Omega;\Gamma)}.
	\end{align*}
	
	We now turn to the case \(2\le \ell\le k+1\). Since the inverse lift is induced by a geometric mapping that is only piecewise smooth, one cannot directly invoke a global \(H^\ell\)-stability estimate for \(\mathcal L_h^{-1}\). Instead, we use the best-approximation property of the \(L^2\)-projection together with local interpolation estimates for the nodal interpolant applied to \(\mathcal L_h^{-1}u\).
	
	Since \(d\le3\) and \(\ell\ge2\), nodal interpolation is well defined. Moreover, \(u\in H^\ell(\Omega;\Gamma)\) is trace compatible, so \(I_h\mathcal L_h^{-1}u\) coincides with the standard trace-compatible nodal interpolant; see Remark~\ref{rem:relation_Ih}. Let \(\mathcal F_h\) denote the set of boundary faces of \(\mathcal T_h\). Then
	\begin{align*}
		&\|(1-Q_h)\mathcal L_h^{-1}u\|_{L^2(\Omega_h;\Gamma_h)}
		\le
		\|(1-I_h)\mathcal L_h^{-1}u\|_{L^2(\Omega_h;\Gamma_h)} \\
		&\qquad=
		\left(
		\sum_{K\in \mathcal T_h}
		\|(1-I_{h,\Omega})\mathcal L_h^{-1}u\|_{L^2(K)}^2
		+
		\sum_{F\in \mathcal F_h}
		\|(1-I_{h,\Gamma})\mathcal L_h^{-1}u\|_{L^2(F)}^2
		\right)^{1/2} \\
		&\qquad\lesssim
		\left(
		\sum_{K\in \mathcal T_h}
		h^{2\ell}\|\mathcal L_h^{-1}u\|_{H^\ell(K)}^2
		+
		\sum_{F\in \mathcal F_h}
		h^{2\ell}\|\mathcal L_h^{-1}u\|_{H^\ell(F)}^2
		\right)^{1/2} \\
		&\qquad\lesssim
		\left(
		\sum_{K\in \mathcal T_h}
		h^{2\ell}\|u\|_{H^\ell(\mathcal L_h(K))}^2
		+
		\sum_{F\in \mathcal F_h}
		h^{2\ell}\|u\|_{H^\ell(\mathcal L_h(F))}^2
		\right)^{1/2} \lesssim h^\ell \|u\|_{H^\ell(\Omega;\Gamma)}.
	\end{align*}
	This completes the proof.
\end{proof}

\subsection{The proof of Lemma~\ref{lem:estimate_Lh}} \label{Appendix:A}
We proceed in three steps. 
First, we prove the \(L^2\)-estimate \eqref{eq:L2_eastimate_Lh*}. 
Second, we establish the auxiliary bound \eqref{eq:bounded_H1}, which is later recorded in Corollary~\ref{cor:boundedness_properties} but is needed here. 
Finally, we prove the \(H^{-1}\)-estimate \eqref{eq:H-1_estimate}.

\begin{proof}[Proof of \eqref{eq:L2_eastimate_Lh*}]
	From the definition of the bilinear form $m$, we have the following equality:
	\begin{align}\label{eq:L2_estimate_Lh*_1}
		\|(1-\L_h \L_h^{H*})u\|_{L^2(\Omega;\Gamma)}&=\sup_{\|v\|_{L^2(\Omega;\Gamma)}=1}m\left((1-\L_h \L_h^{H*})u,v\right)\notag\\
		&=\sup_{\|v\|_{L^2(\Omega;\Gamma)}=1}m\left((\L_h\L_h^{-1}-\L_h \L_h^{H*})u,v\right)\notag\\
		&\leq\sup_{\|v\|_{L^2(\Omega;\Gamma)}=1}m\left(\L_h(1-Q_h)\L_h^{-1}u,v\right)\notag\\
		&\quad +\sup_{\|v\|_{L^2(\Omega;\Gamma)}=1}\Big(m\left(\L_hQ_h\L_h^{-1}u,v\right)-m\left(\L_h \L_h^{H*}u,v\right)\Big),
	\end{align}
	where $Q_h$ is the classical $L^2$ projection defined in \eqref{eq:def_Qh}.
	Since \(Q_h\L_h^{-1}u \in X^k_h\) and \(\L_h^{H*}u \in X^k_h\), the definitions of \(Q_h\) and \(\L_h^{H*}\) imply
	\begin{align}\label{eq:relation_m}
		m\left(\L_hQ_h\L_h^{-1}u,v\right)=m_h\left(Q_h\L_h^{-1}u,\L^{H*}_hv\right),\quad \text{and}\quad m\left(\L_h \L_h^{H*}u,v\right)=m\left(u,\L_h \L_h^{H*}v\right).
	\end{align}
	In particular, \(\L_h\L_h^{H*}\) is self-adjoint with respect to \(m\).

	We infer from the second equality in \eqref{eq:relation_m} that $\L_h \L_h^{H*}$ is a self-adjoint operator with respect to the inner product in $L^2(\Omega; \Gamma)$. Using the relations in \eqref{eq:relation_m}, we can further decompose \eqref{eq:L2_estimate_Lh*_1} as follows:
	\begin{align}\label{eq:L2_estimate_Lh*_2}
		&\|(1-\L_h \L_h^{H*})u\|_{L^2(\Omega;\Gamma)}\notag\\
		&\leq \sup_{\|v\|_{L^2(\Omega;\Gamma)}=1}m\left(\L_h(1-Q_h)\L_h^{-1}u,v\right)+\sup_{\|v\|_{L^2(\Omega;\Gamma)}=1}\Big(m_h\left(Q_h\L_h^{-1}u,\L^{H*}_hv\right)-m\left(u,\L_h \L_h^{H*}v\right)\Big)\notag\\
		&\leq\sup_{\|v\|_{L^2(\Omega;\Gamma)}=1}m\left(\L_h(1-Q_h)\L_h^{-1}u,v\right)+\sup_{\|v\|_{L^2(\Omega;\Gamma)}=1}m_h\left((1-Q_h)\L_h^{-1}u,\L^{H*}_hv\right)\notag\\
		&\quad +\sup_{\|v\|_{L^2(\Omega;\Gamma)}=1}\Big(m_h\left(\L_h^{-1}u,\L^{H*}_hv\right)-m\left(u,\L_h \L_h^{H*}v\right)\Big).
	\end{align}
	The first two terms on the right-hand side of \eqref{eq:L2_estimate_Lh*_2}
	are estimated by Lemma~\ref{lem:L2_projection_error} together with
	\eqref{eq:boundedness_lift_operator} and \eqref{eq:L2_bound_Lh*}. Indeed,
	for every \(\ell\in\{0,1,\dots,k\}\),
	\begin{align}\label{eq:L2_estimate_part1_part2}
		&m\!\left(\L_h(1-Q_h)\L_h^{-1}u,v\right)+m_h\left((1-Q_h)\L_h^{-1}u,\L^{H*}_hv\right)\notag\\
		&\lesssim \|\L_h(1-Q_h)\L_h^{-1}u\|_{L^2(\Omega;\Gamma)}\|v\|_{L^2(\Omega;\Gamma)}+\|(1-Q_h)\L_h^{-1}u\|_{L^2(\Omega_h;\Gamma_h)}\|\L^{H*}_hv\|_{L^2(\Omega_h;\Gamma_h)}\notag\\
		&\lesssim \|(1-Q_h)\L_h^{-1}u\|_{L^2(\Omega_h;\Gamma_h)}\|v\|_{L^2(\Omega;\Gamma)}\;\lesssim\;
		h^{\ell}\|u\|_{H^\ell(\Omega; \Gamma)}\,
		\|v\|_{L^2(\Omega;\Gamma)}.
	\end{align}
	
	To estimate the last term on the right-hand side of
	\eqref{eq:L2_estimate_Lh*_2}, we apply the geometric consistency estimate
	\eqref{eq:geom_consistency_m} from Lemma~\ref{lem:geom_consistency}. This yields
	\begin{align}\label{eq:L2_estimate_part3}
		&\Big|m_h\left(\L_h^{-1}u,\L^{H*}_hv\right)-m\left(u,\L_h \L_h^{H*}v\right)\Big|\notag\\
		&\quad\lesssim h^{k}\|u\|_{L^2(B_h^{\L})}\|\L_h \L_h^{H*}v\|_{L^2(B_h^{\L})}
		+ h^{k+1}\|u\|_{L^2(\Omega;\Gamma)}\|\L_h \L_h^{H*}v\|_{L^2(\Omega;\Gamma)}\notag\\
		&\quad\lesssim h^{k}\|u\|_{L^2(\Omega;\Gamma)}\|\L_h \L_h^{H*}v\|_{L^2(\Omega;\Gamma)}\lesssim h^{k}\|u\|_{L^2(\Omega;\Gamma)}.
	\end{align}
	In the last step, we used the inclusion \(B_h^{\L}\subset\Omega\) together with the \(L^2(\Omega;\Gamma)\)-boundedness of the operators \(\L_h\) and \(\L_h^{H*}\); see \eqref{eq:boundedness_lift_operator} and \eqref{eq:L2_bound_Lh*}, respectively.

	Now, substituting \eqref{eq:L2_estimate_part1_part2} and \eqref{eq:L2_estimate_part3} into \eqref{eq:L2_estimate_Lh*_2}, we deduce \eqref{eq:L2_eastimate_Lh*}.
\end{proof}

As an immediate consequence of \eqref{eq:L2_eastimate_Lh*}, combined with the inverse inequality, we obtain the \(H^1\)-boundedness of \(\L_h^{H*}\), namely \eqref{eq:bounded_H1}. Although this estimate is stated later in Corollary~\ref{cor:boundedness_properties}, we prove it here since it will be used in the proof of \eqref{eq:H-1_estimate}.
\begin{proof}[Proof of \eqref{eq:bounded_H1}]
	By the triangle inequality,
	\begin{align}\label{eq:triangle_LH*}
		\|\L_h^{H*}u\|_{H^{1}(\Omega_h;\Gamma_h)}
		\leq \|Q_h \L_h^{-1}u\|_{H^{1}(\Omega_h;\Gamma_h)}
		+ \|Q_h \L_h^{-1}u - \L_h^{H*}u\|_{H^{1}(\Omega_h;\Gamma_h)}.
	\end{align}
	Using the estimate in \eqref{eq:boundedness_lift_operator} and the \(H^1\)-boundedness of the classical \(L^2\)-projection operator \(Q_h\), we obtain
	\begin{align}\label{eq:boundedness_Qh}
		\|Q_h \L_h^{-1}u\|_{H^{1}(\Omega_h;\Gamma_h)}
		\lesssim \|\L_h^{-1}u\|_{H^{1}(\Omega_h;\Gamma_h)}
		\lesssim \|u\|_{H^1(\Omega;\Gamma)}.
	\end{align}
	For the second term on the right-hand side of \eqref{eq:triangle_LH*}, the inverse inequality and the \(L^2\)-estimate \eqref{eq:L2_eastimate_Lh*} yield
	\begin{align*}
		\|Q_h \L_h^{-1}u - \L_h^{H*}u\|_{H^{1}(\Omega_h;\Gamma_h)}
		&\lesssim h^{-1}\|Q_h \L_h^{-1}u - \L_h^{H*}u\|_{L^2(\Omega_h;\Gamma_h)} \\
		&\leq h^{-1}\Bigl(
		\|(1-Q_h)\L_h^{-1}u\|_{L^2(\Omega_h;\Gamma_h)}
		+ \|\L_h^{-1}u - \L_h^{H*}u\|_{L^2(\Omega_h;\Gamma_h)}
		\Bigr) \\
		&\lesssim
		h^{-1}\|(1-Q_h)\L_h^{-1}u\|_{L^2(\Omega_h;\Gamma_h)}
		+  h^{-1}\|(1-\L_h\L_h^{H*})u\|_{L^2(\Omega;\Gamma)} \\
		&\lesssim \|u\|_{H^1(\Omega;\Gamma)},
	\end{align*}
	where in the last step we used Lemma~\ref{lem:L2_projection_error} and \eqref{eq:L2_eastimate_Lh*}. Combining this estimate with \eqref{eq:triangle_LH*} and \eqref{eq:boundedness_Qh}, we obtain \eqref{eq:bounded_H1}.
\end{proof}

We are now in a position to prove the weak-norm error estimate \eqref{eq:H-1_estimate}. The argument combines the \(L^2\)-estimate \eqref{eq:L2_eastimate_Lh*} with the \(H^1\)-boundedness of \(\L_h^{H*}\) established in \eqref{eq:bounded_H1}.

\begin{proof}[Proof of \eqref{eq:H-1_estimate}]
	{By duality, and decomposing 
		\(v=(1-\L_h\L_h^{H*})v+\L_h\L_h^{H*}v\), we obtain
		\begin{align}\label{eq:decompose_H-1_estimate}
			&\|(1-\L_h \L_h^{H*})u\|_{H^{-1}(\Omega;\Gamma)}\notag\\
			&\quad=\sup_{\|v\|_{H^1(\Omega;\Gamma)}=1}m\left((1-\L_h \L_h^{H*})u,v\right)\notag\\
			&\quad\leq \sup_{\|v\|_{H^1(\Omega;\Gamma)}=1}m\left((1-\L_h \L_h^{H*})u,(1-\L_h \L_h^{H*})v\right)+\sup_{\|v\|_{H^1(\Omega;\Gamma)}=1}m\left((1-\L_h \L_h^{H*}) u,\L_h \L_h^{H*}v\right)\notag\\
			&\quad\leq \sup_{\|v\|_{H^1(\Omega;\Gamma)}=1} \|(1-\L_h \L_h^{H*})u\|_{L^2(\Omega;\Gamma)}\|(1-\L_h \L_h^{H*})v\|_{L^2(\Omega;\Gamma)}\notag\\
			&\qquad+ \sup_{\|v\|_{H^1(\Omega;\Gamma)}=1}\Big(m\left(u,\L_h \L_h^{H*} v\right)-m\left(\L_h \L_h^{H*} u,\L_h \L_h^{H*}v\right)\Big)
		\end{align}
		From the estimate of the operator $\mathcal{L}_h \mathcal{L}_h^{H*}$ in \eqref{eq:L2_eastimate_Lh*}, the first term on the right-hand side of
		\eqref{eq:decompose_H-1_estimate} can be bounded as
		\begin{align}\label{eq:decompose_H-1_estimate_part1}
			&\sup_{\|v\|_{H^1(\Omega;\Gamma)}=1} 
			\|(1-\L_h \L_h^{H*})u\|_{L^2(\Omega;\Gamma)}\,
			\|(1-\L_h \L_h^{H*})v\|_{L^2(\Omega;\Gamma)}\notag\\
			&\quad\lesssim\;
			\sup_{\|v\|_{H^1(\Omega;\Gamma)}=1} h^{\ell}\|u\|_{H^{\ell}(\Omega;\Gamma)}\cdot h\|v\|_{H^1(\Omega;\Gamma)}=h^{\ell+1}\|u\|_{H^{\ell}(\Omega;\Gamma)}.
		\end{align}
		for all $0\leq \ell\leq k$.

		For the second term in \eqref{eq:decompose_H-1_estimate}, we treat separately the cases \(\ell\geq1\) and \(\ell=0\). We first note that \(m\left(u,\L_h \L_h^{H*} v\right)=m_h\left(\L_h^{H*} u,\L_h^{H*}v\right)\).
		Applying the geometric consistency estimate \eqref{eq:geom_consistency_m} and boundary-layer estimate \eqref{eq:boundary_layer_estimate}, we obtain
		\begin{align}\label{eq:decompose_H-1_estimate_part2_1}
			&\Big|m\left(u,\L_h \L_h^{H*} v\right)-m\left(\L_h \L_h^{H*} u,\L_h \L_h^{H*}v\right)\Big| =\Big|m_h\left(\L_h^{H*} u,\L_h^{H*}v\right)-m\left(\L_h \L_h^{H*} u,\L_h \L_h^{H*}v\right)\Big|\notag\\
			&\quad\lesssim h^{k}\|\L_h \L_h^{H*} u\|_{L^2(B_h^{\L})}\|\L_h \L_h^{H*}v\|_{L^2(B_h^{\L})}
			+ h^{k+1}\|\L_h \L_h^{H*} u\|_{L^2(\Omega;\Gamma)}\|\L_h \L_h^{H*}v\|_{L^2(\Omega;\Gamma)}\notag\\
			&\quad \lesssim h^{k+1}\|\L_h \L_h^{H*} u\|_{H^1(\Omega)}\|\L_h \L_h^{H*}v\|_{H^1(\Omega)}
			+ h^{k+1}\|\L_h \L_h^{H*} u\|_{L^2(\Omega;\Gamma)}\|\L_h \L_h^{H*}v\|_{L^2(\Omega;\Gamma)}\notag\\
			&\quad\lesssim h^{k+1}\|\L_h \L_h^{H*} u\|_{H^1(\Omega;\Gamma)}\|\L_h \L_h^{H*}v\|_{H^1(\Omega;\Gamma)}\notag\\
			&\quad\lesssim h^{k+1}\| u\|_{H^1(\Omega;\Gamma)}\|v\|_{H^1(\Omega;\Gamma)}\lesssim
			h^{\ell+1}
			\|u\|_{H^\ell(\Omega;\Gamma)}
			\|v\|_{H^1(\Omega;\Gamma)},
		\end{align}
		for $1\leq \ell\leq k$, where we have used the \(H^1\)-boundedness of \(\L_h^{H*}\) established in \eqref{eq:bounded_H1} and boundedness of \(\L_h\) in the last inequality. 
		
		And Similarly, for $\ell=0$, retaining the \(L^2(B_h^{\L})\)-norm in the above estimate and using
		the inclusion \(B_h^{\L}\subset \Omega_h\), we have
		\begin{align}\label{eq:decompose_H-1_estimate_part2_2}
			\Big|
			m\left(u,\L_h \L_h^{H*} v\right)
			-
			m\left(\L_h \L_h^{H*} u,\L_h \L_h^{H*}v\right)
			\Big| 
			&\lesssim
			h^{k}
			\|\L_h \L_h^{H*} u\|_{L^2(\Omega;\Gamma)}
			\|\L_h \L_h^{H*}v\|_{L^2(\Omega;\Gamma)}
			\notag\\
			&\lesssim
			h
			\|u\|_{L^2(\Omega;\Gamma)}
			\|v\|_{L^2(\Omega;\Gamma)} .
		\end{align}
		Finally, combining
		\eqref{eq:decompose_H-1_estimate},
		\eqref{eq:decompose_H-1_estimate_part1}, \eqref{eq:decompose_H-1_estimate_part2_1},
		and \eqref{eq:decompose_H-1_estimate_part2_2},
		we obtain the desired estimate \eqref{eq:H-1_estimate}.}
\end{proof}

\subsection{Proof of Corollary~\ref{cor:boundedness_properties}}\label{Appendix:B}
The estimate \eqref{eq:bounded_H1} has already been proved in Subsection~\ref{Appendix:A}, since it is needed in the proof of \eqref{eq:H-1_estimate} in Lemma~\ref{lem:estimate_Lh}. 
Hence, it remains to establish \eqref{eq:equivalence_H-1} and \eqref{eq:bounded_H-1}.

\begin{proof}[Proof of \eqref{eq:equivalence_H-1}]
	Using duality and the $H^1$-boundedness of $\L^{H*}_h$ in \eqref{eq:bounded_H1}, we have
	\begin{align}\label{eq:bound_H-1}
		\|\L_h u_h\|_{H^{-1}(\Omega;\Gamma)}&=\sup_{\|v\|_{H^1(\Omega;\Gamma)}=1}m\left(\L_h u_h,v\right)=\sup_{\|v\|_{H^1(\Omega;\Gamma)}=1}m_h\left( u_h,\L^{H*}_hv\right)\notag\\
		&\leq \sup_{\|v\|_{H^1(\Omega;\Gamma)}=1} \|u_h\|_{H^{-1}(\Omega_h;\Gamma_h)}\|\L^{H*}_hv\|_{H^1(\Omega_h;\Gamma_h)}\lesssim \|u_h\|_{H^{-1}(\Omega_h;\Gamma_h)}.
	\end{align}
	This establishes the boundedness result of $\L_h$ in $H^{-1}$ in \eqref{eq:equivalence_H-1}.
	
	On the other hand, we begin by noting that
	\begin{align}\label{eq:decompose_H-1}
		&\|u_h\|_{H^{-1}(\Omega_h;\Gamma_h)}=\sup_{\|v_h\|_{H^1(\Omega_h;\Gamma_h)}=1}m_h\left(u_h,v_h\right)\notag\\
		&\leq\sup_{\|v_h\|_{H^1(\Omega_h;\Gamma_h)}=1}\Big(m_h\left(u_h,v_h\right)-m(\L_h u_h, \L_h v_h)\Big)+\sup_{\|v_h\|_{H^1(\Omega_h;\Gamma_h)}=1}m(\L_h u_h, \L_h v_h).
	\end{align}
	For the second term on the right-hand side of \eqref{eq:decompose_H-1}, using the boundedness of the lift operator in \eqref{eq:boundedness_lift_operator}, we deduce:
	\begin{align}\label{eq:m_LhuLhv}
		m(\L_h u_h, \L_h v_h)\leq \|\L_h u_h\|_{H^{-1}(\Omega;\Gamma)}\|\L_h v_h\|_{H^1(\Omega;\Gamma)}\lesssim \|\L_h u_h\|_{H^{-1}(\Omega;\Gamma)}\|v_h\|_{H^1(\Omega_h;\Gamma_h)}.
	\end{align}
	To estimate the first term on the right-hand side of \eqref{eq:decompose_H-1}, we again invoke the geometric consistency estimate \eqref{eq:geom_consistency_m}, the inclusion \(B_h^{\L}\subset\Omega\) together with the boundary-layer estimate from Lemma~\ref{lem:boundary_layer_estimate}.
	\begin{align}\label{eq:mh-m}
		&m_h\left(u_h,v_h\right)-m(\L_h u_h,\L_h v_h)\notag\\
		&\quad\lesssim h^{k}\|\L_h u_h\|_{L^2(B_h^{\L})}\|\L_h v_h\|_{L^2(B_h^{\L})}
		+ h^{k+1}\|\L_h u_h\|_{L^2(\Omega;\Gamma)}\|\L_h v_h\|_{L^2(\Omega;\Gamma)}\notag\\
		&\quad\lesssim \left(h^{k+\frac{1}{2}}\|\L_h  v_h\|_{H^1(\Omega;\Gamma)}+h^{k+1}\|\L_h v_h\|_{L^2(\Omega;\Gamma)}\right)\|\L_h u_h\|_{L^2(\Omega;\Gamma)}.
	\end{align}
	Furthermore, since \(u_h\in X_h^k\), the inverse inequality on the discrete space, together with the norm equivalence under lifting, yields
	\begin{align*}
		\|\L_h u_h\|^2_{L^2(\Omega;\Gamma)}=m(\L_h u_h,\L_h u_h)&\leq \|\L_h u_h\|_{H^{-1}(\Omega;\Gamma)}\|\L_h u_h\|_{H^{1}(\Omega;\Gamma)}\\
		&\lesssim h^{-1}\|\L_h u_h\|_{H^{-1}(\Omega;\Gamma)}\|\L_h u_h\|_{L^2(\Omega;\Gamma)}.
	\end{align*}
	This implies:
	$\|\L_h u_h\|_{L^2(\Omega;\Gamma)}\lesssim h^{-1}\|\L_h u_h\|_{H^{-1}(\Omega;\Gamma)}$. Now, substituting this inequality into \eqref{eq:mh-m}, and noting that $k\geq 1$, we derive:
	\begin{align}\label{eq:bound_mh-m}
		m_h\left(u_h,v_h\right)-m(\L_h u_h, \L_h v_h)&\lesssim h^{k-\frac{1}{2}}\|\L_h u_h\|_{H^{-1}(\Omega;\Gamma)}\cdot\|\L_h v_h\|_{H^{1}(\Omega;\Gamma)}\notag\\
		&\lesssim h^{\frac{1}{2}}\|\L_h u_h\|_{H^{-1}(\Omega;\Gamma)}\|v_h\|_{H^1(\Omega_h;\Gamma_h)}.
	\end{align}
	Finally, combining \eqref{eq:decompose_H-1}, \eqref{eq:m_LhuLhv}, and \eqref{eq:bound_mh-m}, we conclude that:
	\begin{align*}
		\|u_h\|_{H^{-1}(\Omega_h;\Gamma_h)}\lesssim \|\L_h u_h\|_{H^{-1}(\Omega;\Gamma)}.
	\end{align*}
	This completes the proof of \eqref{eq:equivalence_H-1}.
\end{proof}

\begin{proof}[Proof of \eqref{eq:bounded_H-1}]
	Since $\L_h^{H*}u \in X^k_h$, applying \eqref{eq:equivalence_H-1} yields
	\begin{align}\label{eq:boundedness_Lh_1}
		\|\L_h^{H*}u\|_{H^{-1}(\Omega_h;\Gamma_h)}\lesssim \|\L_h \L_h^{H*}u\|_{H^{-1}(\Omega;\Gamma)}.
	\end{align}
	Furthermore, by duality and the boundedness result established in \eqref{eq:bounded_H1}, we have
	\begin{align*}
		\|\L_h \L_h^{H*}u\|_{H^{-1}(\Omega;\Gamma)}&=\sup_{\|w\|_{H^1(\Omega;\Gamma)}=1}m(\L_h \L_h^{H*}u, w)\\
		&=\sup_{\|w\|_{H^1(\Omega;\Gamma)}=1}m(u, \L_h \L_h^{H*}w)\\
		&\leq \sup_{\|w\|_{H^1(\Omega;\Gamma)}=1} \|u\|_{H^{-1}(\Omega;\Gamma)} \|\L_h \L_h^{H*}w\|_{H^{1}(\Omega;\Gamma)}\lesssim \|u\|_{H^{-1}(\Omega;\Gamma)}.
	\end{align*}
	Combining this estimate with \eqref{eq:boundedness_Lh_1} completes the proof.
\end{proof}

\subsection{The proof of Lemma~\ref{lem:high_order_estimates}}\label{Appendix:estimate_kp1}
\begin{proof}[Proof of \eqref{eq:L2_eastimate_Lh*_kp1}]
	Starting from \eqref{eq:L2_estimate_Lh*_2}, it remains to estimate the three terms on the right-hand side. For the first two terms, the \(L^2\)-projection error estimate in Lemma~\ref{lem:L2_projection_error} yields
	\begin{align}
		\left|m\!\left(\L_h(1-Q_h)\L_h^{-1}u,v\right)\right|
		+
		\left|m_h\!\left((1-Q_h)\L_h^{-1}u,\L_h^{H*}v\right)\right|
		\lesssim
		h^{k+1}\|u\|_{H^{k+1}(\Omega;\Gamma)}\|v\|_{L^2(\Omega;\Gamma)}.
		\label{eq:L2_estimate_part12}
	\end{align}
	Using the geometric consistency estimate \eqref{eq:geom_consistency_m}, the inclusion \(B_h^{\L}\subset\Omega\), the boundary-layer estimate from Lemma~\ref{lem:boundary_layer_estimate}, and the boundedness of \(\L_h\) and \(\L_h^{H*}\) from \eqref{eq:boundedness_lift_operator} and \eqref{eq:L2_bound_Lh*}, we obtain
	\begin{align}
		&\left|m_h\left(\L_h^{-1}u,\L_h^{H*}v\right)-m\left(u,\L_h \L_h^{H*}v\right)\right|
		\notag\\
		&\quad\lesssim
		h^{k}\|u\|_{L^2(B_h^{\L})}\|\L_h \L_h^{H*}v\|_{L^2(B_h^{\L})}
		+
		h^{k+1}\|u\|_{L^2(\Omega;\Gamma)}\|\L_h \L_h^{H*}v\|_{L^2(\Omega;\Gamma)}
		\notag\\
		&\quad\lesssim
		h^{k+\frac12}\|u\|_{H^1(\Omega;\Gamma)}\|\L_h \L_h^{H*}v\|_{L^2(\Omega;\Gamma)}\lesssim
		h^{k+\frac12}\|u\|_{H^1(\Omega;\Gamma)}\|v\|_{L^2(\Omega;\Gamma)}.
		\label{eq:equiv_L2_error_1}
	\end{align}
	Combining \eqref{eq:L2_estimate_Lh*_2}, \eqref{eq:L2_estimate_part12}, and \eqref{eq:equiv_L2_error_1}, we obtain \eqref{eq:L2_eastimate_Lh*_kp1}.
\end{proof}

\begin{proof}[Proof of \eqref{eq:Hm1_eastimate_Lh*_kp1}]
	By the definition of the dual norm,
	\begin{align}\label{eq:decompose_Hm1_Lh*}
		\|(1-\L_h\L_h^{H*})u\|_{H^{-1}(\Omega;\Gamma)}
		&=\sup_{\|v\|_{H^1(\Omega;\Gamma)}=1}
		m((1-\L_h\L_h^{H*})u,v)\notag\\
		&\leq \sup_{\|v\|_{H^1(\Omega;\Gamma)}=1}m((1-\L_h\L_h^{H*})u,(1-\L_h\L_h^{H*})v)\notag\\
		&\quad +\sup_{\|v\|_{H^1(\Omega;\Gamma)}=1}m((1-\L_h\L_h^{H*})u,\L_h\L_h^{H*}v).
	\end{align}
	For the first term, by the Cauchy--Schwarz inequality, the estimates \eqref{eq:L2_eastimate_Lh*} and \eqref{eq:L2_eastimate_Lh*_kp1} yield
	\begin{align}\label{eq:1st_term_Hm1_Lh*}
		&\big|m((1-\L_h\L_h^{H*})u,(1-\L_h\L_h^{H*})v)\big|\notag\\
		&\quad\le \|(1-\L_h\L_h^{H*})u\|_{L^2(\Omega;\Gamma)}
		\|(1-\L_h\L_h^{H*})v\|_{L^2(\Omega;\Gamma)}\notag\\
		&\quad\lesssim \Bigl(
		h^{k+1}\|u\|_{H^{k+1}(\Omega;\Gamma)}
		+
		h^{k+\frac12}\|u\|_{H^1(\Omega;\Gamma)}
		\Bigr)\cdot h\|v\|_{H^1(\Omega;\Gamma)}\notag\\
		&\quad\lesssim \Bigl(h^{k+2}\|u\|_{H^{k+1}(\Omega;\Gamma)}
		+
		h^{k+\frac32}\|u\|_{H^1(\Omega;\Gamma)}\Bigr)\,\|v\|_{H^1(\Omega;\Gamma)}.
	\end{align}
	Next we estimate the second term on the right-hand side of \eqref{eq:decompose_Hm1_Lh*}. By the definition of $\L_h^{H*}$, the geometric consistency estimate \eqref{eq:geom_consistency_m}, boundary layer estimate in Lemma~\ref{lem:boundary_layer_estimate} and the $H^1$-boundedness of $\L_h^{H*}$ from \eqref{eq:bounded_H1},
	\begin{align}\label{eq:2nd_term_Hm1_Lh*}
		&\big|m((1-\L_h\L_h^{H*})u,\L_h\L_h^{H*}v)\big|\notag\\
		&\quad=
		\bigl|
		m_h(\L_h^{H*}u,\L_h^{H*}v)
		-
		m(\L_h\L_h^{H*}u,\L_h\L_h^{H*}v)
		\bigr|\notag\\
		&\quad\lesssim
		h^k
		\|\L_h\L_h^{H*}u\|_{L^2(B_h^\L)}
		\|\L_h\L_h^{H*}v\|_{L^2(B_h^\L)}
		+
		h^{k+1}
		\|\L_h\L_h^{H*}u\|_{L^2(\Omega;\Gamma)}
		\|\L_h\L_h^{H*}v\|_{L^2(\Omega;\Gamma)}\notag\\
		&\quad\lesssim h^{k+1}
		\|\L_h\L_h^{H*}u\|_{H^1(\Omega;\Gamma)}
		\|\L_h\L_h^{H*}v\|_{H^1(\Omega;\Gamma)}\notag\\
		&\quad\lesssim h^{k+1}\|u\|_{H^1(\Omega;\Gamma)}\|v\|_{H^1(\Omega;\Gamma)}
		\lesssim h^{k+1}\|u\|_{H^1(\Omega;\Gamma)}.
	\end{align}
	Combining \eqref{eq:decompose_Hm1_Lh*}--\eqref{eq:2nd_term_Hm1_Lh*}, and noting that $h^{k+\frac32}\le h^{k+1}$ for $0<h\le1$, we conclude \eqref{eq:Hm1_eastimate_Lh*_kp1}.
\end{proof}

\begin{proof}[Proofs of \eqref{eq:Hm1_eastimate_Lv*_kp1}]
	
	Since \(\widetilde A\) is invertible, for any \(v\in H^1(\Omega;\Gamma)\) with \(\|v\|_{H^1(\Omega;\Gamma)}=1\), we set \(w=\widetilde{A}^{-1}v.\)
	Then, by the definition of the \(H^{-1}(\Omega;\Gamma)\)-norm, we have
	\begin{align}\label{eq:decompose_Hm1_Lv*}
		\|(1-\L_h\L_h^{V*})u\|_{H^{-1}(\Omega;\Gamma)}
		&=\sup_{\|v\|_{H^1(\Omega;\Gamma)}=1}
		m((1-\L_h\L_h^{V*})u,v)\notag\\
		&=\sup_{\|v\|_{H^1(\Omega;\Gamma)}=1}
		\tilde{a}((1-\L_h\L_h^{V*})u,w)\notag\\
		&\leq \sup_{\|v\|_{H^1(\Omega;\Gamma)}=1}
		\tilde{a}((1-\L_h\L_h^{V*})u,(1-\L_h\L_h^{V*})w)\notag\\
		&\quad+
		\sup_{\|v\|_{H^1(\Omega;\Gamma)}=1}
		\tilde{a}((1-\L_h\L_h^{V*})u,\L_h\L_h^{V*}w).
	\end{align}
	
	We first estimate the first term on the right-hand side of \eqref{eq:decompose_Hm1_Lv*}. By the Cauchy--Schwarz inequality, the estimate \eqref{eq:H1_estimate_Lv*}, and the assumption \(k\ge 2\), we obtain
	\begin{align*}
		\big|\tilde{a}((1-\L_h\L_h^{V*})u,(1-\L_h\L_h^{V*})w)\big|
		&\lesssim
		\|(1-\L_h\L_h^{V*})u\|_{H^1(\Omega;\Gamma)}
		\|(1-\L_h\L_h^{V*})w\|_{H^1(\Omega;\Gamma)}\\
		&\lesssim
		h^k\|u\|_{H^{k+1}(\Omega;\Gamma)}
		\cdot
		h^2\|w\|_{H^3(\Omega;\Gamma)}.
	\end{align*}
	Moreover, elliptic regularity implies \(\|w\|_{H^3(\Omega;\Gamma)}\lesssim \|v\|_{H^1(\Omega;\Gamma)}\lesssim 1\).
	Therefore,
	\begin{align}\label{eq:1st_term_Hm1_Lv*}
		\big|\tilde{a}((1-\L_h\L_h^{V*})u,(1-\L_h\L_h^{V*})w)\big|
		\lesssim
		h^{k+2}\|u\|_{H^{k+1}(\Omega;\Gamma)}\|v\|_{H^1(\Omega;\Gamma)}.
	\end{align}
	
	Next, we estimate the second term on the right-hand side of \eqref{eq:decompose_Hm1_Lv*}. By applying the geometric consistency estimate \eqref{eq:geom_consistency_a} we obtain
	\begin{align*}
		&\big|\tilde{a}((1-\L_h\L_h^{V*})u,\L_h\L_h^{V*}w)\big|\notag\\
		&\quad=
		\bigl|
		\tilde{a}_h(\L_h^{V*}u,\L_h^{V*}w)
		-
		\tilde{a}(\L_h\L_h^{V*}u,\L_h\L_h^{V*}w)
		\bigr|\notag\\
		&\quad\lesssim
		h^k
		\|\L_h\L_h^{V*}u\|_{H^1(B_h^\L)}
		\|\L_h\L_h^{V*}w\|_{H^1(B_h^\L)}
		+
		h^{k+1}
		\|\L_h\L_h^{V*}u\|_{H^1(\Omega;\Gamma)}
		\|\L_h\L_h^{V*}w\|_{H^1(\Omega;\Gamma)}.
	\end{align*}
	Using \eqref{eq:H1_estimate_Lv*}, the boundary-layer estimate, and the elliptic regularity estimate for \(w=\widetilde{A}^{-1}v\), we have
	\begin{align}\label{eq:appendix_u_Bh}
		\|\L_h\L_h^{V*}u\|_{H^1(B_h^\L)}
		\le
		\|\L_h\L_h^{V*}u\|_{H^1(\Omega;\Gamma)}
		\lesssim
		\|u\|_{H^1(\Omega;\Gamma)},
	\end{align}
	and
	\begin{align*}
		\|\L_h\L_h^{V*}w\|_{H^1(B_h^\L)}
		&\le
		\|w\|_{H^1(B_h^\L)}
		+
		\|(1-\L_h\L_h^{V*})w\|_{H^1(\Omega;\Gamma)}\\
		&\lesssim
		h^{\frac12}\|w\|_{H^2(\Omega;\Gamma)}
		+
		h\|w\|_{H^2(\Omega;\Gamma)}\lesssim
		h^{\frac12}\|v\|_{L^2(\Omega;\Gamma)}.
	\end{align*}
	Hence,
	\begin{align}\label{eq:2nd_term_Hm1_Lv*}
		\big|\tilde{a}((1-\L_h\L_h^{V*})u,\L_h\L_h^{V*}w)\big|
		\lesssim
		h^{k+\frac12}\|u\|_{H^1(\Omega;\Gamma)}\|v\|_{L^2(\Omega;\Gamma)}.
	\end{align}
	Combining \eqref{eq:decompose_Hm1_Lv*}, \eqref{eq:1st_term_Hm1_Lv*}, and \eqref{eq:2nd_term_Hm1_Lv*}, we obtain \eqref{eq:Hm1_eastimate_Lv*_kp1}.
\end{proof}

\section{Proof of the estimates on the nonlinear terms}\label{Appendix:C}
\begin{proof}[Proof of Lemma~\ref{lem:nonlinear_term_estimates}]
	Among the estimates \eqref{eq:fu_L2}--\eqref{eq:fuv}, we only prove \eqref{eq:fu-fv}, since the remaining ones follow by analogous arguments from the same growth assumptions and Sobolev embeddings.
	
	By the fundamental theorem of calculus,
	\begin{align*}
		f_{\Omega}(u)-f_{\Omega}(v)=(u-v)\int_{0}^{1}f_{\Omega}^{\prime}(\theta u +(1-\theta)v)\d \theta.
	\end{align*}
	By the growth condition \eqref{eq:growth_condition}, we obtain
	\begin{align}\label{eq:taylor_fu-fv}
		\Big|f_{\Omega}(u)-f_{\Omega}(v)\Big|\lesssim \Big(1+|u|^{\zeta_{\Omega}-1}+|v|^{\zeta_{\Omega}-1}\Big)\cdot|u-v|.
	\end{align}
	Analogously,
	\begin{align}
		\Big|f_{\Gamma}(u)-f_{\Gamma}(v)\Big|\lesssim \Big(1+|u|^{\zeta_{\Gamma}-1}+|v|^{\zeta_{\Gamma}-1}\Big)\cdot|u-v|.
	\end{align}
	Since $\Omega$ is a smooth domain, by the definition of $H^{-1}(\Omega;\Gamma)$ and the Sobolev embeddings 
	$L^{p_{1}}(\Omega)\hookrightarrow H^{-1}(\Omega) $
	and { $L^{p_{2}}(\Gamma)\hookrightarrow H^{-1}(\Gamma)$}
	for 
	\begin{align}\label{eq:p1p2}
		p_{1}=\left\{\begin{array}{l}
			1+,\quad\text{for}\quad d=2;\\[2mm]
			{\displaystyle \frac{6}{5}} ,\quad \text{for}\quad d= 3;
		\end{array}\right.\quad \text{and}\quad
		p_{2}=1+,\quad\text{for}\quad d=2,3.
	\end{align}
	Here and below, \(1+\) denotes fixed exponents
	strictly larger than \(1\), possibly different for \(p_1\) and
	\(p_2\), chosen sufficiently close to \(1\), depending only on the prescribed growth exponents \(\zeta_\Omega\) and \(\zeta_\Gamma\), respectively, so that the Hölder exponents used below are admissible for the corresponding Sobolev embeddings.
	
	We have
	\begin{align*}
		&\|f(u)-f(v)\|_{H^{-1}(\Omega;\Gamma)}=\sup_{\|w\|_{H^1(\Omega;\Gamma)}=1}m\big(f(u)-f(v),w\big)\\
		&\leq \sup_{\|w\|_{H^1(\Omega;\Gamma)}=1}\Big(\|f_{\Omega}(u)-f_{\Omega}(v)\|_{H^{-1}(\Omega)}\|w\|_{H^1(\Omega)}+\|f_{\Gamma}(u)-f_{\Gamma}(v)\|_{H^{-1}(\Gamma)}\|w\|_{H^1(\Gamma)}\Big)\\
		&\lesssim \|f_{\Omega}(u)-f_{\Omega}(v)\|_{L^{p_{1}}(\Omega)}+\|f_{\Gamma}(u)-f_{\Gamma}(v)\|_{L^{p_{2}}(\Gamma)}\\
		&\lesssim \Big(1+\left\||u|^{\zeta_{\Omega}-1}\right\|_{L^{q_1}(\Omega)}+\left\||v|^{\zeta_{\Omega}-1}\right\|_{L^{q_1}(\Omega)}\Big)\cdot\|u-v\|_{L^2(\Omega)}\\
		&\qquad+\Big(1+\left\||u|^{\zeta_{\Gamma}-1}\right\|_{L^{q_2}(\Gamma)}+\left\||v|^{\zeta_{\Gamma}-1}\right\|_{L^{q_2}(\Gamma)}\Big)\cdot\|u-v\|_{L^2(\Gamma)},
	\end{align*}
	where we have used the H\"{o}lder's inequality in the last inequality with
	\begin{align*}
		\frac{1}{q_1}=\frac{1}{p_1}-\frac{1}{2},\quad\text{and}\quad  \frac{1}{q_2}=\frac{1}{p_2}-\frac{1}{2}.
	\end{align*}
	That is,
	\begin{align}\label{eq:q1q2}
		q_1=\frac{2p_1}{2-p_1},\quad\text{and}\quad q_2=\frac{2p_2}{2-p_2},
	\end{align}
	Hence,
	\begin{align}\label{eq:estimate_fu-fv}
		\|f(u)-f(v)\|_{H^{-1}(\Omega;\Gamma)}\lesssim& \big(1+\big\|u\big\|^{\zeta_{\Omega}-1}_{L^{q_1(\zeta_{\Omega}-1)}(\Omega)}+\big\|v\big\|^{\zeta_{\Omega}-1}_{L^{q_1(\zeta_{\Omega}-1)}(\Omega)}\Big)\cdot\|u-v\|_{L^2(\Omega)}\notag\\
		&+\big(1+\big\|u\big\|^{\zeta_{\Gamma}-1}_{L^{q_2(\zeta_{\Gamma}-1)}(\Gamma)}+\big\|v\big\|^{\zeta_{\Gamma}-1}_{L^{q_2(\zeta_{\Gamma}-1)}(\Gamma)}\Big)\cdot\|u-v\|_{L^2(\Gamma)}.
	\end{align}
	Next, recall the Sobolev embeddings
	\begin{align*}
		H^1(\Omega)\hookrightarrow L^{\theta_1}(\Omega), \quad \text{and}\quad H^1(\Gamma)\hookrightarrow L^{\theta_2}(\Gamma),
	\end{align*}
	where
	\begin{align}\label{eq:l1l2}
		\theta_1=\left\{\begin{array}{l}
			<\infty,\quad \;\;\text{for}\quad d=2;\\[2mm]
			6,\quad\qquad\text{for}\quad d= 3;
		\end{array}\right.\quad \text{and}\quad \theta_2<\infty,\quad\text{for}\quad d=2,3.
	\end{align}
	By the restriction \eqref{eq:growth_condition_zeta}, we may choose $p_1$ and $p_2$ such that
	\begin{align*}
		q_1(\zeta_{\Omega}-1)\le \theta_1,
		\qquad
		q_2(\zeta_{\Gamma}-1)\le \theta_2
	\end{align*}
	for some finite $\theta_1,\theta_2$ as above. Indeed, for $d=3$, choosing $p_1=\frac65$ gives $q_1=3$, and hence
	\begin{align*}
		q_1(\zeta_{\Omega}-1)=3(\zeta_{\Omega}-1)\le 6
	\end{align*}
	by \eqref{eq:growth_condition_zeta}. For $d=2$, $\theta_1$ can be taken arbitrarily large. Likewise, since $\theta_2$ may be chosen arbitrarily large on $\Gamma$, we can select $p_2\in(1,2)$ so that $q_2(\zeta_{\Gamma}-1)<\infty$ is admissible.
	
	Combining \eqref{eq:estimate_fu-fv} with \eqref{eq:growth_condition_zeta}, \eqref{eq:p1p2}, \eqref{eq:q1q2}, and \eqref{eq:l1l2}, we deduce
	\begin{align*}
		\|f(u)-f(v)\|_{H^{-1}(\Omega;\Gamma)}\lesssim& \big(1+\big\|u\big\|^{\zeta_{\Omega}-1}_{L^{\theta_1}(\Omega)}+\big\|v\big\|^{\zeta_{\Omega}-1}_{L^{\theta_1}(\Omega)}\Big)\cdot\|u-v\|_{L^2(\Omega)}\\
		&+\big(1+\big\|u\big\|^{\zeta_{\Gamma}-1}_{L^{\theta_2}(\Gamma)}+\big\|v\big\|^{\zeta_{\Gamma}-1}_{L^{\theta_2}(\Gamma)}\Big)\cdot\|u-v\|_{L^2(\Gamma)}\\
		\lesssim& \Big(1+\big\|u\big\|^{\max\{\zeta_{\Gamma},\zeta_{\Omega}\}-1}_{H^{1}(\Omega;\Gamma)}+\big\|v\big\|^{\max\{\zeta_{\Gamma},\zeta_{\Omega}\}-1}_{H^{1}(\Omega;\Gamma)}\Big)\cdot \|u-v\|_{L^2(\Omega;\Gamma)}.
	\end{align*}
	This completes the proof of \eqref{eq:fu-fv}.
\end{proof}

\begin{proof}[Proof of Lemma~\ref{lem:high_order_nonlinear_term_estimates}]
We first prove \eqref{eq:fu_H_ell}. By the norm equivalence result established in \eqref{eq:Hminus1_Hminus1h_db},
\begin{align}\label{eq:LH*f-Ihf}
		\|\L_h^{H*}f(\Pi_r u)- I_h f(\Pi_{r}u)\|_{H^{-1}(\Omega_h;\Gamma_h)}\lesssim \sup_{\substack{w_h\in X_h^k\\ \|w_h\|_{H^1(\Omega_h;\Gamma_h)}=1}}
		m_h\big(\L_h^{H*}f(\Pi_r u)- I_h f(\Pi_{r}u), w_h\big).
\end{align}
By the definitions of $\L_h^{H*}$ and $I_h$ in \eqref{eq:def:Lh*H} and \eqref{eq:def_Ih}, for any $w_h\in X_h^k$,
\begin{align*}
	m_h(\L_h^{H*}f(\Pi_r u)- I_h f(\Pi_{r}u), w_h)=m(f(\Pi_r u),\L_h w_h)-m_h\Big(\big(I_{h,\Omega}f_{\Omega}(\Pi_{r}u),I_{h,\Gamma}f_{\Gamma}(\Pi_{r}u)\big), w_h\Big).
\end{align*}
Moreover, by Lemma~\ref{lem:geom_consistency} and estimate \eqref{eq:fu_L2},
\begin{align*}
	\Big|m(f(\Pi_r u),\L_h w_h)-m_h(\L_h^{-1}f(\Pi_r u), w_h)\Big|&\lesssim h^k \big\|f(\Pi_r u)\big\|_{L^2(\Omega;\Gamma)}\big\|\L_h w_h\big\|_{L^2(\Omega;\Gamma)}\\
	&\lesssim h^k\,C(\|u\|_{H^1(\Omega;\Gamma)})\,\|w_h\|_{L^2(\Omega_h;\Gamma_h)}.
\end{align*}
On the other hand,
\begin{align*}
	&\Big|m_h(\L_h^{-1}f(\Pi_r u)- I_h f(\Pi_{r}u), w_h)\Big|\\
	&\quad\leq \left|\int_{\Omega_h} \Big(\L_h^{-1}f_{\Omega}(\Pi_r u)-I_{h,\Omega}f_{\Omega}(\Pi_r u)\Big)\,w_h\d x\right|+\left|\int_{\Gamma_h} \Big(\L_h^{-1}f_{\Gamma}(\Pi_r u)-I_{h,\Gamma}f_{\Gamma}(\Pi_r u)\Big)\,w_h\d s\right|\\
	&\quad\leq \Big\|\L_h^{-1}f_{\Omega}(\Pi_r u)-I_{h,\Omega}f_{\Omega}(\Pi_r u)\Big\|_{H^{-1}(\Omega_h)}\,\|w_h\|_{H^1(\Omega_h)}\\
	&\qquad+\Big\|\L_h^{-1}f_{\Gamma}(\Pi_r u)-I_{h,\Gamma}f_{\Gamma}(\Pi_r u)\Big\|_{H^{-1}(\Gamma_h)}\,\|w_h\|_{H^1(\Gamma_h)}.
\end{align*}
Substituting these bounds into \eqref{eq:LH*f-Ihf}, we obtain
\begin{align}\label{eq:estimate_LH*f-Ihf}
	&\|\L_h^{H*}f(\Pi_r u)- I_h f(\Pi_{r}u)\|_{H^{-1}(\Omega_h;\Gamma_h)}\lesssim \Big\|\L_h^{-1}f_{\Omega}(\Pi_r u)-I_{h,\Omega}f_{\Omega}(\Pi_r u)\Big\|_{H^{-1}(\Omega_h)}\notag\\
	&\qquad+\Big\|\L_h^{-1}f_{\Gamma}(\Pi_r u)-I_{h,\Gamma}f_{\Gamma}(\Pi_r u)\Big\|_{H^{-1}(\Gamma_h)}+h^k\,C(\|u\|_{H^1(\Omega;\Gamma)}).
\end{align}
Therefore, it remains to estimate the nodal interpolation errors for the bulk and surface nonlinearities $f_{\Omega}$ and $f_{\Gamma}$.

For simplicity, we present only the estimate for the bulk nonlinearity $f_{\Omega}$; the estimate for $f_{\Gamma}$ is derived in the same way and is in fact simpler. Using the Sobolev embedding $L^p(\Omega)\hookrightarrow H^{-1}(\Omega)$ for
		\begin{align}\label{eq:p}
		p=\left\{\begin{array}{l}
			1+,\quad\;\text{for}\quad d=2;\\[2mm]
			\frac{6}{5},\qquad \text{for}\quad d= 3;
		\end{array}\right.
	\end{align}
	and the boundedness of $\L_h^{-1}$ in $L^{p}(\Omega)$, we have
	\begin{align}\label{eq:interpolation_H-1}
		\Big\|\L_h^{-1}f_{\Omega}(\Pi_r u)-I_{h,\Omega}f_{\Omega}(\Pi_r u)\Big\|_{H^{-1}(\Omega_h)}
		\lesssim \|\L_h^{-1}f_{\Omega}(\Pi_r u)- I_{h,\Omega} f_{\Omega}(\Pi_{r}u)\|_{L^p(\Omega_h)}.
	\end{align}
	We next estimate the right-hand side by means of the standard local interpolation estimate for degree-$k$ isoparametric nodal interpolation on a shape-regular mesh; see, for example, \cite[Theorem~4.4.20]{brenner2008mathematical}. For each element $K\in\mathcal T_h$,
	\begin{align*}
		\|\L_h^{-1}f_{\Omega}(\Pi_r u)-I_{h,\Omega} f_{\Omega}(\Pi_{r}u)\|_{L^p(K)}
		\lesssim  h_K^{\ell}\Big|\L_h^{-1}f_{\Omega}(\Pi_{r}u)\Big|_{W^{\ell,p}(K)},
	\end{align*}
	provided that 
	\begin{align}\label{eq:condition_ell}
		1\leq \ell\leq k+1\quad \text{and}\quad \ell\cdot p>d.
	\end{align}
	
	Let  $\widetilde{K}=G(K)$. By \cite[Lemma~4.6]{ER2013}, the lift mapping $G$ defined in \eqref{eq:def_G} is piecewise $C^{k+1}$, and its derivatives up to order $k+1$, are uniformly bounded independently of $h$. Hence, by the definition of the inverse lift $\L_h^{-1}$ and the higher-order chain rule, we obtain
	\begin{align*}
		\Big|\L_h^{-1}f_{\Omega}(\Pi_{r}u)\Big|_{W^{\ell,p}(K)}\lesssim \Big|f_{\Omega}(\Pi_{r}u)\Big|_{W^{\ell,p}(\widetilde{K})},
	\end{align*}
	Summing the local estimates over all elements and using the quasi-uniformity of the mesh, we infer that
	\begin{align}\label{eq:interpolation_f_Lp}
		\|\L_h^{-1}f_{\Omega}(\Pi_r u)-I_{h,\Omega} f_{\Omega}(\Pi_{r}u)\|_{L^p(\Omega_h)}&\lesssim h^{\ell}\left(\sum_{K\in \mathcal{T}_h}\Big|f_{\Omega}(\Pi_{r}u)\Big|_{W^{\ell,p}(\widetilde{K})}^p\right)^{\frac{1}{p}}\notag\\
		&\leq h^{\ell}\Big|f_{\Omega}(\Pi_r u)\Big|_{W^{\ell,p}(\Omega)}.
	\end{align}
	In the present case since $k\ge d-1$, we can choose $\ell=d$ to satisfy \eqref{eq:condition_ell} in the above estimate. Thus it remains to bound $\Big|f_{\Omega}(\Pi_r u)\Big|_{W^{d,p}(\Omega)}$.
	
	By a direct application of the higher-order chain rule and Assumption~\ref{assump_f_db}, we obtain
	\begin{align}\label{eq:calculate_df}
		|\nabla^{d}f_{\Omega}(\Pi_{r} u)|&\lesssim \sum_{a=1}^{d}|f_{\Omega}^{(a)}(\Pi_{r}u)|\sum_{\xi_1+\xi_2+\cdots+\xi_a=d}\prod_{j=1}^{a} \bigl|\nabla^{\xi_j}\Pi_r u\bigr|\notag\\
		&\lesssim
		\sum_{a=1}^{d}\max\{(1+|\Pi_r u|)^{\zeta_\Omega-a},1\}
		\sum_{\substack{\xi_1+\cdots+\xi_a=d\\ \xi_j\ge 1}}
		\prod_{j=1}^{a} \bigl|\nabla^{\xi_j}\Pi_r u\bigr|.
	\end{align}
	
	We now distinguish the cases $d=2$ and $d=3$.
	
	\medskip
	\noindent
	\textit{Case $d=2$.}
	By \eqref{eq:calculate_df} and Assumption~\ref{assump_f_db}, it suffices to consider the case $\zeta_{\Omega}\ge 2$, since the case $\zeta_{\Omega}<2$ is simpler and can be treated in the same way.
	\begin{align*}
		&\|\nabla^{d}f_{\Omega}(\Pi_r u)\|_{L^p(\Omega)}\lesssim \left\|(1+|\Pi_r u|)^{\zeta_{\Omega}-2}\left|\nabla\Pi_r u\right|^2\right\|_{L^p(\Omega)}+\left\|(1+|\Pi_r u|)^{\zeta_{\Omega}-1}\left|\nabla^{2}\Pi_r u\right|\right\|_{L^p(\Omega)}\notag\\
		&\quad\leq \left\|(1+|\Pi_r u|)^{\zeta_{\Omega}-2}\right\|_{L^q(\Omega)}\left\|\left|\nabla\Pi_r u\right|^2\right\|_{L^2(\Omega)}+\left\|(1+|\Pi_r u|)^{\zeta_{\Omega}-1}\right\|_{L^q(\Omega)}\left\|\nabla^2\Pi_r u\right\|_{L^2(\Omega)},
	\end{align*}
	where
	\begin{align*}
		\frac{1}{q}+\frac{1}{2}=\frac{1}{p},\quad \text{thus}\quad q<\infty.
	\end{align*}
	From the Sobolev embedding $H^{\frac{3}{2}}(\Omega)\hookrightarrow W^{1,4}(\Omega)$ and $H^{1}(\Omega)\hookrightarrow L^\theta(\Omega)$ for all $\theta<\infty$ when $d=2$, there holds
	\begin{align*}
		\|\nabla^{d}f_{\Omega}(\Pi_r u)\|_{L^p(\Omega)}&\lesssim (1+\|\Pi_ru\|^{\zeta_{\Omega}-2}_{L^{(\zeta_{\Omega}-2)q}(\Omega)})\cdot\|\nabla \Pi_{r} u\|_{L^4(\Omega)}^2+(1+\|\Pi_ru\|^{\zeta_{\Omega}-1}_{L^{(\zeta_{\Omega}-1)q}(\Omega)})\cdot\left\|\nabla^2\Pi_r u\right\|_{L^2(\Omega)}\\
		&\leq C(\|u\|_{H^1(\Omega)})\cdot \left(\|\Pi_r u\|_{H^{\frac{3}{2}}(\Omega)}^2+\|\Pi_r u\|_{H^2(\Omega)}\right).
	\end{align*}
	Using the Bernstein inequality \eqref{eq:Bernstein_low} for $\Pi_r$, i.e. $\|\Pi_r u\|_{H^2(\Omega)}\lesssim r\|u\|_{H^1(\Omega)}$ and $\|\Pi_r u\|_{H^{\frac{3}{2}}(\Omega)}\lesssim\|\Pi_r u\|^{\frac{1}{2}}_{H^{1}(\Omega)}\|\Pi_r u\|^{\frac{1}{2}}_{H^{2}(\Omega)} \lesssim r^{\frac{1}{2}}\|u\|_{H^1(\Omega)}$,
	 we conclude that
	\begin{align}\label{eq:estimat_Lp_2d}
		\|\nabla^{d}f_{\Omega}(\Pi_r u)\|_{L^p(\Omega)}\leq r\cdot C(\|u\|_{H^1(\Omega)}), \quad \text{for}\quad d=2.
	\end{align}
	
	\medskip
	\noindent
	\textit{Case $d=3$.}
	By \eqref{eq:calculate_df} and  Assumption~\ref{assump_f_db}, it suffices to consider the case $2\leq \zeta_{\Omega}<3$,
	\begin{align*}
		&\|\nabla^{d}f_{\Omega}(\Pi_r u)\|_{L^p(\Omega)}\\
		&\lesssim \left\|\left|\nabla\Pi_r u\right|^3\right\|_{L^p(\Omega)}+\left\|(1+|\Pi_ru|)^{\zeta_{\Omega}-2}\left|\nabla^{2}\Pi_r u\right|\cdot \left|\nabla\Pi_r u\right|\right\|_{L^p(\Omega)}+\left\|(1+|\Pi_r u|)^{\zeta_{\Omega}-1}\left|\nabla^{3}\Pi_r u\right|\right\|_{L^p(\Omega)}\\
		&\leq \left\|\nabla\Pi_r u\right\|_{L^{2}(\Omega)}\cdot\left\|\nabla\Pi_r u\right\|^2_{L^{q_1}(\Omega)}+\left(1+\left\||\Pi_ru|^{\zeta_{\Omega}-2}\right\|_{L^{q_2}(\Omega)}\right)\left\|\nabla^{2}\Pi_r u\right\|_{L^2(\Omega)}\left\|\nabla\Pi_r u\right\|_{L^6(\Omega)}\\
		&\quad +\left(1+\left\||\Pi_ru|^{\zeta_{\Omega}-1}\right\|_{L^{q_3}(\Omega)}\right)\left\|\nabla^{3}\Pi_r u\right\|_{L^2(\Omega)},
	\end{align*}
	where $q_1=q_2=6$ and $q_3=3$ such that 
	\begin{align*}
		\frac{1}{2}+\frac{2}{q_1}=\frac{1}{q_2}+\frac{1}{2}+\frac{1}{6}=\frac{1}{q_3}+\frac{1}{2}=\frac{1}{p}=\frac{5}{6}.
	\end{align*}
	Since $\zeta_{\Omega}<3$, which implies $(\zeta_{\Omega}-2)\cdot q_2<6$ and $(\zeta_{\Omega}-1)\cdot q_3<6$, from the Sobolev embedding $H^1(\Omega)\hookrightarrow L^6(\Omega)$, in three dimensions, we infer that
	\begin{align}\label{eq:estimat_Lp_3d}
		\|\nabla^{d}f_{\Omega}(\Pi_r u)\|_{L^p(\Omega)}\lesssim &\|\Pi_r u\|_{H^1(\Omega)}\cdot \|\Pi_{r} u\|_{H^2(\Omega)}^2+\left(1+\|u\|_{H^1(\Omega)}^{\zeta_{\Omega}-2}\right) \|\Pi_r u\|^2_{H^2(\Omega)}\notag\\
		&+\left(1+\|u\|_{H^1(\Omega)}^{\zeta_{\Omega}-1}\right) \|\Pi_r u\|_{H^3(\Omega)}\notag\\
		\leq& r^2\cdot C(\|u\|_{H^1(\Omega)}),
	\end{align}
	where we have used the Bernstein inequality for $\Pi_r$.
	
	Combining \eqref{eq:estimat_Lp_2d} and \eqref{eq:estimat_Lp_3d} with \eqref{eq:interpolation_H-1} and \eqref{eq:interpolation_f_Lp}, we deduce that
	\begin{align*}
		\|\L_h^{-1}f_{\Omega}(\Pi_r u)- I_{h,\Omega}  f_{\Omega}(\Pi_{r}u)\|_{H^{-1}(\Omega_h)}\leq C(\|u\|_{H^1(\Omega)}) r^{d-1}h^{d},
	\end{align*}
	when $k\geq d-1$.
	Combining this with the analogous estimate for the surface nonlinearity $f_\Gamma$, and with \eqref{eq:estimate_LH*f-Ihf}, yields \eqref{eq:fu_H_ell}.

	We next prove \eqref{eq:fu_H_ell_1}, that is, the case $k<d-1$. This can occur only when $k=1$ and $d=3$. In this case, Proposition~5.4 of \cite{ER2013} yields
	\begin{align*}
		&\| f_{\Omega}(\Pi_{r}u(t_n))-\L_h I_{h,\Omega}  f_{\Omega}(\Pi_{r}u(t_n))\|_{L^2(\Omega)}+\| f_{\Gamma}(\Pi_{r}u(t_n))-\L_h I_{h,\Gamma}  f_{\Gamma}(\Pi_{r}u(t_n))\|_{L^2(\Gamma)}\\
		&\quad\lesssim  h^{2} \|f(\Pi_{r}u(t_n))\|_{H^{2}(\Omega)\times H^2(\Gamma)}.
	\end{align*}
	By the same chain-rule and Sobolev-embedding argument as above,
	together with the Bernstein estimate \eqref{eq:Bernstein_low}, we obtain
	\[
	\|f(\Pi_r u)\|_{H^2(\Omega)\times H^2(\Gamma)}
	\lesssim
	r^2 C\bigl(\|u\|_{H^1(\Omega;\Gamma)}\bigr).
	\]

	Therefore, by the boundedness of $\L_h^{-1}$ on $L^2(\Omega)$ and $L^2(\Gamma)$, together with the Sobolev embeddings $L^2(\Omega_h)\hookrightarrow H^{-1}(\Omega_h)$ and $L^2(\Gamma_h)\hookrightarrow H^{-1}(\Gamma_h)$, we obtain
	\begin{align*}
		&\Big\|\L_h^{-1}f_{\Omega}(\Pi_r u)-I_{h,\Omega}f_{\Omega}(\Pi_r u)\Big\|_{H^{-1}(\Omega_h)}+\Big\|\L_h^{-1}f_{\Gamma}(\Pi_r u)-I_{h,\Gamma}f_{\Gamma}(\Pi_r u)\Big\|_{H^{-1}(\Gamma_h)}\\
		&\quad\leq h^2r^2 C(\|u\|_{H^1(\Omega;\Gamma)}).
	\end{align*}
	Together with \eqref{eq:estimate_LH*f-Ihf}, this completes the proof of \eqref{eq:fu_H_ell_1}.
\end{proof}

\begin{proof}[Proof of Lemma~\ref{lem:interpolation}]
	We first prove \eqref{eq:Ihfuh-Ihfvh_1}. By the equivalence of the discrete dual norm, H\"{o}lder's inequality, and the Sobolev embeddings 	$L^{p_{1}}(\Omega_h)\hookrightarrow H^{-1}(\Omega_h)$ and $L^{p_{2}}(\Gamma_h)\hookrightarrow H^{-1}(\Gamma_h)$, we obtain
	\begin{align}\label{eq:decompose_Ihfuh-Ihfvh}
		&\|f_h(u_h)-f_h(v_h)\|_{H^{-1}(\Omega_h;\Gamma_h)}
		\lesssim
		\sup_{\substack{w_h\in X^k_h\\
				\|w_h\|_{H^1(\Omega_h;\Gamma_h)}=1}}
		\big|m_h(f_h(u_h)-f_h(v_h),w_h)\big|
		\notag\\
		&\quad\lesssim
		\|I_{h,\Omega} f_{\Omega}(\L_h u_h)
		-I_{h,\Omega} f_{\Omega}(\L_h v_h)\|_{L^{p_1}(\Omega_h)}+
		\|I_{h,\Gamma} f_{\Gamma}(\L_h u_h)
		-I_{h,\Gamma} f_{\Gamma}(\L_h v_h)\|_{L^{p_2}(\Gamma_h)}.
	\end{align}
	Here, we have used the definition of $I_h$ in \eqref{eq:def_Ih}
	and the $m_h$-orthogonality of $Q_h$.
	We mainly estimate the first term on the right-hand side of 
	\eqref{eq:decompose_Ihfuh-Ihfvh}; the second term can be handled analogously.
	
	Note that the interpolation operator $I_{h,\Omega}$ is not continuous with respect to the $L^{p_1}(\Omega_h)$-norm. Following the idea in \cite{HL2020}, we instead employ the discrete 
	$L^{p}(\Omega_h)$-norm, defined for $p\in[1,\infty)$ by
	\begin{align}\label{eq:def_discrete_norm}
		\left|\!\left|\!\left| u_h \right|\!\right|\!\right|_{p}:=h^{\frac{d}{p}}\left(\sum_{i=1}^{N}\big|u_h(a_i)\big|^p\right)^{\frac{1}{p}},
	\end{align}
	where $a_i$ $(i=1,\dots,N)$ denote the nodal points and $N=O(h^{-d})$ is the total number of these nodes.  
	For a finite element function $u_h \in X^k_h$, owing to the locality of the basis functions and the finite-dimensionality of the space, one can readily show that on $X^k_h$ the discrete norm 
	$\left|\!\left|\!\left|\,\cdot\,\right|\!\right|\!\right|_{p}$ and the $L^p$-norm $\|\cdot\|_{L^p(\Omega_h)}$ are equivalent 
	(for $p=2$ this is well known; for a general $p$ one may refer to \cite[Lemma~5.2]{Leibold2017}).
	
	Since \(I_{h,\Omega} f_{\Omega}(\mathcal L_h u_h),\, I_{h,\Omega} f_{\Omega}(\mathcal L_h v_h) \in V^{\Omega}_h\) and
	\(I_{h,\Omega}\) is the nodal interpolation operator, we evaluate these functions at the nodal
	points \(a_i\), \(i=1,\dots,N\). Moreover, for the isoparametric finite element
	discretization, the lifting operator \(\mathcal L_h\) leaves the nodal points \(a_i\)
	invariant. Hence,
	\[
	I_{h,\Omega} f_{\Omega}(\mathcal L_h u_h)(a_i)=f_{\Omega}(u_h(a_i)),
	\quad\text{and}\quad
	I_{h,\Omega} f_{\Omega}(\mathcal L_h v_h)(a_i)=f_{\Omega}(v_h(a_i)).
	\]
	Therefore, by the equivalence between the \(L^{p_1}(\Omega_h)\)-norm and the
	discrete norm \(\left|\!\left|\!\left| \cdot \right|\!\right|\!\right|_{p_1}\), together with
	\eqref{eq:taylor_fu-fv}, we obtain
	\begin{align}\label{eq:Ihfu-Ihfv_Lp}
		\|I_{h,\Omega} f_{\Omega}(\L_h u_h)-I_{h,\Omega} f_{\Omega}(\L_h v_h)\|_{L^{p_1}(\Omega_h)}&\sim h^{\frac{d}{p_1}}\left(\sum_{i=1}^{N}\Big|f_{\Omega}(u_h(a_i))- f_{\Omega}(v_h(a_i))\Big|^{p_1}\right)^{\frac{1}{p_1}}\notag\\
		&\lesssim h^{\frac{d}{p_1}}\left(\sum_{i=1}^{N}A_i^{(\zeta_{\Omega}-1)p_1}B_i^{p_1}\right)^{\frac{1}{p_1}},
	\end{align}
	where
	\begin{align}\label{eq:notations_AB}
		A_i=1+|u_h(a_i)|+|v_h(a_i)|,\quad \text{and}\quad B_i=|u_h(a_i)-v_h(a_i)|.
	\end{align}
	Using H\"{o}lder's inequality and the equivalence between the discrete and continuous norms, we obtain
	\begin{align*}
		h^{\frac{d}{p_1}}\left(\sum_{i=1}^{N}A_i^{(\zeta_{\Omega}-1)p_1}B_i^{p_1}\right)^{\frac{1}{p_1}}&\leq h^{\frac{d}{q_1}}\left(\sum_{i=1}^{N}A_i^{(\zeta_{\Omega}-1)q_1}\right)^{\frac{1}{q_1}}\cdot h^{\frac{d}{2}}\left(\sum_{i=1}^{N}B_i^{2}\right)^{\frac{1}{2}}\\
		&\lesssim  \Big(1+\left|\!\left|\!\left| u_h \right|\!\right|\!\right|_{(\zeta_{\Omega}-1)q_1}^{\zeta_{\Omega}-1}+\left|\!\left|\!\left| v_h \right|\!\right|\!\right|_{(\zeta_{\Omega}-1)q_1}^{\zeta_{\Omega}-1}\Big)\cdot \left|\!\left|\!\left|u_h- v_h \right|\!\right|\!\right|_{2}\\
		&\lesssim \left(1+\left\|u_h\right\|_{L^{(\zeta_{\Omega}-1)q_1}(\Omega_h)}^{\zeta_{\Omega}-1}+\left\|v_h\right\|_{L^{(\zeta_{\Omega}-1)q_1}(\Omega_h)}^{\zeta_{\Omega}-1}\right)\cdot \|u_h- v_h\|_{L^2(\Omega_h)},
	\end{align*}
	where $q_1$ is given by \eqref{eq:q1q2}.  
	Since $H^1(\Omega_h)\hookrightarrow L^{\theta_1}(\Omega_h)$ with $\theta_1$ given in \eqref{eq:l1l2}, 
	and under the assumption on $\zeta_{\Omega}$ in \eqref{eq:growth_condition_zeta} such that $(\zeta_{\Omega}-1)q_1\leq \theta_1$, we have
	\begin{align*}
		h^{\frac{d}{p_1}}\left(\sum_{i=1}^{N}A_i^{(\zeta_{\Omega}-1)p_1}B_i^{p_1}\right)^{\frac{1}{p_1}}\lesssim \left(1+\left\|u_h\right\|_{H^1(\Omega_h)}^{\zeta_{\Omega}-1}+\left\|v_h\right\|_{H^1(\Omega_h)}^{\zeta_{\Omega}-1}\right)\cdot \|u_h- v_h\|_{L^2(\Omega_h)}.
	\end{align*}
	Substituting the above estimate into \eqref{eq:Ihfu-Ihfv_Lp} yields
	\begin{align*}
		\|I_{h,\Omega} f_{\Omega}(\L_h u_h)-I_{h,\Omega} f_{\Omega}(\L_h v_h)\|_{L^{p_1}(\Omega_h)}\lesssim \left(1+\left\|u_h\right\|_{H^1(\Omega_h)}^{\zeta_{\Omega}-1}+\left\|v_h\right\|_{H^1(\Omega_h)}^{\zeta_{\Omega}-1}\right)\cdot \|u_h- v_h\|_{L^2(\Omega_h)}.
	\end{align*}
	Analogously, for the second term on the right-hand side of 
	\eqref{eq:decompose_Ihfuh-Ihfvh}, we have
	\begin{align*}
		\|I_{h,\Gamma} f_{\Gamma}(\L_h u_h)-I_{h,\Gamma} f_{\Gamma}(\L_h v_h)\|_{L^{p_2}(\Gamma_h)}\lesssim \left(1+\left\|u_h\right\|_{H^1(\Gamma_h)}^{\zeta_{\Gamma}-1}+\left\|v_h\right\|_{H^1(\Gamma_h)}^{\zeta_{\Gamma}-1}\right)\cdot \|u_h- v_h\|_{L^2(\Gamma_h)}.
	\end{align*}
	This completes the proof of \eqref{eq:Ihfuh-Ihfvh_1}.
	
	We next proceed to the proof of the estimate \eqref{eq:Ihfuh-Ihfvh_2}.  
	From the definition of the space $L^2(\Omega_h; \Gamma_h)$ and the definition of $I_h$ in \eqref{eq:def_Ih}, we have
	\begin{align}\label{eq:decompose_Ihfuh-Ihfvh_L2}
		&\|I_{h} f_h(\L_h u_h)-I_{h} f_h(\L_h v_h)\|_{L^2(\Omega_h;\Gamma_h)}\notag\\
		&\quad\leq \|I_{h,\Omega} f_{\Omega}(\L_h u_h)-I_{h,\Omega} f_{\Omega}(\L_h v_h)\|_{L^{2}(\Omega_h)}+\|I_{h,\Gamma} f_{\Gamma}(\L_h u_h)-I_{h,\Gamma} f_{\Gamma}(\L_h v_h)\|_{L^{2}(\Gamma_h)}.
	\end{align}
	As before, we mainly estimate the first term on the right-hand side of 
	\eqref{eq:decompose_Ihfuh-Ihfvh_L2}; the second term can be treated analogously.
	
	Using the discrete norm defined in \eqref{eq:def_discrete_norm} and the notation introduced in \eqref{eq:notations_AB}, we have
	\begin{align}\label{eq:Ihfu-Ihfv_L2}
		\|I_{h,\Omega} f_{\Omega}(\L_h u_h)-I_{h,\Omega} f_{\Omega}(\L_h v_h)\|_{L^{2}(\Omega_h)}&\sim h^{\frac{d}{2}}\left(\sum_{i=1}^{N}\Big| f_{\Omega}(u_h(a_i))- f_{\Omega}(v_h(a_i))\Big|^{2}\right)^{\frac{1}{2}}\notag\\
		&\lesssim h^{\frac{d}{2}}\left(\sum_{i=1}^{N}A_i^{2(\zeta_{\Omega}-1)}B_i^{2}\right)^{\frac{1}{2}}\notag\\
		&=h^{\frac{d}{p}}h^{\frac{\alpha d}{2}}\left(\sum_{i=1}^{N}A_i^{2(\zeta_{\Omega}-1)}B_i^{2(1-\alpha)}\cdot B_i^{2\alpha}\right)^{\frac{1}{2}},
	\end{align}
	where the exponent $p$ is chosen such that
	\begin{align}\label{eq:def_p}
		\frac{1}{p}+\frac{\alpha}{2}=\frac{1}{2},\quad \text{i.e.}\quad p=\frac{4}{2-2\alpha}.
	\end{align}
	Applying H\"{o}lder's inequality yields
	\begin{align}\label{eq:holder_L2}
		h^{\frac{d}{p}}h^{\frac{\alpha d}{2}}\left(\sum_{i=1}^{N}A_i^{2(\zeta_{\Omega}-1)}B_i^{2(1-\alpha)}\cdot B_i^{2\alpha}\right)^{\frac{1}{2}}\leq h^{\frac{d}{p}}h^{\frac{\alpha d}{2}} \left(\sum_{i=1}^{N}\Big(A_i^{\zeta_{\Omega}-1}B_i^{1-\alpha}\Big)^{p}\right)^{\frac{1}{p}}\cdot\left(\sum_{i=1}^{N} \big(B_i^{\alpha}\big)^{\frac{2}{\alpha}}\right)^{\frac{\alpha}{2}}
	\end{align}
	Note that, since $B_i=|u_h(a_i)-v_h(a_i)|
	\le |u_h(a_i)|+|v_h(a_i)|
	\le A_i$, we have
	\begin{align*}
		A_i^{\zeta_{\Omega}-1}B_i^{1-\alpha}=\Big(1+|u_h(a_i)|+|v_h(a_i)|\Big)^{\zeta_{\Omega}-1}\cdot |u_h(a_i)-v_h(a_i)|^{1-\alpha}\lesssim A_i^{\zeta_{\Omega}-\alpha}.
	\end{align*}
	Combining \eqref{eq:Ihfu-Ihfv_L2}, \eqref{eq:holder_L2}, and the equivalence between the discrete and continuous norms gives
	\begin{align}\label{eq:Ihfu-Ihfv_L2_new}
		&\|I_{h,\Omega} f_{\Omega}(\L_h u_h)-I_{h,\Omega} f_{\Omega}(\L_h v_h)\|_{L^{2}(\Omega_h)}\lesssim h^{\frac{d}{p}}h^{\frac{\alpha d}{2}} \left(\sum_{i=1}^{N}\big(A_i^{\zeta_{\Omega}-\alpha}\big)^{p}\right)^{\frac{1}{p}}\cdot\left(\sum_{i=1}^{N} \big(B_i^{\alpha}\big)^{\frac{2}{\alpha}}\right)^{\frac{\alpha}{2}}\notag\\
		&\lesssim  \Big(1+\left|\!\left|\!\left| u_h \right|\!\right|\!\right|_{(\zeta_{\Omega}-\alpha)p}^{\zeta_{\Omega}-\alpha}+\left|\!\left|\!\left| v_h \right|\!\right|\!\right|_{(\zeta_{\Omega}-\alpha)p}^{\zeta_{\Omega}-\alpha}\Big)\cdot \left|\!\left|\!\left|u_h- v_h \right|\!\right|\!\right|^{\alpha}_{2}\notag\\
		&\lesssim \left(1+\left\|u_h\right\|_{L^{(\zeta_{\Omega}-\alpha)p}}^{\zeta_{\Omega}-\alpha}+\left\|v_h\right\|_{L^{(\zeta_{\Omega}-\alpha)p}}^{\zeta_{\Omega}-\alpha}\right)\cdot \|u_h- v_h\|^{\alpha}_{L^2(\Omega_h)}.
	\end{align}
	For the Sobolev embeddings, we know that 
	{ \begin{align*}
		H^1(\Omega_h)\hookrightarrow L^{\theta}(\Omega_h), \quad \text{for}\quad \theta\left\{\begin{array}{l}
			<\infty,\quad \;\;\text{for}\quad d=2;\\[2mm]
			=6,\quad\text{for}\quad d= 3.
		\end{array}\right.
	\end{align*}}

	If $d=2$, it is straightforward to verify that 
	\begin{align*}
		(\zeta_{\Omega}-\alpha)p<\infty, \quad \text{for}\quad \alpha=\frac{1}{2},\quad \text{and} \quad p=\frac{4}{2-2\alpha}.
	\end{align*}

	For $d=3$, one can check directly that
	\begin{align*}
		(\zeta_{\Omega}-\alpha)p=\frac{2d}{d-2},\quad \text{for}\quad \alpha=\frac{d-2}{2}\left(\frac{d}{d-2}-\zeta_{\Omega}\right),\quad \text{and} \quad p=\frac{4}{2-2\alpha}.
	\end{align*}
	
	Hence, under the choice of $\alpha$ in \eqref{eq:choose_alpha}, we obtain
	\begin{align*}
		\left\|u_h\right\|_{L^{(\zeta_{\Omega}-\alpha)p}}+\left\|v_h\right\|_{L^{(\zeta_{\Omega}-\alpha)p}}\lesssim \|u_h\|_{H^1(\Omega_h)}+\|v_h\|_{H^1(\Omega_h)}.
	\end{align*}
	Substituting this estimate into \eqref{eq:Ihfu-Ihfv_L2_new}, we conclude that
	\begin{align*}
		\|I_{h,\Omega} f_{\Omega}(\L_h u_h)-I_{h,\Omega} f_{\Omega}(\L_h v_h)\|_{L^{2}(\Omega_h)}\lesssim \left(1+\|u_h\|_{H^1(\Omega_h)}^{\zeta_{\Omega}-\alpha}+\|v_h\|_{H^1(\Omega_h)}^{\zeta_{\Omega}-\alpha}\right)\cdot \|u_h- v_h\|^{\alpha}_{L^2(\Omega_h)}.
	\end{align*}
	
	Analogously, we have
	\begin{align*}
		\|I_{h\Gamma} f_{\Gamma}(\L_h u_h)-I_{h,\Gamma} f_{\Gamma}(\L_h v_h)\|_{L^{2}(\Gamma_h)}\lesssim \left(1+\|u_h\|_{H^1(\Gamma_h)}^{\zeta_{\Gamma}-\alpha}+\|v_h\|_{H^1(\Gamma_h)}^{\zeta_{\Gamma}-\alpha}\right)\cdot \|u_h- v_h\|^{\alpha}_{L^2(\Gamma_h)}.
	\end{align*}
	This completes the proof of \eqref{eq:Ihfuh-Ihfvh_2}.
\end{proof}

\bibliographystyle{abbrv} 
\bibliography{literature}

\end{document}